\documentclass[12pt]{amsart}
\usepackage[hmargin=2.5cm,vmargin=2.5cm]{geometry}
\usepackage{amsfonts, amstext, amsmath, amsthm, amscd, amssymb, amsbsy}
\usepackage[pdftex]{graphicx, color}

\usepackage[noadjust]{cite}
\usepackage[hidelinks,pagebackref]{hyperref}
\usepackage{enumitem}
\usepackage{setspace}
\usepackage{float}
\usepackage{subcaption}
\usepackage{mathrsfs}
\usepackage{tikz}
\usepackage{pinlabel}
\usepackage{multirow}

\AtBeginDocument{
	\def\MR#1{}
}

\pagecolor{white}

\allowdisplaybreaks

\newcommand{\Z}{\mathbb{Z}}

\newcommand{\C}{\mathbb{C}}

\newtheorem{thm}{Theorem}
\numberwithin{thm}{section}

\newtheorem{prop}[thm]{Proposition}
\newtheorem{lemma}[thm]{Lemma}
\newtheorem{cor}[thm]{Corollary}
\newtheorem{question}[thm]{Question}

\newtheorem*{namedtheorem}{\theoremname}
\newcommand{\theoremname}{testing}

\theoremstyle{definition}
\newtheorem{defn}[thm]{Definition}
\newtheorem*{nameddef}{\defname}
\newcommand{\defname}{testing}

\theoremstyle{definition}
\newtheorem{rmk}[thm]{Remark}
\newtheorem{conv}[thm]{Convention}

\begin{document}
	\title{Knot Floer Homology, the Burau Representation, and Quantum $\mathfrak{gl}(1 \vert 1)$}
	\author{Joe Boninger}
	\address{Department of Mathematics, Boston College, Chestnut Hill, MA}
	\email{boninger@bc.edu}
	\maketitle
	
	\begin{abstract}
		The Burau representation of braid groups and knot Floer homology share a link to the Fox calculus. We make this connection explicit, with the following outcome: if $B$ is the full Burau matrix of any braid, and $A$ is any square submatrix of $B - \lambda I$, we define a Heegaard Floer homology theory that categorifies $\det(A)$ and is an invariant of the braid. We also describe an analogous construction for the Gassner representation. Then, we leverage the relationship between the Burau representation and quantum $\mathfrak{gl}(1 \vert 1)$ to exhibit connections between the latter and Heegaard Floer homology. We associate a bordered sutured Heegaard Floer homology group to any tangle, and give a simple, geometric proof that our invariant recovers the $U_q(\mathfrak{gl}(1 \vert 1))$ braid representation.
	\end{abstract}
	
	\section{Introduction}
	
	The full Burau representation $\psi_n$ of the $n$-strand braid group $B_n$,
	$$
	\psi_n : B_n \to GL_n(\Z[t^{\pm 1}]),
	$$
	has been studied intensely for almost one hundred years. It is known to be faithful when $n = 2$ or $n = 3$, and is unfaithful when $n \geq 5$ by work of Long, Paton, and Bigelow \cite{lopa93, big99}. The faithfulness of $\psi_4$ remains an important open problem, related to the faithfulness of the Jones representation and whether or not the Jones polynomial detects the unknot \cite{big02}.
	
	\begin{question}
		\label{q:burau}
		Is the Burau representation of the four-strand braid group faithful?
	\end{question}
	
	When restricted to the $n$-strand pure braid group, the Burau representation also admits a multi-variable generalization called the Gassner representation. Although the Gassner representation is similarly well-studied, it is unknown whether this representation is faithful for any $n$ greather than three.
	
	One source of interest in the Burau representation is its relationship to the Alexander module in knot theory. Let $\rho \in B_n$ be an $n$-strand braid, and $K \subset S^3$ its link closure. If $I_n$ is the $n$-dimensional identity matrix, and $A$ is any matrix resulting from removing one row and column from $\psi_n(\rho) - I_n$, then $A$ is an Alexander matrix for $K$---in particular, $\det(A)$ is the Alexander polynomial $\Delta_K(t)$ \cite[Proposition 3.10]{bir74}. From one point of view, this connection arises from the fact that Alexander matrices and the Burau representation can both be defined using the Fox calculus \cite[Chapter 3]{bir74}.
	
	Knot and link Floer homology, which categorify the single- and multi-variable Alexander polynomials respectively, are also closely related to the Fox calculus. This connection is explained in Rasmussen's original paper \cite[Section 3.4]{ras03}, in which he refers to knot Floer homology as ``essentially just a geometric realization of Fox calculus'' \cite[Section 2.3]{ras03}. It is therefore natural to expect a relationship between the Burau representation and knot Floer homology.
	
	In this paper, we develop this link explicitly to the fullest extent possible. Our first result shows that the determinate of any square submatrix of $\psi_n - \lambda I_n$ can be categorified by an appropriate Heegaard Floer homology theory.
	
	\begin{thm}
		\label{thm:main_burau}
		Let $\rho \in B_n$ be a braid, and let ${\bf j, k} \subset \{1, \dots, n\}$ be multi-indices of size $m$. Let $\lambda$ be indeterminate, and let $A$ be the square submatrix of $\psi_n(\rho) - \lambda I_n$ which is the intersection of the rows ${\bf j}$ and the columns ${\bf k}$. Then there exists a triply-graded Heegaard Floer homology theory $\widehat{HFB}(\rho, {\bf j}, {\bf k})$ which is an invariant of $\rho$, and whose Poincar\'e polynomial is $\det(A) \in \Z[t^{\pm 1}, \lambda]$.
	\end{thm}
	
	Thus, to any braid $\rho \in B_n$, one can associate $\binom{2n}{n} - 1$ invariant Heegaard Floer homology groups depending on multi-indices ${\bf j}$ and ${\bf k}$. Some choices of ${\bf j}$ and ${\bf k}$ merit extra consideration.
	\begin{enumerate}[label=(\roman*)]
		\item Let $K \subset S^3$ be the link which is the braid closure of $\rho$, and let $\bar{\rho}$ be the braid resulting from reflecting a braid diagram of $\rho$ horizontally. Let $m(K) \subset S^3$ be the mirror of $K$, thought of as the braid closure of $\bar{\rho}$. When $A$ is the entire matrix $\psi_n(\rho) - \lambda I_n$, the resulting homology groups
		$$
		\widehat{HFB}(\rho, \{1, \dots, n\}, \{1, \dots, n\})
		$$
		are essentially the link Floer groups of $m(K) \cup U \subset S^3$, where $U$ is the axis around which $m(K)$ is braided.
		\item When $A$ is a codimension one submatrix of $\psi_n(\rho) - \lambda I_n$, or equivalently when
		$$
		|{\bf j}| = |{\bf k}| = n - 1,
		$$
		the corresponding groups $\widehat{HFB}(\rho, {\bf j}, {\bf k})$ admit a spectral sequence to the knot Floer homology of $m(K)$.
		\item When $A$ is a single element $a_{jk}$ of $\psi_n(\rho) - \lambda I_n$, i.e.~when $|{\bf j}| = |{\bf k}| = 1$, the groups $\widehat{HFB}(\rho, \{j\}, \{k\})$ reflect a simple count of intersections between two closed curves in a multiply-punctured torus.
	\end{enumerate}
	Additionally, it is not difficult to show that the groups $\widehat{HFB}$ detect the trivial braid, for example by applying a result of Baldwin and Grigsby \cite{bagr15}. We also explain how analogous groups can be defined for the Gassner representation---see Section \ref{sec:propss} for details, and Section \ref{sec:rel_work} for a discussion of related work.
	
	\subsection{Knot Floer homology and quantum $\mathfrak{gl}(1\vert 1)$}
	
	We have not yet been able to apply the groups $\widehat{HFB}$ to Question \ref{q:burau}. Instead, in this paper, we put our framework toward a different goal: building bridges between Heegaard Floer homology and the quantum $\mathfrak{gl}(1 \vert 1)$ Reshitikhin-Turaev invariant.
	
	Knot homology theories are most often viewed from one of three perspectives: those of Lagrangian Floer theory, gauge theory, and the representation theory of quantum groups. Unifying these points of view is a central project in low-dimensional topology. Knot Floer theory---and Heegaard Floer theory more generally---fall into the first category of homology theories, and are also understood gauge theoretically through equivalences with Seiberg-Witten theories (e.g.~\cite{kuleta20}). Representation theoretically, knot Floer homology is related to the quantum group $U_q(\mathfrak{gl}(1 \vert 1))$, known as quantum $\mathfrak{gl}(1 \vert 1)$, but this connection is more mysterious.
	
	A few key results relate knot Floer theory to quantum $\mathfrak{gl}(1 \vert 1)$. Ellis, Petkova and V\'ertesi identified the Grothendieck group of the tangle Floer algebra---a generalization of knot Floer homology for tangles---with a tensor product of certain $U_q(\mathfrak{gl}(1\vert 1))$ representations \cite{epv19}. Similar work was done by Manion \cite{man19} for a different construction of tangle Floer homology due to Ozsv\'ath and Szab\'o \cite{os18}, and this was developed further in \cite{lama21}. In another direction, Robert and Wagner define a quantum categorification of the Alexander polynomial using $U_q(\mathfrak{gl}(1 \vert 1))$ \cite{rowa22}. It is not known whether Robert and Wagner's theory is equivalent to knot Floer homology.
	
	The $U_q(\mathfrak{gl}(1\vert 1))$ braid (vector) representation is determined by the Burau representation: up to a normalization, the former results from taking exterior powers of the latter \cite{kasa91}. Thus, the theory behind our Theorem \ref{thm:main_burau} provides a natural way of relating knot Floer homology and quantum $\mathfrak{gl}(1\vert 1)$. We first obtain:
	
	\begin{thm}
		\label{thm:main_uq}
		Let $K \subset S^3$ be a link, $\rho \in B_n$ a braid with closure $K$, and $\widehat{CFK}(m(K))$ a certain chain complex for the knot Floer homology of the mirror $m(K)$. Separately, let $\widehat{Q}(K)$ denote the $U_q(\mathfrak{gl}(1\vert 1))$ invariant associated to $K$. Then there is a weight-preserving bijection between generators of $\widehat{CFK}(m(K))$ and terms in a state sum expansion of $\widehat{Q}(K)$. Both the complex $\widehat{CFK}(m(K))$ and the expansion of $\widehat{Q}(K)$ are determined by the braid $\rho$.
	\end{thm}
	
	See Theorem \ref{thm:clear} below for a more detailed statement, and we note that a similar mirroring happens in Manion's work \cite[Theorem 1.4.2]{man19}. Theorem \ref{thm:main_uq} implies a well-known fact: that $\widehat{Q}(K)$ and the Poincar\'e polynomial of $\widehat{HFK}(m(K))$ are both equal to the Alexander polynomial of $K$. Our statement is more precise, however, and the argument is simple and geometric, providing an intuitive justification for the relationship between knot Floer homology and quantum $\mathfrak{gl}(1 \vert 1)$. We believe this justification to be the first of its kind in print.
	
	Our final result extends Theorem \ref{thm:main_uq} further, in the direction of Ellis-Petkova-V\'ertesi's and Manion's work. In Section \ref{sec:bst}, for any tangle $T$ in the unit cube, we define a bordered sutured three-manifold $Y(T)$. The manifold $Y(T)$ has an associated Heegaard Floer theoretic object $\widehat{BFA}(Y(T))$, defined using Zarev's bordered sutured Heegaard Floer theory \cite{za11}, and the map
	$$
	T \longmapsto \widehat{BFA}(Y(T))
	$$
	gives a (semi) new approach to Heegaard Floer homology for tangles. Our definition of $Y(T)$ is motivated by the proof of Theorem \ref{thm:main_uq}, and the resulting tangle Floer groups $ \widehat{BFA}(Y(T))$ are something of a hybrid of the two tangle Floer theories defined by Zibrowius \cite{zib19} and Ozsv\'ath-Szab\'o \cite{os18}. See Figure \ref{fig:bsta} for an example of a bordered sutured manifold involved.
	
	For any tangle $T$, the object $\widehat{BFA}(Y(T))$ has the structure of a $\Z/2$-vector space, which splits as a direct sum over pairs of subsets of $\{1, \dots, n\}$:
	$$
	\widehat{BFA}(Y(T)) = \bigoplus_{{\bf j, k} \subset \{1, \dots, n\}} \widehat{BFA}(Y(T)) \cdot \iota_{\bf j, k}.
	$$
	(For experts, the terms $\iota_{\bf j, k}$ are the indecomposable idempotents of the relevant arc algebra.) To each summand $\widehat{BFA}(Y(T)) \cdot \iota_{\bf j, k}$, we associate a decategorified invariant $\Delta_{Y, {\bf j,k}}(t) \in \Z[t^{\pm 1}]$. We then prove that $\widehat{BFA}(Y(T))$ recovers the $U_q(\mathfrak{gl}(1 \vert 1))$ braid representation in the following sense:
	
	\begin{thm}
		\label{thm:main_uq2}
		Let $\rho \in B_n$ be a geometric braid, thought of as a tangle $T \subset I^3$, and let $Y(T)$ be the corresponding bordered sutured manifold. Let
		$$
		Q : B_n \to Aut((\C(t)^2)^{\otimes n})
		$$
		be the $U_q(\mathfrak{gl}(1 \vert 1))$ braid representation. Then, for any choice of multi-indices
		$$
		{\bf j,k} \subset \{1, \dots, n\},
		$$
		the polynomial $\Delta_{Y, {\bf j,k}}(t)$ of $\widehat{BFA}(Y(T)) \cdot \iota_{\bf j, k}$ is equal to a corresponding coefficient of $Q(\rho)$, up to a global normalization.
	\end{thm}
	See Theorem \ref{thm:last_one} for the precise result, which involves some mirroring. 
	
	Theorem \ref{thm:main_uq2} can be compared directly with the main theorems of \cite{epv19, man19} discussed above, since each result compares a different tangle Floer theory to the $U_q(\mathfrak{gl}(1 \vert 1))$ Reshetikhin-Turaev invariant. Our result is weaker than the others for two reasons: first, we limit ourselves to the special case where the tangle is a braid; second, we avoid the formalism of Grothendieck groups. We chose to limit Theorem \ref{thm:main_uq2} in these ways to keep the scope of the paper manageable. We expect our work can be extended to a full analogue of the results in \cite{epv19, man19}, and we plan to accomplish this in a follow-up paper.  (Such an extension will partly consist of studying $\widehat{BFA}(Y(T))$ when the tangle $T$ is a ``cup'' or a ``cap;'' cf.~\cite[Section 4.3]{man19}.)
	
	In spite of its restricted purview, we believe Theorem \ref{thm:main_uq2} is significant in two respects. First, as with Theorem \ref{thm:main_uq}, our argument is straightforward, and demonstrates an intrinsic geometric connection between quantum $\mathfrak{gl}(1 \vert 1)$ and the Heegaard Floer framework. In contrast, the arguments of \cite{epv19, man19} are indirect and more computational. The resulting correspondence between our tangle Floer invariant and $U_q(\mathfrak{gl}(1 \vert 1))$ is much simpler as well. The arc algebras involved in our tangle Floer theory are relatively lightweight, and we obtain the $U_q(\mathfrak{gl}(1 \vert 1))$ braid representation on the nose. There is no need to extend the representation as in \cite{epv19}, or to consider subalgebras of the arc algebra as in \cite{man19}. We expect our invariant $\widehat{BFA}(Y(T))$ to mirror the behavior of the $U_q(\mathfrak{gl}(1 \vert 1))$ Reshetikhin-Turaev invariant more closely than other tangle Floer theories; this is something that we also plan to address in the sequel.
	
	\subsection{Discussion and related work}
	\label{sec:rel_work}
	
	As alluded to above, Theorem \ref{thm:main_burau} results from combining two long-understood ideas: the relationship between knot Floer homology and Alexander matrices, and the fact that an Alexander matrix can be obtained from the Burau representation. The main tool in our proofs is a choice of Heegaard diagram which unites these perspectives; we then relate our diagram to the Burau representation using a variant of the fork/noodle calculus developed by Krammer and Bigelow \cite{kra00, big01}. The fork/noodle calculus also provides a way to understand the $U_q(\mathfrak{gl}(1 \vert 1))$ braid representation, by counting certain sets of intersections in a fork/noodle diagram, which look a lot like generators of the knot Floer chain complex.
	
	In spite of the familiarity of the ideas involved, a study of the link between knot Floer homology and the Burau representation seems mostly absent from existing literature. Baldwin and Grisgy consider braid invariants coming from link Floer homology, and connect their work with the Burau representation, but only indirectly, in the sense that both recover the Alexander polynomial \cite[Remark 2.3]{bagr15}. On the other hand, Khovanov and Seidel \cite{ks02} and Bouchair \cite{bou10} present Floer-theoretic categorifications (in different senses of the word) of the Burau representation. Both works use some version of the fork/noodle calculus, but neither mentions knot Floer homology. Within this context, our Theorem \ref{thm:main_burau} can be seen as generalizing and unifying Bouchair's work (with individual Burau matrix elements) and Baldwin-Grigsby's (with the entire Burau matrix). We also note that an algebraic categorification of the Burau representation at prime roots of unity was defined by Qi and Sussan \cite{qisu16}.
	
	We believe the connection between the Burau representation and knot Floer homology is well worth understanding in detail. On the representation-theoretic side, Krammer and Bigelow's fork/noodle calculus has been a powerful tool for studying braid representations, yielding the first proof that braid groups are linear \cite{big01}. On the other side, knot Floer homology has led to numerous advancements in low-dimensional topology, for example through its ability to detect knot fibering and genus \cite{ni07, gh08, os04b, mos09}.
	
	\subsection{Further directions}
	
	Our Theorem \ref{thm:main_burau} begs the question:
	
	\begin{question}
		What topological information do the groups $\widehat{HFB}$ encode?
	\end{question}
	
	More specifically, one can ask:
	
	\begin{question}
		\label{q:word}
		For a given braid $\rho \in B_n$, does the family of groups $\widehat{HFB}(\rho, {\bf j}, {\bf k})$ for different multi-indices ${\bf j, k} \subset \{1, \dots n\}$ completely determine $\rho$?
	\end{question}
	
	Since the word problem has been solved for braid groups \cite{art25, gar69}, Question \ref{q:word} is of more theoretical than practical interest. Nonetheless, it seems likely that the groups $\widehat{HFB}$ contain useful geometric data about braids and their link closures. We also remark that, for simplicity, we've limited ourselves in this paper to defining ``hat'' flavors of Heegaard Floer homology. We expect that a full braid Floer chain complex could be defined without too much trouble, which would contain additionally information.
	
	\subsection{Outline}
	
	We have attempted to make this paper self-contained, so that it can be read by both Heegaard Floer theorists and quantum topologists. Section \ref{sec:ubs} gives an overview of the full Burau representation, and defines our variation of the fork/noodle calculus. Section \ref{sec:hfs} then provides background on Heegaard Floer homology---those who are already familiar with the topic may be able to skip this section. We define our homology theories in Section \ref{sec:theory}, and discuss some of their properties in Section \ref{sec:propss}.
	
	In Section \ref{sec:qs}, we begin moving toward the proofs of Theorems \ref{thm:main_uq} and \ref{thm:main_uq2}. We describe Reshitikhin-Turaev invariants, and the $U_q(\mathfrak{gl}(1 \vert 1))$ invariant in particular. In Section \ref{sec:lasts}, we first prove Theorem \ref{thm:main_uq}, and then define our bordered sutured Floer homology groups for tangles. Finally, we prove Theorem \ref{thm:main_uq2}.
	
	\subsection{Acknowledgments}
	
	The author is grateful to Imogen Montague for forcing him to learn bordered sutured Heegaard Floer homology, and to Patricia Sorya for a close reading of a draft of this paper. The author also thanks John Baldwin for encouraging conversations, and the graduate students of Boston College for letting him crash their seminars for the past three years and sometimes eat their food.
	
	\section{Notation}
	\label{sec:not}
	
	We will mostly think of the $n$-strand braid group, $B_n$, as the mapping class group of the $n$-punctured disk, but we will sometimes consider elements of $B_n$ as isotopy classes of arcs in the unit cube. To avoid confusion, we refer to the latter objects as {\em geometric} braids. Additonally, we use $[n]$ to indicate the set $\{1, \dots, n\}$ for any $n \in \Z_{\geq 0}$.
	
	If $\alpha$ and $\beta$ are two oriented, half-dimensional submanifolds of an oriented, even-dimensional manifold $Y$, than $\alpha \cdot \beta$ indicates the signed count of intersections between $\alpha$ and $\beta$. If $x \in \alpha \cap \beta$, then $\text{sgn}(x) \in \{-1,1\}$ is the sign of $x$ as an intersection point. When $\dim(Y)$ is not divisible by four, this sign depends on the order of the pairing; if we write $x \in \alpha \cap \beta$, then we are implicitly taking the order $\alpha \cdot \beta$. Separately, if $\sigma \in S_n$ is an element of the symmetric group on $n$ elements, then $\text{sgn}(\sigma)$ is the sign of $\sigma$, which lies in $\{0,1\}$.
	
	If $Y$ is any topological space, then $\text{Sym}^k(Y)$ denotes the $k$-fold symmetric product of $Y$, i.e.~the quotient of the $k$-fold product of $Y$ by the $S_k$ action that permutes coordinates:
	$$
	\text{Sym}^k(Y) = (Y \times \cdots \times Y) / \{(y_1, \dots, y_k) \sim (y_{\sigma(1)}, \dots, y_{\sigma(k)}) \mid y_j \in Y, \sigma \in S_k\}.
	$$
	If $y = [(y_1, \dots, y_k)]$ is a point of $\text{Sym}^k(Y)$ with $k$ distinct coordinates, then we will often conflate $y$ with the $k$-element subset $\{y_1, \dots, y_k\}$ of $Y$. We will also frequently conflate curves and arcs in surfaces with homology classes or relative homology classes they represent.
	
	\section{Background on the Burau Representation}
	\label{sec:ubs}
	
	\subsection{The full Burau representation}
	\label{sec:ub}
	
	Let $D \subset \C$ be the closed unit disk, and let $z_1, \dots, z_n$ be $n$ points in $D$ with zero imaginary part and
	$$
	-1 < \mathfrak{Re}(z_1) < \cdots < \mathfrak{Re}(z_n) < 1.
	$$
	Let ${\bf D}_n$ denote the {$n$-punctured disk} $D - \{z_1, \dots, z_n\}$, and define the {\em $n$-strand braid group} $B_n$ to be the mapping class group of homeomorphisms of ${\bf D}_n$ which fix $\partial D$ pointwise. For $j = 1, \dots, n - 1$, let $\sigma_j \in B_n$ be the mapping class of the homeomorphism which exchanges the punctures $z_j$ and $z_{j + 1}$ by a clockwise half twist; then the elements $\sigma_1, \dots, \sigma_{n - 1}$ are a standard set of generators for $B_n$.
	
	Our definition of the {\em full}, or {\em unreduced}, {\em Burau representation} $\psi_n : B_n \to GL_n(\Z[t,t^{-1}])$ follows \cite{brma18}. Let $x$ be the basepoint $-i \in \partial D$, and let 
	$$
	\omega : \pi_1({\bf D}_n, x) \to \Z
	$$
	be the homomorphism given by counting the total winding number of an element of $\pi_1({\bf D}_n, x)$ around each puncture $z_j$. Let ${\bf C}_n$ be the cover of ${\bf D}_n$ corresponding to $\ker(\omega)$; then the deck transformation group of ${\bf C}_n$ is infinite cyclic, and we let $t$ denote a (multiplicative) generator of this group.
	
	Let $\{\tilde{x}_k\}_{k \in \Z}$ be the preimage of $x$ in ${\bf C}_n$, indexed so that $t \tilde{x}_k = \tilde{x}_{k + 1}$. Additionally, for $j = 1, \dots, n$, let $\eta_j \subset {\bf D}_n$ be the oriented arc which starts and ends at $x$ and loops clockwise once around $z_j$, as in Figure \ref{fig:fna}. Let $\tilde{\eta}_j$ be the lift of $\eta_j$ to ${\bf C}_n$ which begins at $\tilde{x}_0$. Then the relative homology group $H_1({\bf C}_n, \{\tilde{x}_k\})$ is a free $\Z[t,t^{-1}]$-module of dimension $n$, with preferred basis $\tilde{\eta}_1, \dots, \tilde{\eta}_n$.
	
	Let $\rho \in B_n$ be a braid represented by a homeomorphism $\rho: {\bf D}_n \to {\bf D}_n$. Then $\rho$ lifts uniquely to a homeomorphism $\tilde{\rho} : {\bf C}_n \to {\bf C}_n$ which fixes the fiber $\{\tilde{x}_k\}$ pointwise, and the map $\tilde{\rho}$ induces a $\Z[t,t^{-1}]$-module automorphism $\tilde{\rho}_*$ of $H_1({\bf C}_n, \{\tilde{x}_k\})$. The full Burau representation is the map
	\begin{align*}
	&\psi_n : B_n \to \text{Aut}(H_1({\bf C}_n, \{\tilde{x}_k\})) \cong GL_n(\Z[t,t^{-1}]), \\
	&\psi_n(\rho) = \tilde{\rho}_*.
	\end{align*}
	The representation is defined explicitly by
	\begin{equation}
		\label{eq:explicit_burau}
		\sigma_j \mapsto \begin{bmatrix}
			I_{j - 1} & 0 & 0 & 0 \\
			0 & 1-t & 1 & 0 \\
			0 & t & 0 & 0 \\
			0 & 0 & 0 & I_{n - j - 1}
		\end{bmatrix}
	\end{equation}
	for $j = 1, \dots, n - 1$.
	
	\begin{rmk}
		Some authors define the full Burau representation to be the transpose of the one we've given. Our convention will be justified by computations in the next section.
	\end{rmk}
	
	\begin{rmk}
		\label{rem:red}
		The unreduced Burau representation is so named because it fixes the diagonal subspace
		$$
		\langle (1, \dots, 1) \rangle \in \Z[t,t^{-1}]^n.
		$$
		The {\em reduced Burau representation} is defined using the action of $B_n$ on $H_1({\bf C}_n)$, which is closely related to the action of $B_n$ on $H_1({\bf C}_n, \{\tilde{x}_k\})/\langle (1, \dots, 1) \rangle$.
	\end{rmk}
	
	\subsection{The noodle/fork/chopstick calculus}
	\label{sec:fn}
	
	The Burau representation can be computed diagrammatically, using a calculus developed by Krammer and Bigelow \cite{kra00, big01}. For this, it's helpful to consider the following cut-and-paste construction of ${\bf C}_n$. For $j \in [n]$, let $\ell_j$ be a vertical line segment which begins at $z_j$ and moves upward to $\partial D$. Let
	$$
	X = {\bf D}_n \setminus (\ell_1 \cup \ell_2 \cup \dots \cup \ell_n),
	$$
	and for each $j$ let $\ell^+_j, \ell^-_j \subset \partial X$ be the boundary arcs of $X$ coming from the left and right sides of $\ell_j$ respectively. Let $\{X_k\}$ be $\Z$ copies of $X$; then ${\bf C}_n$ can be constructed by gluing the copy of $\ell_j^+ - \{z_j\}$ in $X_k$ to the copy of $\ell_j^- - \{z_j\}$ in $X_{k + 1}$ for all $k \in \Z$, $j \in \{1, \dots, n\}$. We fix indices so that $\tilde{x}_k \in X_k$ for all $k \in \Z$.

	For $j \in [n]$, let $\varepsilon_j$ be a properly embedded, oriented arc in $D$ which begins at $z_j$ and travels upward to $\partial D$, parallel to $\ell_j$ and to the left of it, so that $\varepsilon_j \cap \ell_j = \{z_j\}$---see Figure \ref{fig:fna}. Let $\tilde{\varepsilon}_j$ be the lift of $\varepsilon_j$ to ${\bf C}_n$ contained in $X_0$, so that $t^m \tilde{\varepsilon}_j$ is contained in $X_m$ for $m \in \Z$. Then the arcs $\tilde{\varepsilon}_1, \dots, \tilde{\varepsilon}_n$ give a basis for $H_1({\bf C}_n, \partial {\bf C}_n - \{\tilde{x}_k\})$ as a free $\Z[t, t^{-1}]$-module, and there is an intersection pairing
	$$
	 H_1({\bf C}_n, \{\tilde{x}_k\}) \times H_1({\bf C}_n, \partial {\bf C}_n - \{\tilde{x}_k\}) \to \Z[t, t^{-1}]
	$$
	defined as follows. For $e \in H_1({\bf C}_n, \{\tilde{x}_k\})$ and $j \in [n]$, set
	\begin{equation}
		\label{eq:pairing}
		\langle e, \tilde{\varepsilon}_j \rangle = \sum_{m \in \Z} (e \cdot (t^m \tilde{\varepsilon}_j))t^m \in \Z[t,t^{-1}].
	\end{equation}
	Here $e$ and $t^m \varepsilon_j$ are represented by oriented arcs which intersect transversely for all $m$, and $\cdot$ is their algebraic intersection number. Since $e$ is represented by a compact arc, the summation is non-zero for all but finitely many values of $m$, and is well defined. Additionally,
	$$
	\langle \tilde{\eta}_j, \tilde{\varepsilon}_k \rangle = \delta_{jk},
	$$
	where $\delta_{jk}$ is the Kronecker delta. Therefore the pairing is non-degenerate, and the basis $\{\tilde{\varepsilon}_j\}$ of $H_1({\bf C}_n, \partial {\bf C}_n - \{\tilde{x}_k\}) \cong  H^*_1({\bf C}_n, \{\tilde{x}_k\})$ is dual to the basis $\{\tilde{\eta}_j\}$. It follows that, for all $e \in H_1({\bf C}_n, \{\tilde{x}_k\})$,
	$$
	e = \sum_{j = 1}^n \langle e, \tilde{\varepsilon}_j \rangle \eta_j.
	$$
	From this we conclude that if $\rho \in B_n$, and $(\psi_n(\rho))_{jk}$ is the $jk$ entry of the full Burau representation of $\rho$, then
	\begin{equation}
		\label{eq:fn}
		(\psi_n(\rho))_{jk} = \langle \tilde{\rho}(\tilde{\eta}_j) , \tilde{\varepsilon}_k \rangle.
	\end{equation}
	
	To simplify notation, we make the following definitions.
	
	\begin{defn}
		\label{def:fn}
		Fix $n > 0$, and let ${\bf D}_n$ be the $n$-punctured disk as above. Then, for all $j \in [n]$:
		\begin{itemize}
			\item The arc $\eta_j \subset {\bf D}_n$ is called the {\em $j$th standard noodle}.
			\item The arc $\varepsilon_j \subset {\bf D}_n$ is the {\em $j$th standard chopstick}.
			\item The arc $\ell_j \subset {\bf D}_n$ is the {\em $j$th cut arc}.
		\end{itemize}
	\end{defn}
	
	Figure \ref{fig:fna} shows a standard fork, standard noodle, and cut arcs in the case of ${\bf D}_3$. 
	
	\begin{rmk}
		The term ``noodle'' was first used by Bigelow in the context of a similar intersection pairing \cite{big01}. In Bigelow's formulation, which computes the reduced Burau representation \cite{big02}, the dual objects to noodles are called ``forks'' after a paper by Krammer \cite{kra00}. Since we care about the unreduced representation, our pairing is slightly different, and ``chopstick'' feels more appropriate.
	\end{rmk}
	
	\begin{figure}
		\labellist
		\small\hair 2pt
		\pinlabel $\eta_1$ at 235 50
		\pinlabel $\varepsilon_2$ at 315 325
		\pinlabel $\ell_1$ at 245 325
		\pinlabel $\ell_2$ at 395 325
		\pinlabel $\ell_3$ at 545 325
		\pinlabel $\sigma_1(\eta_1)$ at 860 40
		\pinlabel $\sigma_1(\eta_2)$ at 1195 60
		\pinlabel $y_1$ at 940 420
		\pinlabel $y_2$ at 1280 265
		\pinlabel $y_3$ at 865 380
		\pinlabel $y_4$ at 1265 115
		\endlabellist
		
		\centering
		\subcaptionbox{The standard noodle $\eta_1$, standard chopstick $\varepsilon_2$, and three cut arcs in ${\bf D}_3$\label{fig:fna}}{
			\includegraphics[height=4.5cm]{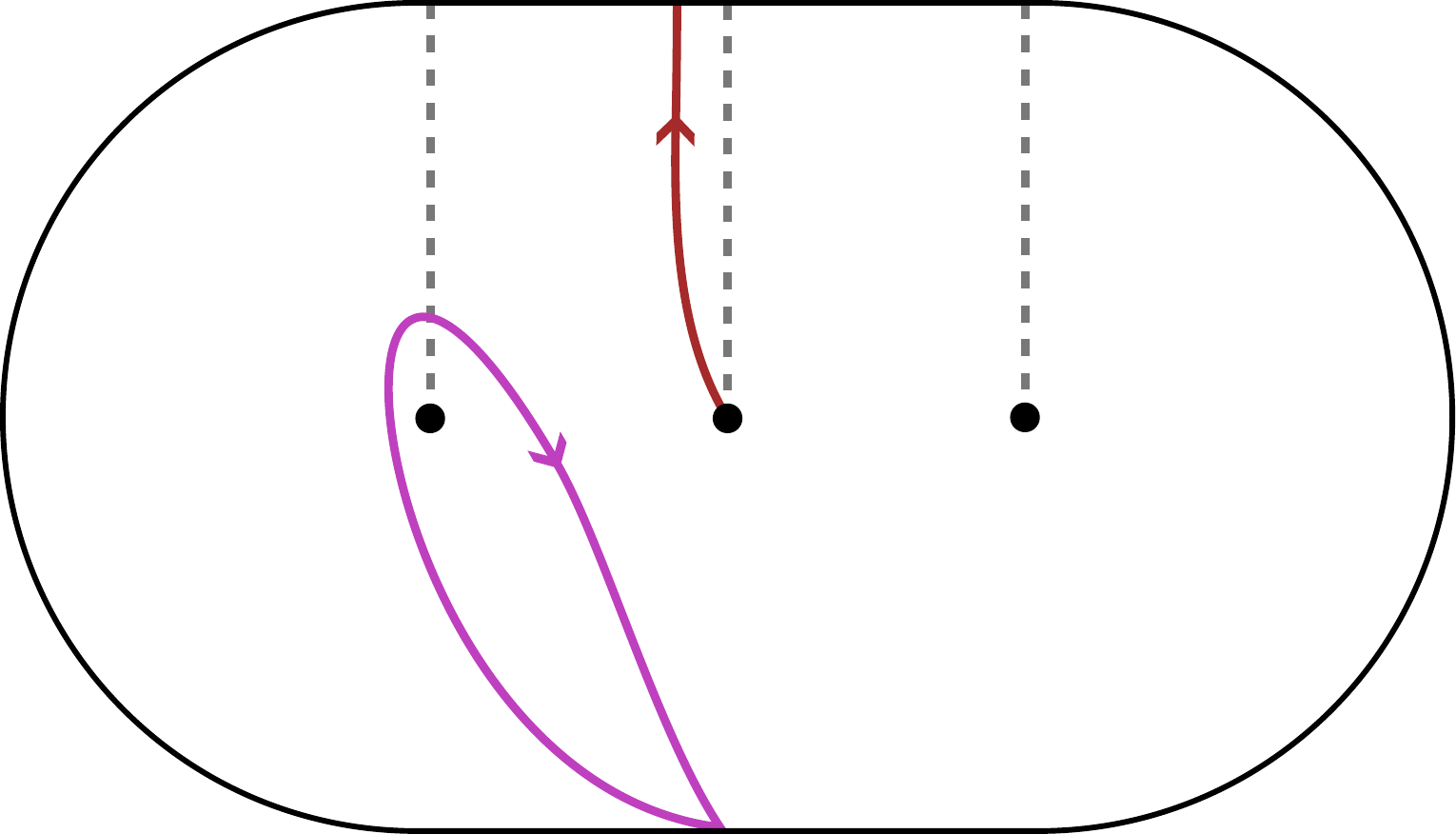}
		}
		\hspace{1cm}
		\subcaptionbox{The chopstick/noodle diagram for the generator $\sigma_1$ in $B_2$ can be used to derive the matrix (\ref{eq:explicit_burau}) \label{fig:fnb}}{
			\includegraphics[height=4.5cm]{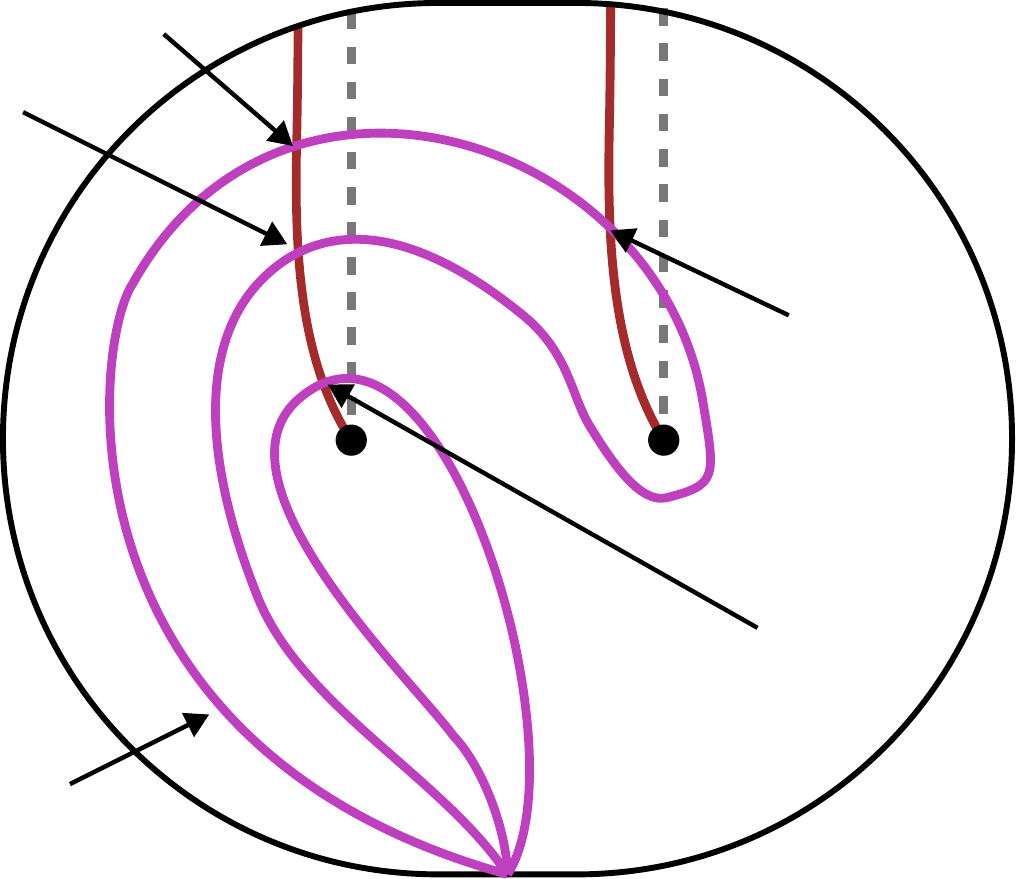}
		}
		\caption{}
		\label{fig:fn}
	\end{figure}

	\subsection{A Heegaard Floer-ish reformulation}
	\label{sec:reform}
	
	In this section, we use the noodle/chopstick calculus to define the Burau representation without referring to the cover ${\bf C}_n$. This formulation of the representation is implicit in Bigelow's work \cite{big02}, and is also used by Bouchair and Khovanov-Seidel \cite{bou10, ks02} (and cf.~\cite{klw01}). Each of these papers deals only with the {\em reduced} Burau representation, however, so we briefly develop our version of the theory. Throughout, we use $\eta$ and $\varepsilon$ to represent standard noodles and chopsticks as in Definition \ref{def:fn}.
	
	First, given a braid $\rho \in B_n$ we expand (\ref{eq:fn}) using (\ref{eq:pairing}):
	\begin{equation}
		\label{eq:first_expansion}
		(\psi_n(\rho))_{jk} = \sum_{m \in \Z} (\tilde{\rho}(\tilde{\eta}_j) \cdot (t^m \tilde{\varepsilon}_k))t^m = \sum_{m \in \Z} \Big( \sum_{\tilde{y} \in  \tilde{\rho}(\tilde{\eta}_j) \cap (t^m \tilde{\varepsilon}_k)} \text{sgn}(\tilde{y}) t^m \Big).
	\end{equation}
	
	Next, there is a natural bijection between the intersection sets
	$$
	\rho(\eta_j) \cap \varepsilon_k \leftrightarrow \tilde{\rho}(\tilde{\eta}_j) \cap \big( \bigcup_{m \in \Z} t^m \tilde{\varepsilon}_k \big),
	$$
	since each point $y \in \rho(\eta_j) \cap \varepsilon_k$ lifts uniquely to a point $\tilde{y} \in \tilde{\rho}(\tilde{\eta}_j) \cap (t^m \tilde{\varepsilon}_k)$ for some $m \in \Z$. We call the relevant $m$ the {\em sheet} of $\tilde{y}$:
	
	\begin{defn}
		\label{def:sheet}
		With notation as above, let $y$ be a point in $\rho(\eta_j) \cap \varepsilon_k$, and let $\tilde{y} \in {\bf C}_n$ be the unique lift of $y$ contained in $\tilde{\rho}(\tilde{\eta}_j)$. Then the {\em sheet} of $\tilde{y}$ is the unique $m$ such that $\tilde{y} \in \tilde{\rho}(\tilde{\eta}_j) \cap t^m\tilde{\varepsilon}_k$, and we define the {\em sheet function}
		$$
		S : \rho(\eta_j) \cap \varepsilon_k \to \Z
		$$
		by sending $y \in \rho(\eta_j) \cap \varepsilon_k$ to the sheet of $\tilde{y}$.
	\end{defn}
	
	Now, (\ref{eq:first_expansion}) can be rewritten.
	\begin{lemma}
		\label{lem:reform}
		The Burau representation of a braid $\rho \in B_n$ is determined by
		$$
			(\psi_n(\rho))_{jk} = \sum_{y \in \rho(\eta_j) \cap \varepsilon_k} \text{sgn}(y) t^{S(y)},
		$$
		where $S$ is the sheet function of Definition \ref{def:sheet}.
	\end{lemma}
	
	Figure \ref{fig:fnb} demonstrates how Lemma \ref{lem:reform} can be used to compute the Burau matrix (\ref{eq:explicit_burau}) of a generator of $B_2$. Referring to the figure, the relevant sheet values are $S(y_1) = S(y_4) = 0$ and $S(y_2) = S(y_3) = 1$, and all intersection points are positive except $y_3$.
		
	Our final step is redefining $S(y)$ as an intersection count. Given $\rho \in B_n$ and $y \in \rho(\eta_j) \cap \varepsilon_k$ for some $j, k \in [n]$, let $a$ be the oriented sub-arc of $\rho(\eta_j)$ which begins at the initial point of $\rho(\eta_j)$ and terminates at $y$. Similarly, let $b$ be the oriented sub-arc of $\varepsilon_k$ which begins at $y$ and ends at $\partial D$. Then $a \cup b$ is a properly embedded arc in ${\bf D}_n$, which we denote by $\gamma_y$.
	
	\begin{lemma}
		\label{lem:s_intersect}
		Let $\rho \in B_n$ be a braid, and $y \in \rho(\eta_j) \cap \varepsilon_k$ for some $j, k \in [n]$. Let the arc $\gamma_y \subset {\bf D}_n$ be defined as above. Then
		$$
		S(y) = \sum_{r = 1}^n \gamma_y \cdot \ell_r,
		$$
		where $S(y)$ is the sheet function of Definition \ref{def:sheet} and $\ell_r$ is the $r$th cut arc of Definition \ref{def:fn}.
	\end{lemma}

	\begin{proof}
		Write $\gamma_y = a \cup b$, as above. Then, since $b \subset \varepsilon_k$ and $\varepsilon_k$ is disjoint from every cut arc, we have
		$$
		\sum_{r = 1}^n \gamma_y \cdot \ell_r = \sum_{r = 1}^n a \cdot \ell_r.
		$$
		It is straightfoward to verify that this quantity coincides with $S(y)$, using the cut-and-paste construction of ${\bf C}_n$ in the previous section.
	\end{proof}
	
	The arc $\gamma_y$ used here is analogous to the difference curves in Heegaard-Floer homology, which we define in the next section.
	
	\section{Background on Heegaard Floer Theory}
	\label{sec:hfs}
	
	\subsection{The Framework}
	\label{sec:hf}
	
	Heegaard Floer homology was originally defined by Ozsvath and Szabo \cite{os04c} for closed three-manifolds. Its basic input data is a {\em multi-pointed Heegaard diagram} $\mathcal{H} = (\Sigma, \alpha, \beta, p)$. By this we mean:
	\begin{itemize}
		\item An oriented, genus $g$ surface $\Sigma$.
		\item A set of distinguished basepoints $p = \{p_1, \dots, p_r\} \subset \Sigma$.
		\item Two sets of $k \geq g$ pairwise disjoint, simple closed curves in $\Sigma \setminus p$, labelled \break $\alpha = \{\alpha_1, \dots, \alpha_k\}$ and $\beta = \{\beta_1, \dots, \beta_k\}$, such that
		$$
		\text{rk}(\text{span}(\alpha_1, \dots, \alpha_k)) = \text{rk}(\text{span}(\beta_1, \dots, \beta_k)) = g
		$$
		as subspaces of $H_1(\Sigma)$, and such that $\alpha$ and $\beta$ are each homologically independent as subsets of $H_1(\Sigma \setminus p)$.
	\end{itemize}

	One can obtain a closed, oriented three-manifold $Y$ from a Heegaard diagram as follows: we begin with $\Sigma \times I$, and for each $j \in [k]$ we attach a two-handle to $\Sigma \times \{0\}$ along $\alpha_j \times \{0\}$, and attach another two-handle to $\Sigma \times \{1\}$ along $\beta_j \times \{1\}$. The resulting manifold has boundary a finite number of spheres, and we fill each of these with a three-ball to produce $Y$.
	
	\begin{conv}
		To avoid some technical complications in the following discussion, we assume $b_1(Y) = 0$.
	\end{conv}

	Given such a Heegaard diagram $\mathcal{H} = (\Sigma, \alpha, \beta, p)$, we consider the {\em $k$-fold symmetric product} $\text{Sym}^k(\Sigma)$ defined in Section \ref{sec:not}. The multi-curves $\alpha$ and $\beta$ embed in $\text{Sym}^k(\Sigma)$ as half-dimensional tori, $\mathbb{T}_\alpha$ and $\mathbb{T}_\beta$:
	\begin{align*}
		\mathbb{T}_\alpha &= \alpha_1 \times \cdots \times \alpha_k \hookrightarrow \text{Sym}^k(\Sigma) \\
		\mathbb{T}_\beta &= \beta_1 \times \cdots \times \beta_k \hookrightarrow \text{Sym}^k(\Sigma),
	\end{align*}
	and to each basepoint $p_j \in p$, we associate a hypersurface $\hat{p}_j$:
	$$
	\hat{p}_j = \{p_j\} \times \text{Sym}^{k -1}(\Sigma) \hookrightarrow \text{Sym}^k(\Sigma).
	$$
	The manifold $\text{Sym}^k(\Sigma)$ can be endowed with a symplectic structure coming from a symplectic structure on $\Sigma$, and the tori $\mathbb{T}_\alpha$ and $\mathbb{T}_\beta$ can be perturbed to Lagrangian submanifolds of $\text{Sym}^k(\Sigma)$.
	
	The Heegaard Floer homology $\widehat{HF}(\mathcal{H})$ of $\mathcal{H}$ is a variant of the Lagrangian Floer homology of $\mathbb{T}_\alpha$ and $\mathbb{T}_\beta$ in $\text{Sym}^k(\Sigma)$. This means $\widehat{HF}(\mathcal{H})$ is the homology of a chain complex $\widehat{CF}(\mathcal{H})$ over $\Z$, which is freely generated by the set of intersection points $\mathbb{T}_\alpha \cap \mathbb{T}_\beta$. Such a point $x \in \mathbb{T}_\alpha \cap \mathbb{T}_\beta$ is equivalent to a set of $k$ points $\{x_1, \dots, x_k\} \subset \Sigma$ with $x_j \in \alpha_j \cap \beta_{\sigma(j)}$ for each $j \in [k]$ and some fixed permutation $\sigma \in S_k$.
	
	The differential of $\widehat{CF}(\mathcal{H})$ is defined by counting homotopy classes of certain ``Whitney disks'' between two generators $x, y \in \mathbb{T}_\alpha \cap \mathbb{T}_\beta$. If $D \subset \C$ is the closed unit disk, then a Whitney disk is a map
	\begin{equation}
		\label{eq:disk}
		u : D \setminus \{i, - i\} \to \text{Sym}^k(\Sigma)
	\end{equation}
	such that $\lim_{z \to i} u(z) = x$, $\lim_{z \to -i} u(z) = y$,
	\begin{align*}
		&u(\{z \in \partial D \mid \mathfrak{Re}(z) < 0\}) \subset \mathbb{T}_\alpha, \text{ and} \\
		&u(\{z \in \partial D \mid \mathfrak{Re}(z) > 0\}) \subset \mathbb{T}_\beta.
	\end{align*}
	The map $u$ is also required to be pseudoholomorphic in an appropriate sense. Explicitly, if $\pi_2(x,y)$ denotes the set of homotopy classes of such disks, then the differential $\partial$ of $\widehat{CF}(\mathcal{H})$ is given by
	$$
	\partial(x) = \sum_{y \in \mathbb{T}_\alpha \cap \mathbb{T}_\beta} \#\{ D \in \pi_2(x,y) \mid \hat{p}_j \cdot D = 0 \text{ for all } j, \ \mu(D) = 1\} \cdot y.
	$$
	Here $\#\{*\} \in \Z$ is a signed count of such disks, and $\hat{p}_j \cdot D$ is the algebraic intersection number. The function $\mu$ is the {\em Maslov index}, which is a characteristic class taking values in $\Z$. The details here won't be important to us, but we note that any map $u$ as in (\ref{eq:disk}) is naturally equivalent to a map
	\begin{equation}
		\label{eq:branched_disk}
		\bar{u} : \bar{D} \to \Sigma,
	\end{equation}
	where $\bar{D}$ is a $k$-fold branched cover of $D$ depending on $u$. The hat flavor of Heegaard Floer homology, $\widehat{HF}(\mathcal{H})$, is the homology of $\widehat{CF}(\mathcal{H})$.
	
	Our assumption that $b_1(Y) = 0$ implies any two generators $x, y \in \mathbb{T}_\alpha \cap \mathbb{T}_\beta$ of $\widehat{CF}(\mathcal{H})$ can be connected by a Whitney disk, albeit one which may not contribute to the differential. One can use these disks to define a relative $\Z$-grading $\widetilde{M}$ on $\widehat{CF}(\mathcal{H})$ and $\widehat{HF}(\mathcal{H})$, called the (relative) {\em Maslov grading}, whose precise formulation depends on the basepoints. This grading can also be non-canonically upgraded to an absolute Maslov grading
	$$
	M : \widehat{CF}(\mathcal{H}) \to \Z
	$$
	by choosing some $x \in  \mathbb{T}_\alpha \cap \mathbb{T}_\beta$ and declaring $M(x) = 0$. We will not use the precise definition of the Maslov grading, but we will need the following fact.
	
	\begin{lemma}
		\label{lem:maslov}
		Fix an absolute Maslov grading $M$ on $\widehat{CF}(\mathcal{H})$, and fix orientations on $\mathbb{T}_\alpha$ and $\mathbb{T}_\beta$. Then for any $x \in \mathbb{T}_\alpha \cap \mathbb{T}_\beta$,
		$$
		(-1)^{M(x)} = c \cdot \text{sgn}(x),
		$$
		where $c \in \{+1, - 1\}$ is a global constant that does not depend on $x$.
	\end{lemma}
	
	It is not necessary to orient $\text{Sym}^k(\Sigma)$ in Lemma \ref{lem:maslov}, because this space has an orientation induced by the orientation of $\Sigma$. Similarly, to orient $\mathbb{T}_\alpha$ and $\mathbb{T}_\beta$, it suffices to orient each $\alpha$ and $\beta$ curve. Then if $x = \{x_1, \dots, x_k\} \in \mathbb{T}_\alpha \cap \mathbb{T}_\beta$, such that $x_j \in \alpha_j \cap \beta_{\sigma(j)}$ for some $\sigma \in S_k$, the sign of $x$ can be computed from the Heegaard diagram by the formula
	\begin{equation}
		\label{eq:orientation}
		\text{sgn}(x) = (-1)^{\text{sgn}(\sigma)} \cdot \prod_{j = 1}^k \text{sgn}(x_j).
	\end{equation}
	Here $\text{sgn}(\sigma) \in \{0,1\}$ reflects whether $\sigma$ is even or odd, and sgn$(x_j) \in \{-1,1\}$ is the sign of $x_j$ as an intersection point $\alpha_j \cap \beta_{\sigma(j)}$ (see Section \ref{sec:not}).
	
	Although Lemma \ref{lem:maslov} is well-known in the literature---see, for example, \cite[Section 2.3.3]{man10}---it is difficult to find an explicit proof. We sketch one here for this reason, and because we will need it in Section \ref{sec:bstq}.
	
	\begin{proof}[Proof of Lemma \ref{lem:maslov}]
		Let $x, y \in \mathbb{T}_\alpha \cap \mathbb{T}_\beta$ be generators of $\widehat{CF}(\mathcal{H})$, and let $D$ be a Whitney disk from $x$ to $y$. Then the relative Maslov grading $\widetilde{M}(y) - \widetilde{M}(x)$ agrees mod 2 with the Maslov index $\mu(D)$. Define a map $\mathcal{O} : D \to \Z/2$ on such disks by
		$$
		\mathcal{O}(D) = \begin{cases}
			0 & \text{sgn}(x) = \text{sgn}(y) \\
			1 & \text{sgn}(x) \neq \text{sgn}(y)
		\end{cases};
		$$
		then we must check that $\mu(D) \equiv \mathcal{O}(D)$ mod 2. Via the correspondence of (\ref{eq:branched_disk}), to any such disk $D$ we can associate a {\em domain} $\mathcal{D}$, which is a formal linear combination of the components of
		$$
		\Sigma \setminus \{\alpha_1, \dots, \alpha_k, \beta_1, \dots, \beta_k\}.
		$$
		The Maslov of index of a domain can be defined combinatorially---see \cite{lip06, kru24} for details. In the special case where $\mathcal{D}$ is a single bigon or square with coefficient 1, one can use (\ref{eq:orientation}) and this combinatorial formula to check that
		$$
		\mathcal{O}(D) \equiv \mu(D) \equiv 1 \text{ mod } 2.
		$$
		In general, by \cite[Theorem 3.1]{kru24}, the Heegaard diagram $\mathcal{H}$ can be adjusted in a controlled way to represent any domain $\mathcal{D}$ by a new domain $\mathcal{D}'$ which is a composition of bigons and rectangles. The result then follows from the fact that the index $\mu$ and orientation difference $\mathcal{O}$ are unchanged by these adjustments, and are both additive under compositions.
	\end{proof}
	
	\subsection{Link Floer homology}
	\label{sec:hfk}
	
	Let $K \subset S^3$ be an oriented link with $n(K)$ components. A {\em Heegaard surface adapted to $K$} is an embedded, oriented, genus $g$ surface $\Sigma \subset S^3$ such that:
	\begin{itemize}
		\item The closure of each component of $S^3 - \Sigma$ is a handlebody. We denote these by $H_1$ and $H_2$.
		\item $K$ intersects $\Sigma$ in $2m$ points for some $m > 0$, and $K \cap H_j$ is a trivial $m$-tangle for $j = 1, 2$. By this we mean that $K \cap H_j$ consists of $m$ properly embedded arcs, each of which is unknotted and parallel to $\Sigma = \partial H_j$ in the complement of the other $m - 1$ components.
	\end{itemize}
	To such a surface $\Sigma$, we can associate a {\em multi-pointed Heegaard diagram adapted to $K$},
	$$
	\mathcal{H} = (\Sigma, {\bf \alpha}, {\bf \beta}, z \cup w).
	$$
	This consists of the following data.
	\begin{itemize}
		\item Two sets of $m$ basepoints, $z = \{z_1, \dots, z_m\}$ and $w = \{w_1, \dots, w_m \}$ in $\Sigma$, such that $z$ is the set of positive intersections in $K \cap \Sigma$ and $w$ is the negative intersection points.
		\item Two sets of $g + m - 1$ pairwise disjoint, simple closed curves $\alpha = \{\alpha_1, \dots, \alpha_{g + m - 1}\}$ and $\beta = \{\beta_1, \dots, \beta_{g + m - 1} \}$, such that:
		\begin{itemize}
			\item The curves $\alpha$ are homologically independent in $\Sigma \setminus (z \cup w)$, as are the the curves $\beta$.
			\item Each $\alpha_j$ bounds a disk in $H_1 \setminus K$, and each $\beta_j$ bounds a disk in $H_2 \setminus K$.
		\end{itemize}
	\end{itemize}
	The hat flavor of link Floer homology is the Heegaard Floer homology of this diagram, as defined in the preceeding section:
	$$
	\widehat{HFL}(K) = \widehat{HF}(\mathcal{H}).
	$$
	These homology groups are an invariant of the link $K$.
		
	In addition to the relative Maslov grading, the chain complex $\widehat{CF}(\mathcal{H})$ admits a relative $\Z^{n(K)}$-grading called the {\em Alexander grading}. To define this, fix two generators $x = \{x_1, \dots, x_{g + m - 1}\}, y = \{y_1, \dots, y_{g + m - 1}\} \in \mathbb{T}_\alpha \cap \mathbb{T}_\beta$. Then each curve $\alpha_j$ contains exactly one $x_k \in x$ and one $y_\ell \in y$, and for each $j \in [g + m - 1]$ we choose an oriented arc $a_j \subset \alpha_j$ with initial point $x_k$ and terminal point $y_\ell$. Likewise, each curve $\beta_j$ contains exactly one $x_k \in x$ and $y_\ell \in y$, and we choose an oriented arc $b_j \subset \beta_j$ with initial point $y_\ell$ and terminal point $x_k$. The union of these arcs,
	\begin{equation}
		\label{eq:diff_curve}
		\gamma_{x, y} = a_1 \cup b_1 \cup \cdots \cup a_{g + m - 1} \cup b_{g + m - 1},
	\end{equation}
	is an oriented multi-curve in $\Sigma \setminus (z \cup w)$, which we also view as a multi-curve in $S^3 \setminus K$ via the embedding $\Sigma \subset S^3$. We call $\gamma_{x,y}$ a {\em difference curve} for $x$ and $y$. The {\em difference class} of $x$ and $y$ is the value
	\begin{equation}
		\label{eq:ma_grading}
		[\gamma_{x,y}] \in H_1(S^3 \setminus K) \cong \Z^{n(K)},
	\end{equation}
	where each oriented meridian of a component of $K$ generates a different summand of $\Z^{n(K)}$. Since each $\alpha$ and $\beta$ curve is nullhomologous in $S^3 \setminus K$, the difference class $[\gamma_{x,y}]$ does not depend on the choice of difference curve $\gamma_{x,y}$. The relative Alexander grading $\widetilde{A}$ is then defined by
	$$
	\widetilde{A}(x) - \widetilde{A}(y) = [\gamma_{x, y}] \in  \Z^{n(K)}.
	$$
	As with the Maslov grading, the Alexander grading can be non-canonically upgraded to an absolute grading $A : \widehat{HFL}(K) \to \Z$ by setting $A(x) = 0$ for an arbitrary generator.
	
	The differential of $\widehat{CF}(\mathcal{H})$ respects the Alexander grading. We sketch a proof of this fact, as it will be useful in the next section.
	
	\begin{lemma}
		\label{lem:alex_disc}
		Let $x, y \in \mathbb{T}_\alpha \cap \mathbb{T}_\beta$ be generators of $\widehat{CF}(\mathcal{H})$, and suppose there exists a Whitney disk
		$$
		u : D \setminus \{i, - i\} \to \text{Sym}^{g + m - 1}(\Sigma)
		$$
		from $x$ to $y$ with
		\begin{equation}
			\label{eq:zero_int}
			u(D) \cdot \hat{z}_j = u(D) \cdot \hat{w}_j = 0
		\end{equation}
		for all basepoints $z_j \in z$ and $w_j \in w$. Then $\widetilde{A}(x) - \widetilde{A}(y) = 0$.
	\end{lemma}

	\begin{proof}
		As discussed in Section \ref{sec:hf}, the map $u$ is equivalent to a map
		$$
		\bar{u} : \bar{D} \to \Sigma \setminus (z \cup w),
		$$
		where $\pi : \bar{D} \to D$ is some $(g + m - 1)$-fold branched cover of $D$. Broadly, the cover $\bar{D}$ and map $\bar{u}$ are determined by requiring that for all $z \in D$, $u(z) \in \text{Sym}^{g + m - 1}(\Sigma)$ and $\bar{u}(\pi^{-1}(z)) \subset \Sigma$ are equal when $u(z)$ is thought of as a subset of $\Sigma$. In particular, the image of the boundary $\bar{u}(\partial \bar{D}) \subset \Sigma$ is a difference curve $\gamma_{x,y}$ for $x$ and $y$. Since $u$ satisfies (\ref{eq:zero_int}), the domain $\bar{D} \to \Sigma$ has algebraic intersection number zero with all basepoints in $z$ and $w$. Viewng $\Sigma$ as a subset of $S^3$, this implies the curve $\gamma_{x,y}$ has zero linking number with $K$, and it follows that $[\gamma_{x,y}]= 0 \in H_1(S^3 \setminus K)$.
	\end{proof}
	Since the Alexander grading obstructs differentials by Lemma \ref{lem:alex_disc}, it descends to a relative grading on the groups $\widehat{HFL}(K)$. This grading is also an invariant of the link $K$.
	
	One appealing feature of link Floer homology is that its graded Euler characteristic is the multi-variable Alexander polynomial. Specifically, given a link $K \subset S^3$ with $n$ components, fix absolute Alexander and Maslov gradings for $\widehat{HFL}(K)$. We view the Alexander grading as a map
	$$
	A : \widehat{HFL}(K) \to \Z^n,
	$$
	and for $k \in \Z$ and ${\bf j} = (j_1, \dots, j_n) \in \Z^n$, let $\widehat{HFL}_{{\bf j}k}(K)$ denote the subgroup of $\widehat{HFL}(K)$ supported in Alexander grading ${\bf j}$ and Maslov grading $k$. Then
	\begin{equation}
		\label{eq:cat_alex}
		\Delta_K(t_1, \dots, t_n) = \sum_{j_1, \dots, j_n,k \in \Z} (-1)^k t_1^{j_1} \cdots t_n^{j_n} \text{rk}(\widehat{HFL}_{{\bf j}k}(K)).
	\end{equation}
	The fact that the Alexander and Maslov gradings are defined up to global shifts matches the fact that $\Delta_K$ is defined up to multiplication by units of $\Z[t_1^{\pm 1}, \dots, t_n^{\pm 1}]$.
	
	\subsection{Knot Floer homology}
	
	For any $n$-component link $K \subset S^3$, the relative Alexander $\Z^n$-grading on $\widehat{HFL}(K)$ can be collapsed to a relative $\Z$-grading by composing the homomorphism (\ref{eq:ma_grading}) with the map
	\begin{align*}
	\Z^{n} &\to \Z, \\
	(e_1, \dots, e_n) &\mapsto \sum_{j = 1}^n e_j.
	\end{align*}
	
	Homologically, this map sends each oriented meridian of $K$ to a generator of $\Z$, and coincides with the Alexander grading above if $K$ is a knot. When we use this single-variable Alexander grading rather than multi-variable one, we call the resulting Floer homology groups the {\em knot Floer homology} of $K$, denoted $\widehat{HFK}(K)$. Just as the link Floer groups categorify the multi-variable Alexander polynomial, the knot Floer groups categorify the single-variable Alexander polynomial in a manner analogous to (\ref{eq:cat_alex}).

	\section{The Burau Representation and Link Floer Homology}
	\label{sec:theory}
	
	\subsection{Bridge presentations for braid closures}
	\label{sec:bridges}
	
	As in Section \ref{sec:ub}, let $D \subset \C$ be the closed unit disk, and let $z = \{z_1, \dots, z_n\}$ be $n$ points in $D$ with zero imaginary part and 
	$$
	-1 < \mathfrak{Re}(z_1) < \mathfrak{Re}(z_2) < \cdots < \mathfrak{Re}(z_n) < 1.
	$$
	Additionally, let $w = \{w_1, \dots, w_n\} \in \C$ be $n$ points with $\mathfrak{Re}(w_j) = \mathfrak{Re}(z_j)$ and $\mathfrak{Im}(w_j) = -2$ for $j = 1, \dots, n$. We consider $z$, $w$, $D$ and $\C$ as subsets of $S^2$ by taking the one-point compactification $S^2 = \C \cup \{\infty\}$. Let $\rho \in B_n$ be a braid, considered as a homeomorphism $\rho : (D, z) \to (D, z)$ defined up to isotopy. Since $\rho$ fixes $\partial D$ pointwise, $\rho$ extends to a homeomorphism $S^2 \to S^2$ which fixes all points of $S^2 - D$, and we denote this extension also by $\rho$.
	
	Let $K \subset S^3$ be the oriented link which is the (geometric) braid closure of $\rho$, and let $m(K) \subset S^3$ be the mirror of $K$. (For convention reasons, we'll need to work with $m(K)$ rather than $K$.) We now describe a {\em bridge presentation} for $m(K)$ coming from $\rho$. By this we mean a diagram for $m(K)$ in $S^2$, along with a $1$-complex structure on $m(K)$ consisting of $2n$ vertices, $n$ {\em underbridge} arcs $u_1, \dots, u_n$, and $n$ {\em overbridge} arcs $o_1, \dots, o_n$, such that:
	\begin{itemize}
		\item The underbridges do not cross each other in the diagram, and likewise the overbridges do not cross each other.
		\item Whenever an underbridge crosses an overbridge, the former passes under the latter.
	\end{itemize}
	To define the overbridges, let $\bar{o}_j$ be a vertical line segment running from $w_j$ to $z_j$ for $j = 1, \dots, n$. Then let $o_j = \rho(\bar{o}_j)$ for all $j$.

	Unlike the overbridges, which depend on $\rho$, the underbridges will be the same for all braids. For $j = 1, \dots, n$, we define $u_j$ to be the arc which:
	\begin{enumerate}[label=\arabic*.]
		\item Begins at $z_j$, and travels upward vertically until it reaches the line $\mathfrak{Im}(x) = 1$, then
		\item Travels to the right along a semi-circle, then
		\item Travels vertically downward to the line $\mathfrak{Im}(x) = -3$,
		\item Travels to the left along a semi-circle, and finally
		\item Travels vertically upward and terminates at the point $w_j$.
	\end{enumerate}
	
	\begin{defn}
		\label{def:bridge}
		We let $\widehat{F}$ be the sphere containing the above bridge presentation, and we denote this presentation of $m(K)$ by
		$$
		\mathcal{B}(\rho) = (\widehat{F}, u_1, \dots, u_n, o_1, \dots, o_n, z, w).
		$$
	\end{defn}
	
	Figure \ref{fig:bridge_diag} shows the diagram $\mathcal{B}(\rho)$ associated to the braid $\rho = \sigma_1^2 \in B_2$; then $\mathcal{B}(\rho)$ is a bridge presentation of the Hopf link, which happens to be amphichiral. (The points $p_L$ and $p_R$ in the figure are defined below.) In general, we leave it to the reader to check that if we take the mirror of the diagram $\mathcal{B}(\rho)$, then we obtain a diagram for $K$---the mirrored $o_j$ arcs correspond to the braid $\rho$, and the mirrored $u_j$ arcs take the closure of $\rho$ in $S^3$. Assuming this, it follows that $\mathcal{B}(\rho)$ is a diagram for $m(K)$. We can also think of $\mathcal{B}(\rho)$ as presenting $m(K)$ as the closure of a different braid---see Remark \ref{rmk:closure} below.
	
	\begin{conv}
		\label{conv:embedded}
		It will be helpful to think of the sphere $\widehat{F}$ as embedded in $S^3$, so that $\widehat{F} \cap m(K) = z \cup w$ and $\widehat{F}$ is a Heegaard surface for $m(K)$ in the sense of Section \ref{sec:hfk}, with the overbridges $o_j$ lying on one side of $\widehat{F}$ and the underbridges $u_j$ lying on the other. Additionally, we let $F$ denote the punctured sphere $F = \widehat{F} \setminus (z \cup w)$.
	\end{conv}
	
	\begin{rmk}
	\label{rmk:closure}
	The diagram $\mathcal{B}(\rho)$ can also be thought of as a braid closure: let $\iota : (D, z) \to (D, z)$ be a reflection of $D$ about the axis $\mathfrak{Re}(\zeta) = 0$, isotoped so that $z_j$ and $z_{n + 1 - j}$ swap places for all $j$. Let $\bar{\rho} \in B_n$ be the conjugation of $\rho$ by $\iota$. Then $\mathcal{B}(\rho)$ expresses $m(K)$ as the closure of $\bar{\rho}$. This can be seen by rotating $m(K)$ 180 degrees around the axis $\mathfrak{Re}(\zeta) = 0$, out of the plane of the diagram, so that underbridges become overbridges and vice versa.
	\end{rmk}
	
	Thinking of $m(\mathcal{B}(\rho))$ as the closure of $\rho$, or of $\mathcal{B}(\rho)$ as the closure of $\bar{\rho}$ as described in the preceeding remark, it makes sense to consider the braid axis of the closure. This axis will be useful in future sections.
	
	\begin{defn}
		\label{def:meridian}
		Let $U \subset S^3 \setminus m(K)$ be an unknot in the complement of $m(K)$, such that $U$ bounds a disk which intersects each overbridge once positively, and which is disjoint from all the underbridges. Equivalently, $U$ is an axis around which $m(K)$ is braided. We position $U$ so that $U \cap \widehat{F}$ consists of two points, $p_L$ and $p_R$, which lie at $-1 - i$ and $-1 + i$ respectively, as in Figure \ref{fig:bridge_diag}. 
	\end{defn}
	
	\begin{figure}
		\labellist
		\small\hair 2pt
		\pinlabel $z_1$ at 200 490
		\pinlabel $z_2$ at 285 490
		\pinlabel $w_1$ at 205 190
		\pinlabel $w_2$ at 280 190
		\pinlabel $o_1$ at 140 280
		\pinlabel $o_2$ at 345 280
		\pinlabel $u_1$ at 760 420
		\pinlabel $u_2$ at 620 420
		\pinlabel $p_L$ at -25 280
		\pinlabel $p_R$ at 520 280
		\endlabellist
		
		\centering
		\includegraphics[height=8cm]{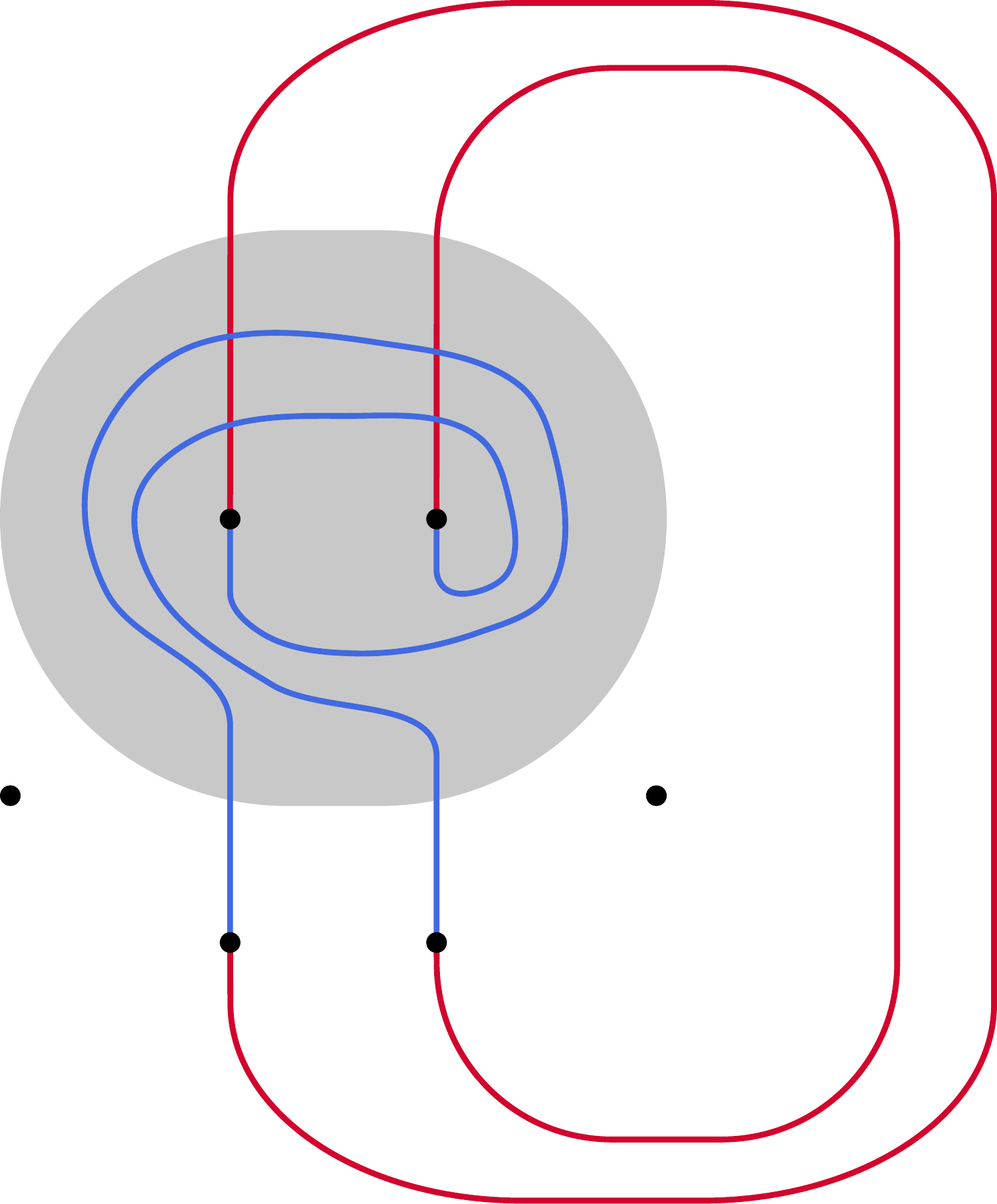}
		\caption{The bridge diagram $\mathcal{B}(\rho)$ for the braid $\rho = \sigma_1^2 \in B_2$}
		\label{fig:bridge_diag}
	\end{figure}
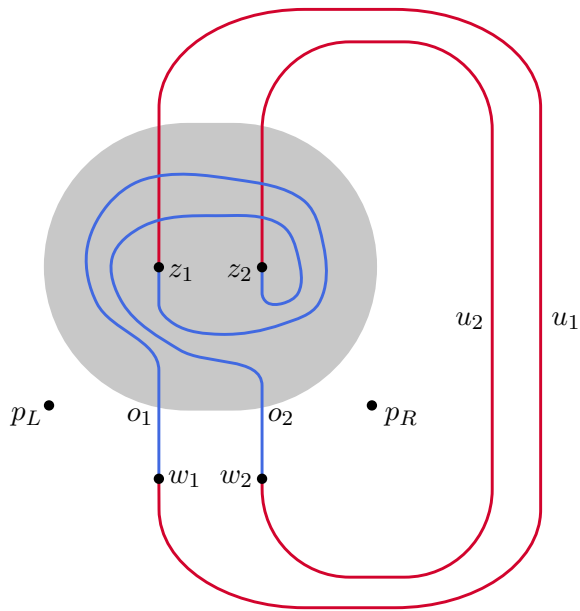

	\subsection{Heegaard diagrams}
	\label{sec:diags}
	
	As in Section \ref{sec:ub}, let
	$$
	\psi_n : B_n \to GL_n(\Z[t,t^{-1}])
	$$
	be the full Burau representation, and fix a braid $\rho \in B_n$. Let $m$ be an integer satisfying $1 \leq m \leq n$, and let ${\bf j}$ and ${\bf k}$ be multi-indices of size $m$. More specifically, let ${\bf j} = \{j_1, \dots, j_m\}$ be such that
	$$
	1 \leq j_1 < j_2 < \cdots < j_m \leq n,
	$$
	and define ${\bf k} = \{k_1, \dots, k_m\}$ analogously. Let $I_n$ be the $n$-by-$n$ identity matrix, let $\lambda$ be a variable, and let $A^{\bf jk}$ be the $m$-by-$m$ submatrix of $\psi_n(\rho) - \lambda I_n$ consisting of the intersection of the rows ${\bf j}$ with the columns ${\bf k}$. In this section we accomplish the first objective of the paper: we will define a Heegaard-Floer type homology theory that categorifies $\det(A^{\bf jk})$.
	
	Let $K \subset S^3$ be the braid closure of $\rho$, and let $\mathcal{B}(\rho)$ be the bridge presentation for $m(K)$ from Definition \ref{def:bridge}. We will use ${\bf j}$, ${\bf k}$ and $\mathcal{B}(\rho)$ to build a multi-pointed Heegaard diagram. Our method is a straightforward generalization of Rasmussen's original construction of Heegaard diagrams from bridge presentations \cite[Section 3.2]{ras03}---specifically, Rasmussen's construction corresponds to case where $m = n - 1$. The $m = n - 1$ case of our construction is also very similar to the Heegaard diagrams defined using braid closures in \cite[Section 3]{bvvv13} and \cite[Section 2.3]{lsvv13}, though these papers don't mention the Burau representation.
	
	We define the curves $\beta = \{\beta_{j_1}, \dots, \beta_{j_m}\}$ by letting $\beta_{j_s} \subset \widehat{F}$ be the boundary of a regular neighborhood of the overbridge $o_{j_s}$ for $s \in [m]$. Next, for $s \in [m]$, let $B_\epsilon(z_{k_s}), B_\epsilon(w_{k_s}) \subset \widehat{F}$ be open disks of radius $\epsilon$ centered at $z_{k_s}$ and $w_{k_s}$, with $\epsilon > 0$ small enough that the disks are disjoint from all $\beta$ curves. For each such $k_s$, we remove $B_\epsilon(z_{k_s})$ and $B_\epsilon(w_{k_s})$ from $\widehat{F}$ and attach a handle along the resulting boundaries in an orientation-preserving way. We then join the endpoints of the truncated underbridge $u_{k_s} \setminus (B_\epsilon(z_{k_s}) \cup B_\epsilon(w_{k_s}))$ via an arc through the handle, and let $\alpha_{k_s}$ be the resulting closed curve. We set $\alpha = \{\alpha_{k_1}, \dots, \alpha_{k_m}\}$, and we let $\Sigma$ denote the genus $m$ surface obtained from the handle attachments. As in Convention \ref{conv:embedded}, we think of the surface $\Sigma$ as embedded in $S^3$, with the handle attachments done in the complement of $m(K)$, so that
	\begin{equation}
		\label{eq:conv}
		\Sigma \cap m(K) = (z \setminus \{z_{k_1}, \dots, z_{k_m}\}) \cup (w \setminus \{w_{k_1}, \dots, w_{k_m}\}).
	\end{equation}
	
	We forget any overbridges $o_j$ such that $j \notin {\bf j}$, and any underbridges $u_k$ with $k \notin {\bf k}$. However, we keep the two sets of basepoints remaining after the handle attachments, $z \setminus \{z_{k_1}, \dots, z_{k_m}\}$ and $w \setminus \{w_{k_1}, \dots, w_{k_m}\}$. We also add the two additional basepoints
	$$
	\{p_L, p_R\} = U \cap \widehat{F} = U \cap \Sigma
	$$
	from Definition \ref{def:meridian}. 
	
	\begin{defn}
		\label{def:heediag}
		We denote the Heegaard diagram resulting from this construction by
		\begin{equation}
			\mathcal{H}(\rho, {\bf j}, {\bf k}) = (\Sigma, \alpha, \beta, (z \setminus \{z_{k_1}, \dots, z_{k_m}\}) \cup (w \setminus \{w_{k_1}, \dots, w_{k_m}\}) \cup \{p_L, p_R\}).
		\end{equation}
	\end{defn}
	
	See Figure \ref{fig:heediags} for two examples of Heegaard diagrams associated to the braid $\sigma_1^2$ in $B_2$, constructed using the bridge diagram of Figure \ref{fig:bridge_diag}.
	
	\begin{rmk}
		The reader has likely noticed that, in the above construction of $\mathcal{H}(\rho, {\bf j}, {\bf k})$, we've labelled the curve sets $\alpha$ and $\beta$, as well as the basepoints coming from $z$ and $w$, using non-consecutive indices. This will simplify notation in the next section.
	\end{rmk}
	
	We now define our homology theory.
	
	\begin{figure}
		\labellist
		\small\hair 2pt
		\pinlabel $\beta_1$ at 125 280
		\pinlabel $\beta_2$ at 370 280
		\pinlabel $\alpha_1$ at 760 420
		\pinlabel $\alpha_2$ at 620 420
		\pinlabel $p_L$ at -25 280
		\pinlabel $p_R$ at 520 280
		\pinlabel $p_L$ at 840 280
		\pinlabel $p_R$ at 1365 280
		\pinlabel $\alpha_2$ at 1535 420
		\pinlabel $\beta_2$ at 1225 280
		\pinlabel $w_1$ at 1062 190
		\endlabellist
		
		\centering
		\subcaptionbox{The diagram $\mathcal{H}(\rho, \{1,2\}, \{1,2\})$ \label{fig:hda}}{
			\includegraphics[height=8cm]{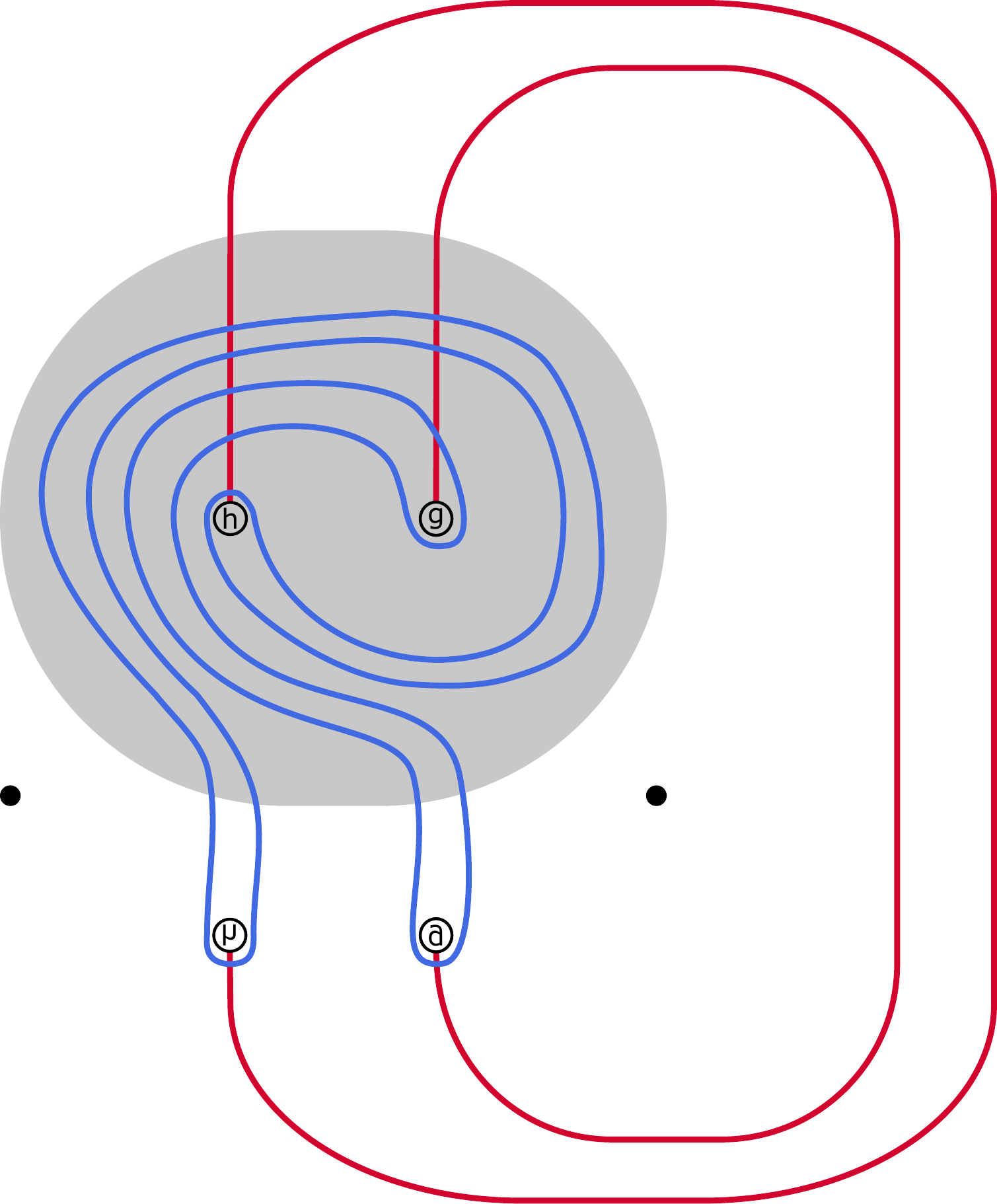}
		}
		\hspace{.75cm}
		\subcaptionbox{The diagram $\mathcal{H}(\rho, \{2\}, \{2\})$ \label{fig:hdb}}{
			\labellist
			\small\hair 2pt
			\pinlabel $z_1$ at 200 442
			\endlabellist
			
			\includegraphics[height=7cm]{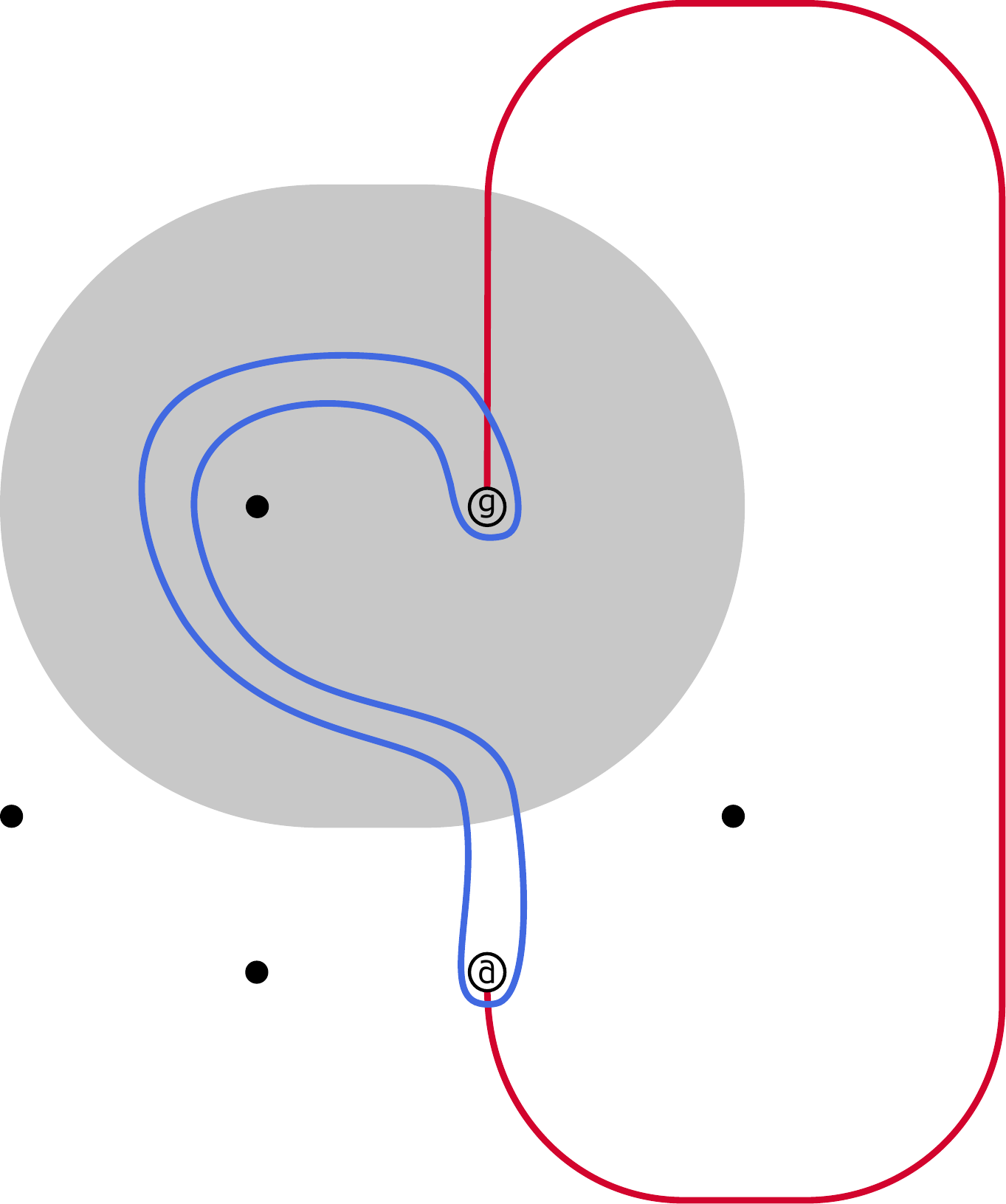}
			\vspace{.5cm}
		}
		\caption{Two Heegaard diagrams for the braid $\rho = \sigma_1^2 \in B_2$. The labelled circles indicate handle attachments (and compare Figure \ref{fig:bridge_diag}).}
		\label{fig:heediags}
	\end{figure}
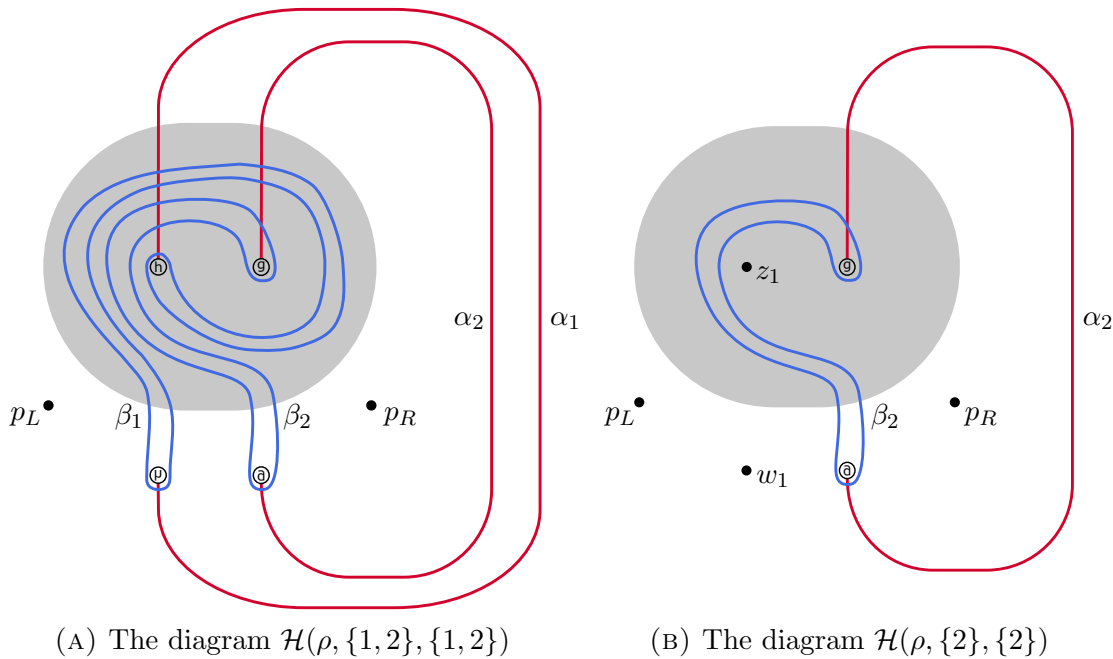
	
	\begin{defn}
		\label{def:grps}
		Given a braid $\rho \in B_n$ and multi-indices ${\bf j}$ and ${\bf k}$ as above, define a chain complex by
		$$
		\widehat{CFB}(\rho, {\bf j}, {\bf k}) = \widehat{CF}(\mathcal{H}(\rho, {\bf j}, {\bf k})),
		$$
		and the resulting homology groups by
		$$
		\widehat{HFB}(\rho, {\bf j}, {\bf k}) = \widehat{HF}(\mathcal{H}(\rho, {\bf j}, {\bf k})),
		$$
		where $\mathcal{H}(\rho, {\bf j}, {\bf k})$ is the multi-pointed Heegaard diagram of Definition \ref{def:heediag}.
	\end{defn}
	
	\subsection{Gradings, and the Euler characteristic}
	\label{sec:gradings}
	
	Since all Heegaard diagrams constructed in the preceeding section are (multi-pointed) diagrams for $S^3$, the Floer homology groups of Definition \ref{def:grps} admit a relative Maslov $\Z$-grading. Following the convention of knot Floer homology, we define the Maslov grading using the Maslov index, the $w$ basepoints, and the $p_R$ basepoint, as in \cite[Section 2]{mos09}. We omit a precise definition, but we emphasize that Lemma \ref{lem:maslov} holds in this context. Thus, the unfamiliar reader may consider the Maslov grading as a relative $\Z/2$-grading which depends only on orientations.
	
	The groups  of Definition \ref{def:grps} also admit relative Alexander gradings, which we now define. We will not use a separate grading for each link component; as in knot Floer theory, we use fewer gradings to obtain the desired Euler characteristic.
	
	Let $\rho \in B_n$ be a braid, $K \subset S^3$ its closure, and ${\bf j}, {\bf k} \subset [n]$ multi-indices of size $m$. Then we have the Heegaard diagram $\mathcal{H}(\rho, {\bf j}, {\bf k})$ of Definition \ref{def:heediag}, and we let $U \subset S^3 \setminus m(K)$ be the braid axis as in Definition \ref{def:meridian}. Following Convention \ref{conv:embedded}, we think of the Heegaard surface $\Sigma$ as embedded in $S^3$, with $\Sigma \cap m(K)$ as in (\ref{eq:conv}) and $\Sigma \cap U = \{p_L, p_R\}$.
	
	Let 
	$$
	h_1 : H_1(S^3 \setminus m(K)) \to \Z
	$$
	be the homomorphism which sends each positively oriented meridian of a component of $m(K)$ to $1$. Similarly, let
	$$
	h_2 : H_1(S^3 \setminus U) \to \Z
	$$
	be the homomorphism that sends a meridian of $U$ to $1$. As in (\ref{eq:diff_curve}), given two generators $x$ and $y$ of $\widehat{CFB}(\rho, {\bf j}, {\bf k})$, let $\gamma_{x,y} \subset \Sigma$ be a corresponding difference curve. Then we define two relative Alexander gradings $\widetilde{A}_1, \widetilde{A}_2$ on $\widehat{CFB}(\rho, {\bf j}, {\bf k})$ by
	\begin{equation}
		\label{eq:hom_hom}
		\widetilde{A}_1(x) - \widetilde{A}_1(y) = h_1([\gamma_{x,y}])
	\end{equation}
	and
	\begin{equation}
		\label{eq:mer_mer}
		\widetilde{A}_2(x) - \widetilde{A}_2(y) = h_2([\gamma_{x,y}]).
	\end{equation}
	Since each curve in $\alpha$ and $\beta$ is null-homologous in $S^3 \setminus K$ and $S^3 \setminus U$, the difference classes $[\gamma_{x,y}]$ in $H_1(S^3 \setminus K)$ and $H_1(S^3 \setminus U)$ do not depend on the choice of curve. Furthermore, by the proof of Lemma \ref{lem:alex_disc}, the gradings $\widetilde{A}_1$ and $\widetilde{A}_2$ obstruct differentials of $\widehat{CFB}(\rho, {\bf j}, {\bf k})$, and descend to relative gradings on $\widehat{HFB}(\rho, {\bf j}, {\bf k})$.
	
	The relative grading $\widetilde{A}_1$ can be canonically lifted to an absolute grading $A_1$ on $\widehat{CFB}(\rho, {\bf j}, {\bf k})$. To accomplish this, we first relate the Heegaard diagram $\mathcal{H}(\rho, {\bf j}, {\bf k})$ with the noodles and chopsticks appearing in Section \ref{sec:fn}. Let ${\bf D}_n$ denote the punctured disk $D \setminus z$, and let
	$$
	\mathring{\bf D}_n = {\bf D}_n \setminus \big(N(\partial D) \cup B_\epsilon(z_1) \cup \dots \cup B_\epsilon(z_n) \big),
	$$
	where $N(\partial D)$ is a collar neighborhood of $\partial D$, and $B_\epsilon(z_j)$ is a small open ball centered at the basepoint $z_j \in z$. In Section \ref{sec:diags}, we built the Heegaard surface $\Sigma$ of $\mathcal{H}(\rho, {\bf j}, {\bf k})$ by attaching handles to $\widehat{F} = \C \cup \{\infty\}$. Since the only regions of $D \subset \C$ removed for the handle attachments were the open balls $B_\epsilon(z_j)$ for certain $j$, we have a well-defined inclusion $\mathring{\bf D}_n \hookrightarrow \Sigma$.
	
	\begin{conv}
		\label{con:coords}
		We identify $\mathring{\bf D}_n$ with a subset of $\Sigma$ via the above inclusion map.
	\end{conv}
	
	Fix indices $j \in {\bf j}$ and $k \in {\bf k}$, and let $\eta_j, \varepsilon_k \subset {\bf D}_n$ be the $j$th standard noodle and $k$th standard chopstick of Definition \ref{def:fn}. Let $\rho(\eta_j) \subset {\bf D}_n$ be the image of $\eta_j$ under $\rho : {\bf D}_n \to {\bf D}_n$. Isotoping $\rho$ if necessary, we assume all intersection points of $\rho(\eta_j) \cap \varepsilon_k$ are contained in $\mathring{\bf D}_n$.
	
	\begin{lemma}
		\label{lem:relating}
		By applying isotopies to the arcs $\rho(\eta_j)$ and $\varepsilon_k$ within ${\bf D}_n$ if necessary, we can assume that
		\begin{align}
		\label{eq:rel_one}
		\rho(\eta_j) \cap \mathring{\bf D}_n &= \beta_j \cap \mathring{\bf D}_n \text{ and} \\
		\label{eq:rel_two}
		\varepsilon_k \cap \mathring{\bf D}_n &= \alpha_k \cap \mathring{\bf D}_n,
 		\end{align}
 		where $\beta_j \subset \Sigma$ and $\alpha_k \subset \Sigma$ are the relevant curves in $\mathcal{H}(\rho, {\bf j}, {\bf k})$.
	\end{lemma}

	\begin{proof}
		As Figure \ref{fig:cc} shows, (\ref{eq:rel_one}) and (\ref{eq:rel_two}) hold in the special case that $\rho$ is the identity braid. Since neither $\varepsilon_k$ nor $\alpha_k$ depends on $\rho$, (\ref{eq:rel_two}) always holds. Furthermore, let $\beta^{\bf 1}_j$ denote the curve $\beta_j$ in the case of the identity braid, and fix another braid $\rho$. Then we may assume that
		$$
		\eta_j \cap \mathring{\bf D}_n = \beta^{\bf 1}_j \cap \mathring{\bf D}_n,
		$$
		and by construction $\beta_j = \rho(\beta^{\bf 1}_j)$. It follows that
		$$
		\rho(\eta_j) \cap \mathring{\bf D}_n = \rho(\beta^{\bf 1}_j) \cap \mathring{\bf D}_n = \beta_j \cap \mathring{\bf D}_n,
		$$
		as desired.
	\end{proof}
	
	\begin{figure}
		\labellist
		\small\hair 2pt
		\pinlabel $\beta_1$ at 120 280
		\pinlabel $\beta_2$ at 370 280
		\pinlabel $\alpha_1$ at 760 420
		\pinlabel $\alpha_2$ at 620 420
		\pinlabel $\eta_1$ at -25 400
		\pinlabel $\eta_2$ at -25 345
		\pinlabel $\varepsilon_1$ at 130 600
		\pinlabel $\varepsilon_2$ at 280 600
		\endlabellist
	
		\includegraphics[height=8cm]{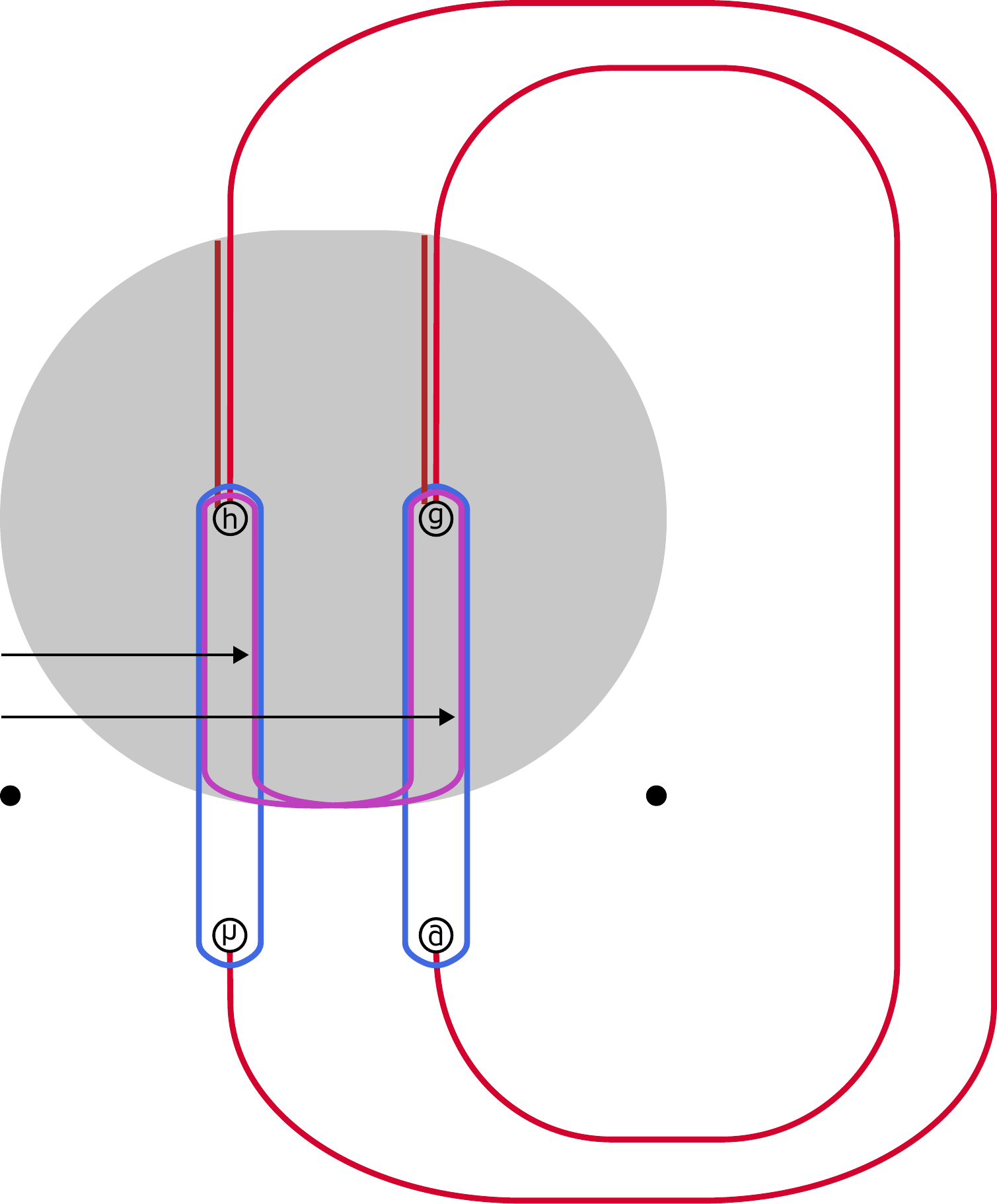}
		\caption{Comparing chopsticks and noodles with $\alpha$ and $\beta$ curves}
		\label{fig:cc}
	\end{figure}
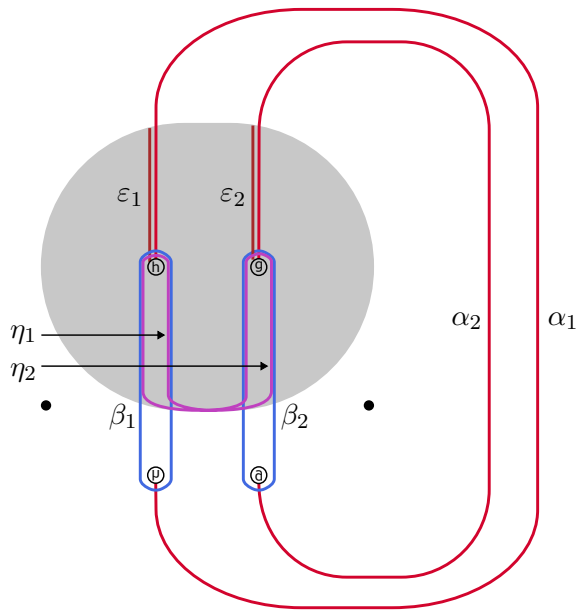
	
	Next, we use Lemma \ref{lem:relating} to define a local grading function
	$$
	A_1^\text{loc} : \{x \in \beta_j \cap \alpha_k \mid j,k \in [n]\} \to \Z.
	$$
	Fix indices $j \in {\bf j}$ and $k \in {\bf k}$, and a point $x \in \beta_j \cap \alpha_k$. Then exactly one of the following holds:
	\begin{enumerate}[label=(\roman*)]
		\item We have $j = k$, and $x$ is the unique intersection point of $\beta_j \cap \alpha_j$ which occurs in $\C \subset \widehat{F}$ directly below where the basepoint $w_j$ has been surgered out.
		\item The intersection point $x$ occurs in $\mathring{\bf D}_n$.
	\end{enumerate}
	This is certainly true if $\rho$ is the identity braid, so holds for all braids since $\rho$ fixes $\C \setminus D$. Case (i) will be particularly important to us, so we name it for future reference.
	
	\begin{defn}
		\label{def:anchor}
		If the multi-indices ${\bf j}, {\bf k} \subset [n]$ contain a common index $j$, then there is a unique intersection point $x \in \beta_j \cap \alpha_j$ occurring in $\Sigma$ near where the basepoint $w_j$ has been removed. We call $x$ the {\em $j$th anchor point}.
	\end{defn}
	
	In Figures \ref{fig:heediags} and \ref{fig:cc}, the anchors points are those intersection points not contained in the grey unit disk. If $x$ is an anchor point, then we set
	$$
	A_1^\text{loc}(x) = 0.
	$$
	Otherwise, if $x \in \beta_j \cap \alpha_k$ and case (ii) holds, then by Lemma \ref{lem:relating}, $x$ can be identified with an intersection point $x' \in \rho(\eta_j) \cap \varepsilon_k$. In this case, we define
	$$
	A_1^\text{loc}(x) = S(x'),
	$$
	where $S$ is the sheet function of Definition \ref{def:sheet}. Having specified $A_1^\text{loc}$ for all intersection points, we define $A_1$.
	
	\begin{defn}
		\label{def:grading}
		Given a generator $x = \{x_1, \dots, x_m\} \in \mathbb{T}_\alpha \cap \mathbb{T}_\beta$ of $\widehat{CFB}(\rho, {\bf j}, {\bf k})$, define an absolute grading $A_1 : \widehat{CFB}(\rho, {\bf j}, {\bf k}) \to \Z$ by
		$$
		A_1(x) = \sum_{j = 1}^m A_1^\text{loc}(x_j).
		$$
	\end{defn}

	\begin{lemma}
		\label{lem:grading}
		Let $x = \{x_1, \dots, x_m\}$ and $y = \{y_1, \dots, y_m\}$ be two generators of $\widehat{CFB}(\rho, {\bf j}, {\bf k})$, and let $\gamma_{x,y}$ be a difference curve for $x$ and $y$. Then the function $A_1$ satisfies
		\begin{equation}
			\label{eq:diff}
			A_1(x) - A_1(y) = h_1([\gamma_{x,y}]).
		\end{equation}
		Equivalently, $A_1$ lifts $\widetilde{A}_1$ to an absolute grading on $\widehat{CFB}(\rho, {\bf j}, {\bf k})$.
	\end{lemma}

	\begin{proof}
		Let $x$ and $y$ be as above. Since each curve in $\alpha$ and $\beta$ is null-homologous in $S^3 \setminus K$, it suffices to check (\ref{eq:diff}) for any choice of difference curve $\gamma_{x,y}$. We will choose sub-arcs $a_1, \dots, a_m$ of the $\alpha$ curves and sub-arcs $b_1, \dots, b_m$ of the $\beta$ curves, as in (\ref{eq:diff_curve}), to construct a particular difference curve
		$$
		\gamma_{x,y} = a_1 \cup b_1 \cup \dots \cup a_m \cup b_m.
		$$
		
		Suppose $x_j$ and $y_k$ are respective elements of $x$ and $y$ which are contained in the same curve $\alpha_s \in \alpha$. Then we must choose an arc $a_s \subset \alpha_s$ with $\partial a_s = \{x_j, y_k\}$, and there are two possible choices: one which passes through the handle of $\Sigma$ attached at $w_s$ and $z_s$, and one which does not. For all such $\alpha_s \in \alpha$, we choose $a_s \subset \alpha_s$ to be the sub-arc of $\alpha_s$ which avoids the handle attachment.
		
		Likewise, if $x_j$ and $y_k$ are both contained in the curve $\beta_s$, then we must choose an arc $b_s \subset \beta_s$ with $\partial b_s = \{x_j, y_k\}$, and again we have two choices. We make the choice as follows. If neither $x_j$ nor $y_k$ is an anchor point, then both points are contained in $\mathring{\bf D}_n \subset \Sigma$, and we choose the arc $b_s \subset \beta_s$ which is also contained in $\mathring{\bf D}_n$. (Such an arc exists since $ \beta_s \setminus \mathring{\bf D}_n$ is connected, so $\beta_s \cap \mathring{\bf D}_n$ is connected as well.) If either $x_j$ or $y_k$ is an anchor point, then we let $b_s \subset \beta_s$ be the arc which occupies the space immediately to the left of the anchor point on $\beta_s$, with our notion of ``left'' coming from the coordinates on $\C$ used in Section \ref{sec:bridges}. (If $x_j = y_k$, then $b_s$ is just the anochor point $\{x_j\}$.) Having chosen all necessary arcs, let $\gamma_{x,y}$ be the resulting difference curve. We must show $h_1(\gamma_{x,y})$ satisfies (\ref{eq:diff}).
		
		Since the $a_s$ sub-arcs of $\gamma_{x,y}$ were chosen to avoid handle attachments, and the $\beta$ curves always avoid handles, the curve $\gamma_{x,y}$ lies in punctured sphere $F \subset S^3$ used to construct $\Sigma$ (see Convention \ref{conv:embedded}). For each puncture $p \in z \cup w$ of $F$, let $\delta(p) \subset F$ be the closed curve
		$$
		\delta(p) = \partial B_\epsilon(p),
		$$
		oriented to be homologous to a meridian of $K$ in $S^3$. Then the curves $\delta(p)$ span $H_1(F)$, and the homomorphism $h_1 : H_1(S^3 \setminus K) \to \Z$ of (\ref{eq:hom_hom}), when restricted to $F$, is determined by 
		$$
		h_1(\delta(p)) = 1
		$$
		for all $p \in z \cup w$.
		
		For $k \in [n]$, let $\bar{\ell}_k \subset F$ be a properly embedded, oriented arc which travels parallel to and to the right of the underbridge $u_k$, going from $z_k$ to $w_k$. As in the proof of Lemma \ref{lem:relating}, we can assume that
		$$
		\bar{\ell}_k \cap \mathring{\bf D}_n = \ell_k \cap \mathring{\bf D}_n
		$$
		for all $k \in [n]$, where $\ell_k \subset {\bf D}_n$ is the $k$th cut arc of Definition \ref{def:fn}. Additionally, it's straightforward to check that
		\begin{equation}
			\label{eq:hone}
			h_1(e) = \sum_{k = 1}^n e \cdot \bar{\ell}_k
		\end{equation}
		for any class $e \in H_1(F)$, for example by checking (\ref{eq:hone}) on the spanning set $\delta(p), p \in z \cup w$. Thus, the proof is completed by showing that
		$$
		A_1(x) - A_1(y) = \sum_{r = 1}^n \gamma_{x,y} \cdot \bar{\ell}_r.
		$$
		
		Since $\alpha_k \cap F = u_k \cap F$ for all $k \in {\bf k}$, and $a_k \subset \alpha_k$, we have $a_k \cap \bar{\ell}_r = \varnothing$ for all $k \in {\bf k}$ and $r \in [n]$. Additionally, by our choice of arcs $b_j \subset \beta_j$ in the construction of $\gamma_{x,y}$, all intersections $b_j \cap \bar{\ell}_r$ occur within $\mathring{\bf D}_n$, so $b_j \cdot \bar{\ell}_r = b_j \cdot \ell_r$ for all $j \in {\bf j}$ and $r \in [n]$. Therefore,
		$$
		h_1([\gamma_{x,y}]) = \sum_{r = 1}^n \gamma_{x,y} \cdot \bar{\ell}_r = \sum_{j \in {\bf j}}\sum_{r = 1}^n b_j \cdot \bar{\ell}_r = \sum_{j \in {\bf j}} \sum_{r = 1}^n b_j \cdot \ell_r.
		$$
		Suppose, for notational simplicity, that the arc $b_j \subset \beta_j$ has boundary points $x_j \in x$ and $y_j \in y$ for all $j \in {\bf j}$. Then, by Lemma \ref{lem:s_intersect}, for all $j \in {\bf j}$,
		$$
		\sum_{r = 1}^n b_j \cdot \ell_r = \begin{cases}
			S(x_j) - S(y_j) \text{ if neither $x_j$ nor $y_j$ is an anchor point.} \\
			S(x_j) \text{ if only $y_j$ is an anchor point.} \\
			-S(y_j) \text{ if only $x_j$ is an anchor point.} \\
			0 \text{ if $x_j = y_j$ is an anchor point.}
		\end{cases}
		$$
		It follows from the definition of $A_1^\text{loc}$ that
		$$
		h_1([\gamma_{x,y}])  = \sum_{j \in {\bf j}} \sum_{r= 1}^n b_j \cdot \ell_r = \sum_{j \in {\bf j}} A_1^\text{loc}(x_j) - A_1^\text{loc}(y_j) = A_1(x) - A_1(y),
		$$
		completing the proof.
	\end{proof}
	
	We can also lift $\widetilde{A}_2$ to an absolute grading $A_2$ on $\widehat{CFB}(\rho, {\bf j}, {\bf k})$. This is easier than in the $A_1$ case.	
	\begin{defn}
		\label{def:grading_two}
		Given a generator $x = \{x_1, \dots, x_m\} \in \mathbb{T}_\alpha \cap \mathbb{T}_\beta$ of $\widehat{CFB}(\rho, {\bf j}, {\bf k})$, define an absolute grading $A_2 : \widehat{CFB}(\mathcal{H}(\rho, {\bf j}, {\bf k})) \to \Z$ by
		$$
		A_2(x) = |\{x_j \in x \mid x_j \text{ is an anchor point}\}|.
		$$
	\end{defn}
	
	The proof that $A_2$ lifts $\widetilde{A}_2$ is similar to Lemma \ref{lem:grading}.
	
	\begin{lemma}
		\label{lem:grading_two}
		Let $x = \{x_1, \dots, x_m\}$ and $y = \{y_1, \dots, y_m\}$ be two generators of $CF(\rho, {\bf j}, {\bf k})$, and let $\gamma_{x,y}$ be a difference curve for $x$ and $y$. Then the function $A_2$ defined above satisfies
		\begin{equation}
			\label{eq:diff_two}
			A_2(x) - A_2(y) = h_2([\gamma_{x,y}]).
		\end{equation}
		Equivalently, $A_2$ lifts the relative grading $\widetilde{A}_2$ of (\ref{eq:mer_mer}).
	\end{lemma}
	
	\begin{proof}
		Given $x$ and $y$ as in the lemma, we choose the difference curve $\gamma_{x,y}$ by the same process as in the proof of Lemma \ref{lem:grading}. Then we have
		$$
		\gamma_{x,y} = a_1 \cup b_1 \cup \cdots \cup a_m \cup b_m,
		$$
		where $a_j \subset \alpha_j$ and $b_j \subset \beta_j$ for all $j$.
		
		Similar to the proof of Lemma \ref{lem:grading}, we consider the difference curve $\gamma_{x,y}$ as lying in the punctured sphere
		$$
		\widehat{F} \setminus \{p_L, p_R\} \subset (S^3 \setminus \mu).
		$$
		Let $\ell$ be a properly embedded, horizontal arc in $\C \setminus \{p_L, p_R\} \subset \widehat{F} \setminus \{p_L, p_R\}$ which travels right from $p_L$ to $p_R$. Then it is not difficult to show that
		$$
		h_2([\gamma_{x,y}]) = \gamma_{x,y} \cdot \ell,
		$$
		and it remains to verify that
		$$
		A_2(x) - A_2(y) = \gamma_{x,y} \cdot \ell.
		$$
		Now, all curves in $\alpha$ avoid the arc $\ell$, so
		$$
		\gamma_{x,y} \cdot \ell = (b_1 \cup \cdots \cup b_m) \cdot \ell.
		$$
		Furthermore, by our construction of $\gamma_{x,y}$, an arc $b_j$ intersects $\ell$---one time---if and only if one endpoint of $b_j$ is an anchor point and the other is not. Taking orientations into account, we have
		$$
		\gamma_{x,y} \cdot \ell = (\text{$\#$ of anchor points in $x$}) - (\text{$\#$ of anchor points in $y$}) = A_2(x) - A_2(y).
		$$
	\end{proof}
	
	\begin{rmk}
		\label{rmk:local_two}
		Like the grading $A_1$, the grading $A_2$ is induced by a local grading
		$$
		A_2^\text{loc} : \{x \in \beta_j \cap \alpha_k \mid j,k \in [n]\} \to \Z
		$$
		in the sense of Definition \ref{def:grading}. From Definition \ref{def:grading_two}, $A_2^\text{loc}$ is evidently defined by
		$$
		A_2^\text{loc}(x) = \begin{cases}
			1 & \text{$x$ is an anchor point} \\
			0 & \text{otherwise}
		\end{cases}.
		$$
	\end{rmk}
	
	Lemmas \ref{lem:grading} and \ref{lem:grading_two} allow us to easily compute the Euler characteristic of $\widehat{HFB}(\rho, {\bf j}, {\bf k})$. Recall that $A^{\bf jk} = (A^{\bf jk}_{jk})$ denotes the minor of $\psi_n(\rho) - \lambda I_n$ determined by the multi-indices ${\bf j}$ and ${\bf k}$. To simplify notation in what follows, we index the rows and columns of $A^{\bf jk}$ (possibly non-consecutively) by the same indices used in the full matrix $\psi_n(\rho) - \lambda I_n$.
	
	\begin{thm}
		Let $\rho$ and  $A^{\bf jk}$ be as above, and fix a choice of absolute Maslov grading. for $j,k,r \in \Z$, let $\widehat{HFB}_{j,k,r}(\rho, {\bf j}, {\bf k})$ denote the subgroup of $\widehat{HFB}(\rho, {\bf j}, {\bf k})$ supported in $A_1$ grading $j$, $A_2$ grading $k$, and Maslov grading $r$. Then, for an appropriate choice of absolute Maslov grading, we have
		$$
		\sum_{j,k,r \in \Z} (-1)^{j} t^k \lambda^r \text{rk}(\widehat{HFB}_{j,k,r}(\rho, {\bf j}, {\bf k})) = \det(A^{\bf jk}).
		$$
	\end{thm}

	\begin{proof}
		The relationship between Heegaard Floer theories and determinants is well known---see, for example, \cite[Proposition 3.3]{ras03}. In our case, applying Lemma \ref{lem:reform}, we have
		\begin{align*}
			\det(A^{\bf jk}) &= \sum_{\sigma \in S_m} \text{sgn}(\sigma)\prod_{s = 1}^m A^{\bf jk}_{j_s k_{\sigma(s)}} \\
			&= \sum_{\sigma \in S_m} \text{sgn}(\sigma)\prod_{s = 1}^m \big( \psi_n(\rho)_{j_s k_{\sigma(s)}} - \lambda\delta_{j_sk_{\sigma(s)}} \big) \\
			&= \sum_{\sigma \in S_m} \text{sgn}(\sigma) \prod_{s = 1}^m \big( \sum_{x \in \rho(\varepsilon_{j_s}) \cap \eta_{k_{\sigma(s)}}} \text{sgn}(x)t^{S(x)}- \lambda\delta_{j_sk_{\sigma(s)}} \big).
		\end{align*}
		The symbol $\delta_{jk}$ here denotes the Kronecker delta. By the discussion preceeding Definitions \ref{def:anchor} and \ref{def:grading}, there is a bijection between the factors in the above product and the intersection points of $\alpha_{j_s}$ and $\beta_{k_{\sigma(s)}}$, where the Kronecker delta corresponds to the $j_s$ anchor point. Using the definitions of $A^\text{loc}_1$ and $A^\text{loc}_2$ (see Remark \ref{rmk:local_two}), we can write:
		$$
		\det(A^{\bf jk}) = \sum_{\sigma \in S_m} \text{sgn}(\sigma) \prod_{s = 1}^m \sum_{x \in \alpha_{j_s} \cap \beta_{k_{\sigma(s)}}} \text{sgn}(x)t^{A_1^\text{loc}(x)}\lambda^{A_2^\text{loc}(x)}.
		$$
		Rearranging the above equation, we find that
		\begin{align*}
				\det(A^{\bf jk}) &= \sum_{\sigma \in S_m} \sum_{\{(x_1, \dots, x_m) \mid x_s \in \alpha_{j_s} \cap \beta_{k_{\sigma(s)}}\}} \big(\text{sgn}(\sigma) \prod_{s = 1}^m \text{sgn}(x_s) \big) t^{\sum_{s = 1}^m A_1^\text{loc}(x_s)} \lambda^{\sum_{s = 1}^m A_2^\text{loc}(x_s)} \\
			&= \sum_{x \in \mathbb{T}_\alpha \cap \mathbb{T}_\beta} \text{sgn}(x) t^{A_1(x)}\lambda^{A_2(x)}.
		\end{align*}
		Finally, applying Lemma \ref{lem:maslov}, we can choose an absolute Maslov grading $M$ such that
		$$
		\det(A^{\bf jk}) = \sum_{x \in \mathbb{T}_\alpha \cap \mathbb{T}_\beta}(-1)^{M(x)} t^{A_1(x)}\lambda^{A_2(x)} = \sum_{j,k,r} (-1)^{j} t^k \lambda^r \text{rk}(\widehat{HFB}_{j,k,r}(\rho, {\bf j}, {\bf k})).
		$$
	\end{proof}

	\subsection{Invariance}
	
	We have:
	
	\begin{thm}
		For any $\rho \in B_n$ and multi-indices ${\bf j, k} \subset \{1, \dots, n\}$, the graded homology groups $\widehat{HFB}(\rho, {\bf j, k})$ are an invariant of $\rho$, ${\bf j}$ and ${\bf k}$.
	\end{thm}
	
	\begin{proof}
		Since the Heegaard diagram $\mathcal{H}(\rho, {\bf j,k})$ is defined up to isotopy of the $\alpha$ and $\beta$ curves, it suffices to check invariance of $\widehat{HFB}(\rho, {\bf j, k})$ under an isotopy of some curve, say $\alpha_j \in \alpha$, which fixes the $\beta$ curves and the other curves in $\alpha$. Invariance under this kind of isotopy follows just as it does in the cases of knot Floer homology and Heegaard Floer homology---we refer the reader to \cite[Proposition 5.1]{ras03} or \cite[Theorem 7.3]{os04c}.
	\end{proof}
	
	\section{Properties and Extensions of $\widehat{HFB}$}
	\label{sec:propss}
	
	\subsection{Essential properties}
	\label{sec:props}
	
	Let $\rho \in B_n$ be a braid with closure $K \subset S^3$. Then there are
	$$
	\sum_{m = 1}^n \binom{n}{m}^2 = \binom{2n}{n} - 1
	$$
	possible (non-empty) choices of multi-indices ${\bf j, k} \subset [n]$, and as many homology theories associated to $\rho$. We emphasize three special cases:
	\begin{enumerate}
		\item If ${\bf j} = {\bf k} = [n]$, then the Euler characteristic of the resulting group $\widehat{HFB}(\rho, [n], [n])$ is
		$$
		\det(\psi_n(\rho) - \lambda I_n),
		$$
		and the Heegaard diagram $\mathcal{H}(\rho, [n], [n])$ (see Definition \ref{def:heediag}) is also a Heegaard diagram for the link $m(K) \cup U \subset S^3$, where $U$ is the braid axis of Definition \ref{def:meridian}. It follows that if $m(K)$ is a knot, then $\widehat{HFB}(\rho, [n], [n])$ is grading-preserving isomorphic to the link Floer homology group $\widehat{HFL}(m(K) \cup U)$. If $m(K)$ has more than one component, then there is still a group isomorphism
		$$
		\widehat{HFL}(m(K) \cup U) \to \widehat{HFB}(\rho, [n], [n])
		$$
		induced by the common Heegaard diagram $\mathcal{H}(\rho, [n], [n])$. This isomorphism preserves the Maslov grading and the grading $A_2$, but collapses the gradings associated with the components of $m(K)$ to the single grading $A_1$. The equivalence of $\widehat{HFB}(\rho, [n], [n])$ with $\widehat{HFL}(m(K) \cup U)$, combined with the fact that $\widehat{HFB}(\rho, [n], [n])$ categorifies $\det(\psi_n(\rho)) - \lambda I_n)$, yields a new proof of a theorem of Morton \cite{mor99}: that the multi-variable Alexander polynomial $\Delta_{K \cup U}$ is equal to $\det(\psi_n(\rho)) - \lambda I_n)$, after some variable substitutions.

		\item If $|{\bf j}| = |{\bf k}| = n - 1$, then $\mathcal{H}(\rho, {\bf j}, {\bf k})$ is closely related to the knot Floer homology $\widehat{HFK}(m(K))$. To make this precise, let  $\mathcal{H}(\rho, {\bf j}, {\bf k})$ be as in Definition \ref{def:heediag}, and let $\mathcal{H}'$ be the diagram resulting from removing the basepoints $p_L$ and $p_R$ from $\mathcal{H}(\rho, {\bf j}, {\bf k})$. Then $\mathcal{H}'$ is a Heegaard diagram for $m(K)$.
		
		It is well known in Heegaard Floer theory that erasing basepoints from a Heegaard diagram induces a spectral sequence---see, for example, \cite{os08}. Here, we obtain:
		
		\begin{prop}
			There is a spectral sequence whose $E_1$ term is $\widehat{HFB}(\rho, {\bf j}, {\bf k})$, and whose $E^\infty$ term is isomorphic to $\widehat{HFK}(m(K))$, resulting from forgetting the basepoints $p_L$ and $p_R$. Each chain complex involved splits along the $A_1$ grading on \break $\widehat{HFB}(\rho, {\bf j}, {\bf k})$, which becomes the Alexander grading on $\widehat{HFK}(m(K))$.
		\end{prop}
		
		At the decategorified level, the Euler characteristic of $\widehat{HFB}(\rho, {\bf j}, {\bf k})$ is the determinant of a codimension-one submatrix of $\psi_n(\rho) - \lambda I_n$. Substituting $\lambda = 1$, this quantity coincides the Alexander polynomial of $m(K)$, which is also the Alexander polynomial of $K$.
		
		\item If $|{\bf j}| = |{\bf k}| = 1$, so that ${\bf j} = \{j\}$ and ${\bf k} = \{k\}$, then $\widehat{HFB}(\rho, {\bf j}, {\bf k})$ categorifies the matrix element $(\psi_n(\rho) - \lambda I_n)_{jk}$. In this case it's easy to compute the groups $\widehat{HFB}(\rho, {\bf j}, {\bf k})$: they simply count intersections between the curves $\alpha_j$ and $\beta_k$ in $\mathcal{H}(\rho, {\bf j}, {\bf k})$ when the two curves are in minimal position with respect to each other. This is very similar to the theory defined in \cite{bou10}.
	\end{enumerate}
	
	The collection of groups $\widehat{HFB}$ detects the trivial braid in at least two ways: at the ``top level'' ${\bf j} = {\bf k} = [n]$, and at the ``bottom level'' $|{\bf j}| = |{\bf k}| = 1$. The first of these is essentially a result of Baldwin and Grigsby.
	
	\begin{thm}[{\cite[Theorem 1b]{bagr15}}]
		\label{thm:bagr}
		Let $\rho \in B_n$ be an $n$-strand braid, and let $\mathbf{1}_n \in B_n$ be the trivial braid. Then $\rho$ is trivial if and only if $\psi_n(\rho)|_{t = 1}$ is the identity matrix, and 
		\begin{equation}
			\label{eq:triv_hom}
			\widehat{HFB}(\rho, [n], [n]) \cong \widehat{HFB}(\mathbf{1}_n, [n], [n]).
		\end{equation}
	\end{thm}
	
	\begin{proof}
		As discussed above,
		$$
		\widehat{HFB}(\rho, [n], [n]) \cong \widehat{HFK}(m(K) \cup U)
		$$
		by an isomorphism which preserves the Maslov grading and collapses the Alexander gradings $A_1$ and $A_2$ to a single grading. It follows from the cited result of Baldwin and Grigsby that $\rho$ is trivial if and only if (\ref{eq:triv_hom}) holds and $\rho$ is a {\em pure braid}, i.e.~$\rho$ is in the kernel of the homomorphism $B_n \to S_n$ which maps each braid to its induced permutation on the punctures of ${\bf D}_n$.\footnote{Although Baldwin and Grigsby's result is stated for link Floer homology, the proof goes through equally well for knot Floer homology. The invariant $\widehat{t}$ which they use is defined for knot Floer homology in \cite{bvvv13}.} This latter condition occurs if and only if $\psi_n(\rho)|_{t = 1}$ is the identity matrix, since at $t = 1$ the Burau representation recovers the above homomorphism $B_n \to S_n$.
	\end{proof}
	
	The proof that $\widehat{HFB}$ detects the trivial braid at the ``bottom level'' is classical and direct.
	
	\begin{thm}[cf.~{\cite[Theorem 13]{bou10}}]
		Let $\rho \in B_n$. Then $\rho$ is the trivial braid if and only if $\widehat{HFB}(\rho, \{j\}, \{k\}) = 0$ for all $j \neq k$.
	\end{thm}
	
	\begin{proof}
		It is easy to check that the latter hypothesis is satisfied if $\rho$ is the trivial braid. Conversely, suppose $\rho \in B_n$ satisfies $\widehat{HFB}(\rho, \{j\}, \{k\}) = 0$ for all $j \neq k$. Fix curves $\alpha_j$ and $\beta_k$ for distinct $j$ and $k$, and let
		$$
		\mathcal{H}(\rho, \{j\}, \{k\}) = (\Sigma, \{\alpha_j\}, \{\beta_k\}, z \cup w \setminus \{z_j, w_j\} \cup \{p_L, p_R\})
		$$
		be the corresponding Heegaard diagram. Note that $\Sigma$ is a torus; more specifically, following Section \ref{sec:diags}, $\Sigma$ is the sphere $\widehat{F}$ with one handle attached. We use this decomposition below.
		
		We will show that $\alpha_j$ and $\beta_k$ can be made disjoint by an isotopy supported within $\mathring{\bf D}_n$ (see Convention \ref{conv:embedded}). If $\alpha_j$ and $\beta_k$ are already disjoint, then we're done. If not, then as discussed above, $\alpha_j$ and $\beta_k$ are not in minimal position. It follows that $\alpha_j$ and $\beta_k$ cobound a bigon $B \subset \Sigma$, such that $B$ avoids the basepoints of $\Sigma$ \cite[Proposition 3.7]{paro99}.
		
		Let $a \subset \alpha_j$ and $b \subset \beta_k$ be the arcs of $\partial B$, and let $x_1, x_2 \in \alpha_j \cap \beta_k$ be their mutual endpoints. We claim that neither $x_1$ nor $x_2$ is an anchor point: indeed, if $x_1$ is the lone anchor point, then $x_2$ is not, so $A_2(x_1) = 1$ and $A_2(x_2) = 0$. Since $\partial B$ is a difference curve for $x_1$ and $x_2$, this implies $\partial B$ is not null-homologous in $\Sigma \setminus \{p_L, p_R\}$, so $\partial B$ cannot bound a bigon in this case.
		
		If follows that $x_1$ and $x_2$ both lie in $\mathring{\bf D}_n \subset \Sigma$. We claim the arc $a$ lies in $\mathring{\bf D}_n \subset \Sigma$ as well. This follows from the fact that $\alpha_j \setminus \mathring{\bf D}_n$ has one connected component, so if $a$ is not contained in $\mathring{\bf D}_n$, then $a$ must contain all of $\alpha_j \setminus \mathring{\bf D}_n$. Then, since $a \cup b$ bounds the bigon $B$, $\alpha_j \setminus \mathring{\bf D}_n$ and $\partial  \mathring{\bf D}_n$ cobound a sub-bigon of $B$, which implies $\alpha$ can be isotoped to lie in $\mathring{\bf D}_n$. But this is absurd, since $\alpha$ traverses the attached handle.
		
		The same logic shows that if $b$ does not lie in $\mathring{\bf D}_n$, then $\beta_k$ can be isotoped $\mathring{\bf D}_n$ by an isotopy that avoids basepoints. This, too, is absurd, since the component of $\beta_k$ which does not lie in $\mathring{\bf D}_n$ borders a foot of the handle on one side and the basepoints $p_L$ and $p_R$ on the other. Thus $a \cup b = \partial B$ is contained in $\mathring{\bf D}_n$, and since $B$ avoids the basepoints $p_L$ and $p_R$, the whole bigon $B$ is contained in $\mathring{\bf D}_n$. By isotoping the arc $a$ through $B$, we can remove the intersection points $x_1$ and $x_2$ by an isotopy supported in $\mathring{\bf D}_n$. Repeating this process inductively proves our original claim: that $\alpha_j$ and $\beta_k$ can be made disjoint by an isotopy supported within $\mathring{\bf D}_n$.
		
		Appealing to Lemma \ref{lem:relating}, we've shown that the chopstick $\varepsilon_j$ and the noodle $\rho(\eta_k)$ can be made disjoint by isotopy in ${\bf D}_n$. In fact, by hypothesis, $\rho(\eta_k)$ can be made disjoint from $\varepsilon_j$ for all $j \neq k$, by isotopies in ${\bf D}_n$. Then we can assume $\rho(\eta_k)$ lies in the annulus
		$$
		{\bf D}_n \setminus \bigcup \{\varepsilon_j \mid j \neq k\},
		$$
		and is in fact (freely) isotopic to the core of this annulus as a closed curve, since an annulus has only one isotopy class of simple closed curve. It follows that $\rho(\eta_k)$ is isotopic to $\eta_k$ in ${\bf D}_n$ for all $k$, so $\rho$ acts trivially on $\pi_1({\bf D}_n)$. Since the higher homotopy groups of ${\bf D}_n$ are trivial, this implies $\rho$ is the identity braid.
	\end{proof}
	
	\subsection{Generalizations}
	
	Theorem \ref{thm:main_burau} can be extended in a couple directions. First, rather than the Burau representation, one could consider the Gassner representation of the pure braid group,
	$$
	\varphi_n : P_n \to GL_n(\Z[t_1^{\pm 1}, t_2^{\pm 1}, \dots, t_n^{\pm 1}]).
	$$
	Here $P_n$ is the {\em pure braid group on $n$ strands}, which is the kernel of the homomorphism $B_n \to S_n$ discussed in the proof of Theorem \ref{thm:bagr}. The Gassner representation admits a chopstick/noodle interpretation analogous to Lemma \ref{lem:reform}. Briefly, using the notation of Lemma \ref{lem:s_intersect}, let $\rho \in P_n$ be a pure braid, let $y$ be a point of $\rho(\eta_j) \cap \varepsilon_k$, and let $\gamma_y$ be a ``difference arc'' for $y$. Then we define $n$ sheet functions
	$$
	S_1, \dots, S_n : (\rho(\eta_j) \cap \varepsilon_k) \to \Z
	$$
	by
	$$
	S_r(y) = \gamma_y \cdot \ell_r,
	$$
	where $\ell_r$ is the $r$th cut arc. The (unreduced) Gassner representation can then be defined element-wise by
	$$
	(\varphi_n(\rho))_{jk} = \sum_{y \in (\rho(\eta_j) \cap \varepsilon_k)} (-1)^{\text{sgn}(y)}t_1^{S_1(y)}t_2^{S_2(y)} \cdots t_n^{S_n(y)}.
	$$
	We leave this derivation to the reader, and defer to \cite[Section 3.2]{bir74} for a more standard definition of the representation.
	
	With staightforward generalizations of the arguments in Section \ref{sec:hfs}, we obtain:
	
	\begin{thm}
		\label{thm:gassner}
		Let $\rho \in P_n$ be a pure braid, and let ${\bf j, k} \subset \{1, \dots, n\}$ be multi-indices of size $m$. Let $\lambda$ be indeterminate, and let $A$ be the square submatrix of $\varphi_n(\rho) - \lambda I_n$ which is the intersection of the rows ${\bf j}$ and the columns ${\bf k}$. Then there exists a $\Z^{n + 1}$-graded Heegaard Floer homology theory $\widehat{HFP}(\rho, {\bf j}, {\bf k})$ which is an invariant of $\rho$, and whose Poincar\'e polynomial is $\det(A)$.
	\end{thm}
	
	For any pure braid $\rho \in P_n$, the group $\widehat{HFP}(\rho, {\bf j}, {\bf k})$ is defined using the same Heegaard diagram $\mathcal{H}(\rho, {\bf j}, {\bf k})$ as in the construction of $\widehat{HFB}(\rho, {\bf j}, {\bf k})$. The only difference is that the grading
	$$
	A_1 : \widehat{CFB}(\rho, {\bf j}, {\bf k}) \to \Z
	$$
	is replaced by a multi-variable grading
	$$
	A : \widehat{CFB}(\rho, {\bf j}, {\bf k}) \to \Z^n
	$$
	defined using the Gassner representation. In terms of difference classes, for generators $x$ and $y$,
	$$
	A(x) - A(y) = [\gamma_{x,y}] \in H_1(S^3 - m(K)) \cong \Z^n,
	$$
	where $K$ is the $n$-component link which is the closure of $\rho$.
	
	A second, broader way to generalize Theorem \ref{thm:main_burau} is to consider any Alexander matrix coming from applying the Fox calculus to a Wirtinger presentation of a knot or link group. More precisely, let $K \subset S^3$ be a link, and $\mathcal{B}$ a bridge presentation of $K$ in $S^2$. From this bridge presentation, one obtains a Wirtinger presentation for $\pi_1(S^3 - K)$, which yields an Alexander matrix $A$. In the case of the bridge presentations considered in Section \ref{sec:bridges}, defined for a braid $\rho$, the Alexander matrix coincides with a codimension-one submatrix of $\psi_n(\rho) - I_n$. For any bridge presentation, this set-up is discussed in Rasmussen's original paper \cite[Section 3.2]{ras03}. With our construction, we obtain:
	
	\begin{prop}
		\label{prop:fox}
		Let $K \subset S^3$ be an $n$-component link. As discussed above, let $A'$ be an Alexander matrix coming from a bridge presentation of $K$, and let ${\bf j, k} \subset \{1, \dots, n\}$ be multi-indices of size $m$. Let $A$ be the submatrix of $A'$ determined by ${\bf j}$ and ${\bf k}$. Then there exists a relatively $\Z^n$-graded Heegaard Floer homology theory whose Poincar\'e polynomial is $\det(A) \in \Z[t_1^{\pm 1}, t_2^{\pm 1}, \cdots, t_n^{\pm 1}]$.
	\end{prop}
	
	The issue with the groups of Proposition \ref{prop:fox} is that, in general, they depend on a choice of bridge presentation, and are not topological invariants in any meaningful sense. Nonetheless, it is possible that link invariants could be extracted from them.
	
	\section{Background on Quantum $\mathfrak{gl}(1 \vert 1)$}
	\label{sec:qs}
	
	\subsection{The quantum invariant $U_q(\mathfrak{gl}(1 \vert 1))$}
	\label{sec:q_backg}
	
	In this section, we begin moving toward this paper's second objective: using the theory developed in Section \ref{sec:theory} to relate knot Floer homology and the $U_q(\mathfrak{gl}(1 \vert 1))$ Reshitikhin-Turaev invariant. As the introduction discusses, our motivation here is to help unify the symplectic and quantum group theoretic perspectives for knot homology theories.
	
	Broadly speaking, Reshitikhin-Turaev invariants associate linear maps to tangles in the unit cube \cite{rt90}.\footnote{These should not be confused with Witten-Reshitkhin-Turaev invariants of three-manifolds, which are closely related.} {\em Quantum groups} are certain non-commutative Hopf algebras with extra structure well-suited for this setting: for every quantum group (more or less), there is an associated Reshitkhin-Turaev invariant. In what follows, we give a high-level overview of the Reshikhin-Turaev invariant associated to the quantum group $U_q(\mathfrak{gl}(1 \vert 1))$, the quantized universal enveloping algebra of $\mathfrak{gl}(1 \vert 1)$, informally referred to as quantum $\mathfrak{gl}(1 \vert 1)$. Many of the statements are applicable to Reshitikhin-Turaev invariants of other quantum groups as well; we refer the reader to \cite{km91} for an overview of Reshitikhin-Turaev invariants, and \cite{sar15} for details specific to $U_q(\mathfrak{gl}(1 \vert 1))$.
	
	A {\em tangle} $T$ is a one-manifold properly embedded in the unit cube $I^3 = [0,1]^3$, such that
	$$
	\partial T \subset I \times \{1/2\} \times \{0,1\}.
	$$
	We denote $\partial T \cap (I \times \{1/2\} \times \{0\})$ and $\partial T \cap (I \times \{1/2\} \times \{1\})$, the bottom and top boundaries of $T$, by $\partial T_-$ and $\partial T_+$ respectively, and we call $T$ an {\em $(m,n)$-tangle} if $|\partial T_-| = m$ and $|\partial T_+| = n$. Tangles are considered up to ambient isotopies that fix $I \times \{1/2\} \times \{0,1\}$ set-wise.
	
	We view oriented tangles as the set of morphisms in a category $\mathcal{T}$, whose objects are (isotopy classes of) finite sets of signed points in $I$. An oriented tangle $T$ is then a morphism from $-\partial T_- \subset I \times \{1/2\} \times \{0\}$ to $\partial T_+ \subset I \times \{1/2\} \times \{1\}$, with the signs of boundary points induced by the orientation, and $-\partial T_-$ indicating $\partial T_-$ with the signs of all points switched. Given two tangles $T_1$ and $T_2$ with $(\partial T_1)_+ \cong (\partial T_2)_-$, we form the composition $T_2 \circ T_1$ by stacking the tangles vertically with $T_2$ on top. There is also a tensor product on $\mathcal{T}$ given by juxtaposing tangles horizontally.
	
	Separately, let $V$ denote a two-dimensional super vector space over $\C(q)$ with even generator $v_0$ and odd generator $v_1$. By a {\em super vector space}, we mean a vector space equipped with a $\Z_2$-grading as described. We impose a super structure on the dual $V^* = \text{Hom}(V, \C(q))$ as well, such that if $\{v^*_0, v^*_1\}$ is the basis dual to $\{v_0, v_1\}$, then $v^*_0$ is even and $v^*_1$ is odd. Likewise, any tensor product of copies of $V$ and $V^*$ has a natural super structure induced by the structure on $V$. (The empty tensor product, $\C(q)$, is even.) Let $\mathcal{V}$ be the category whose objects are tensor products of $V$ and $V^*$, and whose morphisms are grading-preserving linear maps.
	
	The super algebra $U_q(\mathfrak{gl}(1\vert 1))$ admits an irreducible representation into $\text{Aut}(V)$, called the {\em vector representation}, and the {\em Reshetikhin-Turaev invariant} associated to this representation of $U_q(\mathfrak{gl}(1 \vert 1))$ is a monoidal functor
	$$
	Q : \mathcal{T} \to \mathcal{V}.
	$$
	The functor $Q$ sends a single positively signed point to $V$, a negatively signed point to $V^*$, and a sequence of points to the corresponding tensor product.\footnote{When defining $Q$, it is necessary to put framings on the tangles in $\mathcal{T}$. The framings will not affect our calculations here, however, so we omit them from the discussion.}
	
	We note three special cases of the functor $Q$. First, we consider the restriction of $Q$ to those tangles which are $n$-strand geometric braids, oriented so that all strands run from $\partial T_-$ to $\partial T_+$. In this case, $Q$ yields a braid group representation
	\begin{equation}
		\label{eq:q_braid}
		Q : B_n \to \text{Aut}(V^{\otimes n})
	\end{equation}
	for all $n$. Second, we can focus on a tangle which has no boundary points, i.e.~an oriented link $K \subset I^3$. Then $Q(K) \in \text{Aut}(\C(q))$, and it can be shown that $Q(K)$ is always the zero map in this case \cite[Proposition 4.4]{sar15}.
	
	In spite of the preceeding fact, there is a way to obtain nontrivial link invariants using $Q$. Let $T$ be a $(1,1)$-tangle; then $Q(T)$ is a morphism
	$$
	Q(T) : V \to V.
	$$
	Since the vector representation of $U_q(\mathfrak{gl}(1\vert 1))$ is irreducible, $Q(T) = c \cdot \text{Id}$ for some scalar $c \in \C(q)$ (cf.~\cite[Lemma 3.9]{km91}). This scalar is an invariant of the oriented link formed by closing up $T$ in $S^3$ \cite[Theorem 4.6]{sar15}.
	
	\begin{defn}
		Let $K \subset S^3$ be an oriented link, and $T$ a $(1,1)$-tangle with closure $K$. Then we define $\widehat{Q}(K) \in \C(q)$ to be the unique value such that $Q(T) = \widehat{Q}(K) \cdot \text{Id}$.
	\end{defn}
	
	The invariant $\widehat{Q}(K)$ coincides with the Alexander polynomial of $K$.
	
	\begin{prop}[{\cite[Theorem 4.10]{sar15}}]
		\label{prop:q_alex}
		Let $K \subset S^3$ be an oriented link. Then
		$$
		\widehat{Q}(K) = \Delta_K(t)\vert_{q = t^{1/2}},
		$$
		where $\Delta_K(t) \in \C[t^{1/2},t^{-1/2}]$ is the symmetrized Alexander polynomial of $K$.
	\end{prop}
	
	We will not assume Proposition \ref{prop:q_alex} in any of what follows. To emphasize this, we will continue to refer to the above invariant as $\widehat{Q}(K)$ rather than $\Delta_K$ for a given link $K$.
	
	\subsection{The $U_q(\mathfrak{gl}(1 \vert 1))$ invariant and the Burau repesentation}
	
	It is a classical fact that the braid representation (\ref{eq:q_braid}) can be otained from the Burau representation. As in Section \ref{sec:ub}, let
	$$
	\psi_n : B_n \to \text{GL}_n(\Z[t,t^{-1}])
	$$
	be the full Burau representation of the $n$-strand braid group. Let $W$ be an $n$-dimensional vector space over $\C(q)$, with basis $w_1, \dots, w_n$, and set $q = t^{1/2}$. Extending scalars, we think of $\psi_n$ as a representation
	$$
	\psi_n : B_n \to \text{Aut}(W).
	$$
	
	Let $\bigwedge W$ denote the exterior algebra of $W$, 
	$$
	\bigwedge W = \bigoplus_{m = 0}^n \bigwedge \nolimits^m W,
	$$
	and given a (possibly empty) multi-index ${\bf j} = \{j_1, \dots, j_m\} \subset \{1, \dots, n\}$, let 
	$$
	w_{\bf j} = w_{j_1} \wedge \dots \wedge w_{j_m} \in \bigwedge \nolimits^m W.
	$$
	Then $\bigwedge W$ is a $2^n$-dimensional vector space over $\C(q)$, with basis
	\begin{equation}
		\label{eq:wedge_basis}
		\{w_{\bf j} \mid {\bf j} \subset \{1, \dots ,n\}\},
	\end{equation}
	and $\bigwedge W$ can be given a super structure by declaring elements of $\bigwedge \nolimits^m W$ even if and only if $m$ is. The Burau representation $\varphi_n$ induces a representation
	$$
	\psi^\wedge_n : B_n \to \text{Aut}(\bigwedge W).
	$$
	
	For any multi-index ${\bf j} = \{j_1, \dots, j_m\} \subset [n]$, let $\vec{s}({\bf j}) \in \{0,1\}^n$ denote the binary sequence with a one in the $j$th coordinate if and only if $j \in {\bf j}$. Additionally, for any such sequence $\vec{s} = (s_1, \dots, s_n) \in \{0,1\}^n$, let
	$$
	v_{\vec{s}} = v_{s_1} \otimes \cdots \otimes v_{s_n} \in V^{\otimes n},
	$$
	where $V$ is the super vector space of the previous section. Then the map
	\begin{equation}
		\label{eq:sup_isom}
		\bigwedge W \to V^{\otimes n},
	\end{equation}
	given by $w_{\bf j} \mapsto v_{\vec{s}({\bf j})}$, is an isomorphism of super vector spaces. Via this isomorphism, we think of $\psi^\wedge_n$ as a function
	$$
	\psi^\wedge_n : B_n \to \text{Aut}(V^{\otimes n}).
	$$
	
	We have the following result.
	
	\begin{prop}[{\cite[Section 5]{kasa91}}]
		\label{prop:isom}
		Fix a braid $\rho \in B_n$, and let exp$(\rho) \in \Z$ denote the braid exponent, i.e.~the number of positive crossings minus the number of negative crossings in any geometric braid diagram of $\rho$. Then there exists a diagonal matrix $A$, i.e.~a rescaling of basis elements, such that
		\begin{equation}
			\label{eq:isom}
			A^{-1}Q(\rho)A = q^{-\text{exp}(\rho)} \psi^\wedge_n(\rho).
		\end{equation}
	\end{prop}
	
	\begin{rmk}
		In defining the invariant $\widehat{Q}$, we've followed the conventions used by Sartori \cite{sar15} and Manion \cite{man19}. To verify Proposition \ref{prop:isom}, compare the {\em inverse} of the matrix in (\ref{eq:explicit_burau}) with the first matrix in \cite[Definition 2.2.3]{man19}. It is necessary to take the inverse because we define the Burau representation using clockwise twists, while a positive crossing in Manion's set-up goes counter-clockwise. (This means, unfortunately, that our standard generators $\sigma_i$ in Section \ref{sec:ub} contribute negatively to the braid exponent!) Kauffman and Saleur use different conventions \cite{kasa91}, but their version of Proposition \ref{prop:isom} ultimately looks the same.
	\end{rmk}
	
	\subsection{A state sum formulation of $\widehat{Q}$}
	
	It is common in quantum topology to express Reshetikhin-Turaev invariants as {\em state sums}, which here are sums over elements of matrices. To this end, we introduce some notation. As above, for a sequence $\vec{s} = (s_1, \dots, s_n) \in \{0,1\}^n$ let
	$$
	v_{\vec{s}} = v_{s_1} \otimes \cdots \otimes v_{s_n} \in V^{\otimes n}.
	$$
	Then the set $\{v_{\vec{s}} \mid \vec{s} \in \{0,1\}^n\}$ is a standard basis for $V^{\otimes n}$. By definition, $v_{\vec{s}}$ is even if and only if $\vec{s}$ contains an even number of ones, and we write $|\vec{s}|$ to indicate the number of ones in $\vec{s}$. (Thus, for a multi-index ${\bf j}$, $|{\bf j}| = |\vec{s}({\bf j})|$.) Additionally, for any linear map $f : V^{\otimes n} \to V^{\otimes n}$ and sequences $\vec{s}, \vec{t} \in \{0,1\}^n$, we use $f^{\vec{s}}_{\vec{t}} \in \C(q)$ to denote the coefficient of $v_{\vec{t}}$ in $f(v_{\vec{s}})$. Equivalently, the $f^{\vec{s}}_{\vec{t}}$ are determined by the equation
	\begin{equation}
		\label{eq:coeffs}
		f(v_{\vec{s}}) = \sum_{\vec{t} \in \{0,1\}^n} f^{\vec{s}}_{\vec{t}} v_{\vec{t}}.
	\end{equation}
	With this notation, if a link $K$ is the closure of a geometric braid $\rho$, we can express $\widehat{Q}(K)$ in terms of $Q(\rho)$.
	
	\begin{lemma}
		\label{lem:str}
		Let $\rho \in B_n$ be a geometric braid whose closure is a link $K \subset S^3$, and let $Q(\rho) : V^{\otimes n} \to V^{\otimes n}$ be the image of $\rho$ under $Q$. Then
		\begin{equation}
			\label{eq:str}
			\widehat{Q}(K) = q^{n-1} \sum_{\vec{s} \in \{0,1\}^{n - 1}} (-1)^{|\vec{s}|} Q(\rho)_{\{0\} \times \vec{s}}^{\{0\} \times \vec{s}},
		\end{equation}
		where $\{0\} \times \vec{s}$ indicates the sequence $\vec{s}$ with a zero appended on the left.
	\end{lemma}
	
	\begin{rmk}
		\label{rmk:str}
		Lemma \ref{lem:str} is well known to experts, and is not particular to $U_q(\mathfrak{gl}(1 \vert 1))$. A more general version of the identity holds for any ribbon category---see, for example, \cite[Section 2.1.4]{my18}, for a similar calculation with $U_q(\mathfrak{sl}(2))$. The specific formula (\ref{eq:str}) appears in a slightly different form as equation (2.8) in Kauffman and Saleur's classic paper \cite{kasa91}, motivated by statistical mechanics. It is closely related to taking the {\em super trace} of the map $Q(\rho)$, which is the trace of the even part minus the trace of the odd part.
	\end{rmk}
	
	\begin{proof}[Proof of Lemma \ref{lem:str}]
		We sketch the proof. First, we represent $K$ as a $(1,1)$-tangle $T$ by closing the $n - 1$ rightmost strands of the geometric braid $\rho$, as on the left side of Figure \ref{fig:state_sum}. Then $Q(T) : V \to V$ is a scalar multiple of the identity map, and $Q(T)_0^0 = Q(T)_1^1 = \widehat{Q}(K)$ by definition. We compute $Q(T)_0^0$ by writing $T$ as a composition of three tangles $T_\text{max} \circ T_\text{mid} \circ T_\text{min}$, as on the right side of Figure \ref{fig:state_sum}, where:
		\begin{itemize}
			\item $T_\text{min}$ is a single vertical strand, with $n - 1$ nested ``cups'' to the right of it.
			\item $T_\text{mid}$ is the geometric braid $\rho$, with $n - 1$ parallel vertical strands to the right of it.
			\item $T_\text{max}$ is a single vertical strand, with $n - 1$ nested ``caps'' to the right of it.
		\end{itemize}
		Further, note that $T_\text{mid} = \rho \otimes {\bf 1}_{n - 1}$, where $\rho$ is the geometric braid and ${\bf 1}_{n - 1}$ is the identity braid on $n - 1$ strands. Since the functor $Q$ respects the tensor product, we have
		$$
		Q(T) = Q(T_\text{max}) \circ (Q(\rho) \otimes \text{Id}_{V^{\otimes (n - 1)}}) \circ Q(T_\text{min}),
		$$
		and by basic linear algebra,
		\begin{align}
		\nonumber Q(T)_0^0 &= \sum_{\vec{r}, \vec{t} \in \{0,1\}^{2n - 1}} Q(T_\text{max})^{\vec{t}}_0 \cdot (Q(\rho) \otimes \text{Id}_{V^{\otimes (n - 1)}})_{\vec{t}}^{\vec{r}} \cdot Q(T_\text{min})^0_{\vec{r}} \\
		\nonumber &= \sum_{\vec{s}_1, \vec{s}_2 \in \{0,1\}^n} \sum_{\vec{t}_1, \vec{t}_2 \in \{0,1\}^{n - 1}} Q(T_\text{max})^{\vec{s}_2 \times \vec{t}_2}_0 \cdot Q(\rho)_{\vec{s}_2}^{\vec{s}_1} \cdot {\text{Id}_{V^{\otimes (n - 1)}}}_{\vec{t}_2}^{\vec{t}_1} \cdot Q(T_\text{min})^0_{\vec{s}_1 \times \vec{t}_1} \\
		\nonumber&= \sum_{\vec{s}_1, \vec{s}_2 \in \{0,1\}^n} \sum_{\vec{t}_1, \vec{t}_2 \in \{0,1\}^{n - 1}} Q(T_\text{max})^{\vec{s}_2 \times \vec{t}_2}_0 \cdot Q(\rho)_{\vec{s}_2}^{\vec{s}_1} \cdot \delta_{\vec{t}_2}^{\vec{t}_1} \cdot Q(T_\text{min})^0_{\vec{s}_1 \times \vec{t}_1} \\
		\label{eq:fine}
		&= \sum_{\vec{s}_1, \vec{s}_2 \in \{0,1\}^n} \sum_{\vec{t}_1 \in \{0,1\}^{n - 1}} Q(T_\text{max})^{\vec{s}_2 \times \vec{t}_1}_0 \cdot Q(\rho)_{\vec{s}_2}^{\vec{s}_1} \cdot Q(T_\text{min})^0_{\vec{s}_1 \times \vec{t}_1}.
		\end{align}
		Equation (\ref{eq:fine}) can be simplified further by expressing $T_\text{min}$ as a combination of trivial tangles and tangles which consist of a single, left-oriented cup. Similarly, $T_\text{max}$ can be expressed in terms of trivial tangles and tangles which consist of a single, right-oriented cap. The maps which $Q$ associates to cups and caps can be found in \cite[Section 2.3]{man19}.
		
		If $\vec{s} = (s_1, \dots, s_n) \in \{0,1\}^n$, let flip$(\vec{s}) \in \{0,1\}^n$ be the flipped sequence
		$$
		\text{flip}(\vec{s}) = (s_n, s_{n - 1}, \dots, s_1).
		$$
		It is not difficult to show that, for all $\vec{s}, \vec{t} \in \{0, 1\}^{n - 1}$ and $r \in \{0,1\}$,
		\begin{equation}
			\label{eq:min}
			Q(T_\text{min})^0_{r \times \vec{s} \times \vec{t}} = \begin{cases}
				1 & \text{if } r = 0, \vec{s} = \text{flip}(\vec{t}) \\
				0 & \text{otherwise}
		\end{cases}
		\end{equation}
		and
		\begin{equation}
			\label{eq:max}
			Q(T_\text{max})^0_{r \times \vec{s} \times \vec{t}} = \begin{cases}
				(-1)^{|\vec{s}|}q^{n - 1} & \text{if } r = 0, \vec{s} = \text{flip}(\vec{t}) \\
				0 & \text{otherwise}
		\end{cases}.
		\end{equation}
		In particular, the only nonzero summands of (\ref{eq:fine}) occur when $\vec{s}_1 = \{0\} \times \text{flip}(\vec{t}_1) = \vec{s}_2$. Considering only these summands, and plugging in the values for the coefficients of $Q(T_\text{max})$ and $Q(T_\text{min})$ computed above, yields (\ref{eq:str}).
	\end{proof}
	
	\begin{figure}
		\labellist
		\small\hair 2pt
		\pinlabel $T_\text{min}$ [l] at 1100 85
		\pinlabel $T_\text{max}$ [l] at 1100 515
		\pinlabel {$T_\text{mid} = \rho \otimes {\bf 1}_{n - 1}$} [l] at 1100 295
		\endlabellist
		
		\centering
		\includegraphics[height=5cm]{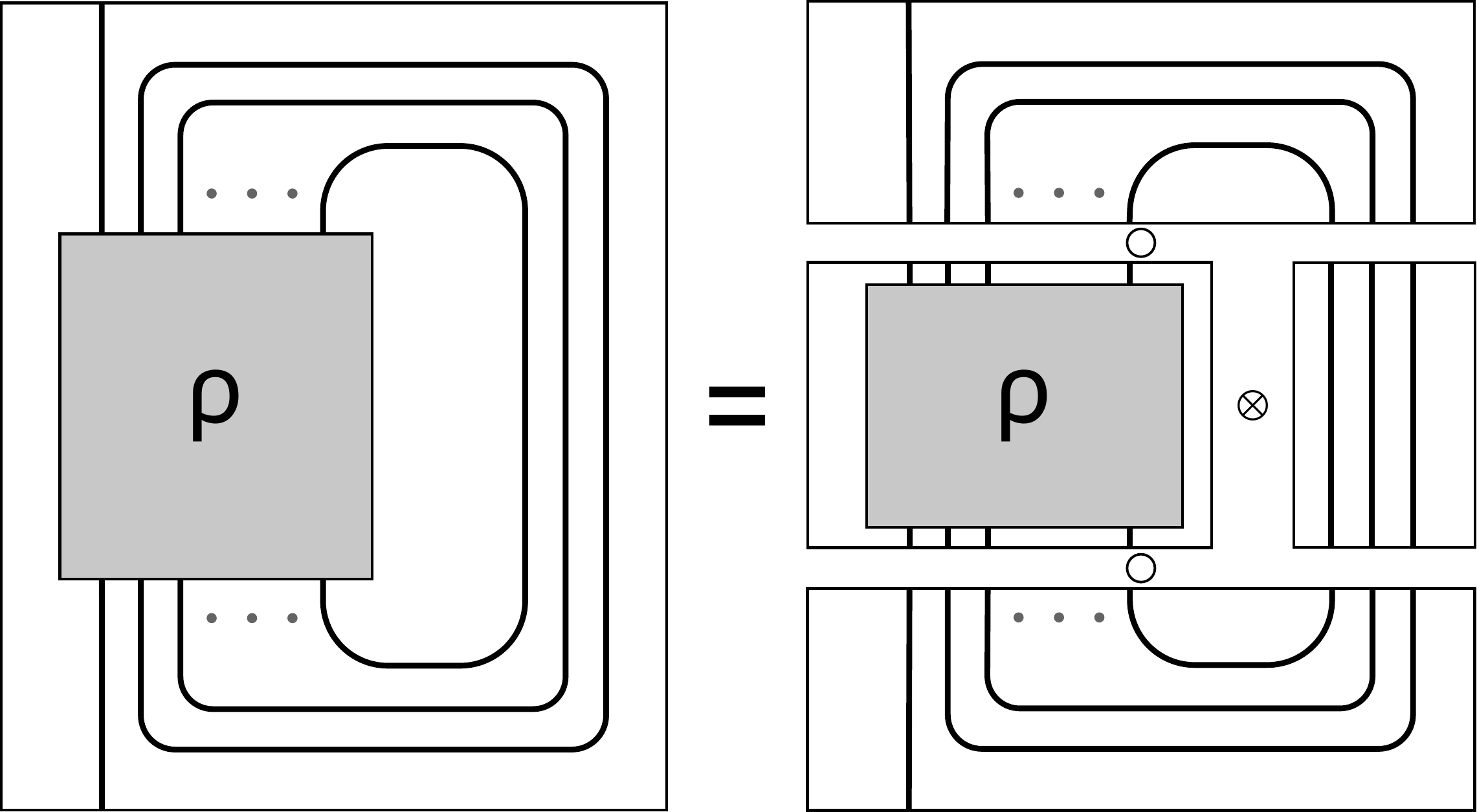}
		\hspace{1cm}
		\caption{Using $\rho$ to express $K$ as a $(1,1)$-tangle}
		\label{fig:state_sum}
	\end{figure}
	
	Combining Proposition \ref{prop:isom} and Lemma \ref{lem:str}, we can rewrite the invariant $\widehat{Q}$ using the representation $\varphi^\wedge_n$.
	
	\begin{lemma}
		\label{lem:str_two}
		Let $\rho \in B_n$ be a braid whose closure is a link $K \subset S^3$, with braid exponent exp$(\rho)$. Then
		\begin{equation}
			\label{eq:str_two}
			\widehat{Q}(K) = q^{n-1-\text{exp}(\rho)} \sum_{\vec{s} \in \{0,1\}^{n - 1}} (-1)^{|\vec{s}|} \psi^\wedge_n(\rho)_{\{0\} \times \vec{s}}^{\{0\} \times \vec{s}}.
		\end{equation}
	\end{lemma}
	
	\begin{proof}
		The proof of Lemma \ref{lem:str_two} is identical to the proof of Lemma \ref{lem:str}, using the rescaled basis determined by the matrix $A$ of Proposition \ref{prop:isom}, and making the substitution (\ref{eq:isom}). Since equation (\ref{eq:fine}) is basis-independent, and rescaling basis vectors does not affect (\ref{eq:min}) or (\ref{eq:max}), the proof goes through without modifications.
	\end{proof}
	
	For future results it will be useful to think of $\varphi^\wedge_n$ as a map from $B_n$ into $\text{Aut}(\bigwedge W)$ rather than into $\text{Aut}(V^{\otimes n})$, so we restate Lemma \ref{lem:str_two} in this framework. Taking (\ref{eq:wedge_basis}) as our basis for $\bigwedge W$, if ${\bf j}, {\bf k} \subset \{1, \dots, n\}$ are multi-indices of size $|{\bf j}| = |{\bf k}| = m$, and $f : \bigwedge W \to \bigwedge W$ is a grading-preserving linear map, we let $f^{\bf j}_{\bf k}$ denote the coefficient of $w_{\bf k}$ in $f(w_{\bf j})$. Then we have:
	
	\begin{cor}
		\label{cor:str_three}
		Let $\rho \in B_n$ be a braid with exponent exp$(\rho)$, whose closure is a link $K \subset S^3$. Thinking of $\varphi^\wedge_n$ as a map from $B_n$ to $\text{Aut}(\bigwedge W)$, (\ref{eq:str_two}) can be rewritten
		\begin{equation}
			\label{eq:str_three}
			\widehat{Q}(K) = q^{n-1-\text{exp}(\rho)} \sum_{{\bf j} \subset \{2, 3, \dots, n\}} (-1)^{|{\bf j}|} \psi^\wedge_n(\rho)_{\bf j}^{\bf j}.
		\end{equation}
	\end{cor}
	
	Excluding $1$ from the multi-indices ${\bf j}$ in (\ref{eq:str_three}) is equivalent to fixing the first coordinate as $0$ in (\ref{eq:str_two}). Additionally, the set of multi-indices in (\ref{eq:str_three}) includes the empty multi-index, which has coefficient $\varphi^\wedge_n(\rho)_\varnothing^\varnothing = 1$.
	
	\section{Heegaard Floer Homology and Quantum $\mathfrak{gl}(1 \vert 1)$}
	\label{sec:lasts}
	
	\subsection{Quantum $\mathfrak{gl}(1 \vert 1)$ and knot Floer theory}
	\label{sec:last}
	
	If $\rho \in B_n$ is a braid, then Proposition \ref{prop:isom} gives a way to understand the coefficients of $Q(\rho)$ geometrically using the noodle/chopstick calculus. As in Definition \ref{def:fn}, let $\eta_1, \dots, \eta_n \subset {\bf D}_n$ be the standard noodles, let $\varepsilon_1, \dots, \varepsilon_n \subset {\bf D}_n$ be the standard chopsticks, and for any intersection point $x \in \rho(\eta_j) \cap \varepsilon_k$, let $S(x) \in \Z$ be the sheet function of Definition \ref{def:sheet}. Finally, given a permutation $\sigma \in S_m$, and multi-indices ${\bf j} = \{j_1, \dots, j_m\}$ and ${\bf k} = \{k_1, \dots, k_m\}$ of size $m$, let $\Gamma_\sigma^{\bf j, k}(\rho)$ denote the set of sets of intersection points
	\begin{equation}
		\label{eq:Gamma}
		\Gamma_\sigma^{\bf j, k}(\rho) = \{\{x_1, \dots, x_m\} \subset {\bf D}_n \mid x_r \in \rho(\eta_{j_r}) \cap \varepsilon_{k_{\sigma(r)}} \text{ for all } r = 1, \dots, m \}.
	\end{equation}
	Although the notation is cumbersome, the next lemma follows directly from Lemma \ref{lem:reform} and the definition of $\psi^\wedge_n$. As in Corollary \ref{cor:str_three}, we think of $\psi^\wedge_n$ as a map from $B_n$ to $\text{Aut}(\bigwedge W)$, and we recall that $q = t^{1/2}$.
	
	\begin{lemma}
		\label{lem:ugly}
		Fix a braid $\rho \in B_n$ and multi-indices ${\bf j}, {\bf k} \subset [n]$ of size $m$. Then
		$$
		\psi^\wedge_n(\rho)^{\bf j}_{\bf k} = \sum_{\sigma \in S_m} \sum_{\{x_1, \dots, x_m\} \in \Gamma^{\bf j, k}_\sigma(\rho)} (-1)^{\text{sgn}(\sigma)} \prod_{r = 1}^m \text{sgn}(x_r)t^{S(x_r)}.
		$$
	\end{lemma}
	
	Using Lemma \ref{lem:ugly}, we can now expand the state sum in Corollary \ref{cor:str_three} geometrically.
	
	\begin{prop}
		\label{prop:expanded}
		Let $\rho \in B_n$ be a braid with closure $K \subset S^3$. Then
		\begin{equation}
			\label{eq:str_four}
		\widehat{Q}(K) = t^{(n - 1 - \text{exp}(\rho))/2} \sum_{{\bf j} \subset \{2, \dots, n\}} \sum_{\sigma \in S_{|{\bf j}|}} \sum_{\{x_1, \dots, x_{|{\bf j}|}\} \in \Gamma_\sigma^{\bf j, j }(\rho)} (-1)^{|{\bf j}| + \text{sgn}(\sigma)} \prod_{r = 1}^{|{\bf j}|} \text{sgn}(x_r)t^{S(x_r)}.
		\end{equation}
	\end{prop}
	
	Returning to knot Floer homology, let $K \subset S^3$ be an oriented link, and let $\rho \in B_n$ be a braid whose closure is $K$. Let $\mathcal{H}(m(K))$ be the Heegaard diagram $\mathcal{H}(\rho, \{2, \dots, n\}, \{2, \dots, n\})$ of Definition \ref{def:heediag} with the basepoints $p_L$ and $p_R$ removed---as discussed in Section \ref{sec:props}, $\mathcal{H}(m(K))$ is then a Heegaard diagram for $m(K)$. If $x \in \mathbb{T}_\alpha \cap \mathbb{T}_\beta$ is a generator for $\widehat{CF}(\mathcal{H}(m(K)))$ with Maslov grading $j$ and Alexander grading $k$, then we say the {\em weight} $w(x)$ of $x$ is the decategorified value of $x$, that is,
	$$
	w(x) = (-1)^jt^k.
	$$
	Similarly, if $x$ is a summand of (\ref{eq:str_four}), then the weight $w(x)$ of $x$ is simply its value in $\Z[t, t^{-1}]$. With this, we are prepared to state the main result of the section.
	
	\begin{thm}
		\label{thm:clear}
		There is a bijection $f$ between the generators of the chain complex \break $\widehat{CF}(\mathcal{H}(m(K)))$ and the summands used to calculate $\widehat{Q}(K)$ in (\ref{eq:str_four}). Moreover, if $x \in \mathbb{T}_\alpha \cap \mathbb{T}_\beta$ is generator of $\widehat{CF}(\mathcal{H}(m(K)))$, then
		$$
		(-1)^{1 - n}t^{(n - 1 - \text{exp}(\rho))/2} w(x) = w(f(x)).
		$$
		In other words, $f$ is weight-preserving up to a global normalization.
	\end{thm}
	
	Of course, we already know the sums of these two quantities are equal up to a normalization, since both give $\Delta_K = \Delta_{m(K)}$. However, as noted after Proposition \ref{prop:q_alex}, we have not assumed this fact for either summation. Theorem \ref{thm:clear} is more or less a geometric version of Kauffman and Saleur's ``det $=$ str'' identity \cite[Appendix]{kasa91}, where str indicates the super trace as in Remark \ref{rmk:str}. This identity states, approximately, that
	$$
	\det(\psi - I) = \text{str}(\psi^\wedge)
	$$
	for any linear map $\psi$. The idea underlying our proof of Theorem \ref{thm:clear} is simple: the bijection $f$ is given by removing anchor points from generators of $\widehat{CF}(\mathcal{H}(m(K)))$, and then using the identification of Lemma \ref{lem:relating}.
	
	\begin{proof}[Proof of Theorem \ref{thm:clear}]
		We follow the notation of Section \ref{sec:diags}. Let $x = \{x_1, \dots, x_{n - 1}\} \in \mathbb{T}_\alpha \cap \mathbb{T}_\beta$ be a generator of $\widehat{CF}(\mathcal{H}(m(K)))$, so that $x_r \in \alpha_k \cap \beta_j$ for some $j, k \in \{2, \dots, n\}$, for each $r \in \{1, \dots, n - 1\}$. As discussed prior to Definition \ref{def:anchor}, each $x_r \in x$ either lies in the punctured disk $\mathring{\bf D}_n \subset \Sigma$, or is an anchor point. Define a subset $x' \subset x$ by
		$$
		x' = \{x_r \in x \mid \text{$x_r$ is not an anchor point}\},
		$$
		and suppose without loss of generality that $x' = \{x_1, \dots, x_m\}$ for some $0 \leq m \leq n - 1$. Define multi-indices ${\bf j} = \{j_1, \dots, j_m\}, {\bf k} = \{k_1, \dots, k_m\} \subset \{2, \dots, n\}$ by
		\begin{align*}
			{\bf j} &= \{j \mid x_r \in \beta_j \text{ for some } x_r \in x'\} \\
			{\bf k} &= \{k \mid x_r \in \alpha_k \text{ for some } x_r \in x'\}.
		\end{align*}
		Then, for some permutation $\sigma \in S_m$, we have that
		$$
		x_r \in \beta_{j_r} \cap \alpha_{k_{\sigma(r)}}
		$$
		for $r = 1, \dots, m$. We claim ${\bf j} = {\bf k}$. Indeed, 
		\begin{align*}
		{\bf j} &= \{2, \dots, n\} \setminus \{s \mid \alpha_s \cap x \text{ is an anchor point}\} \\
			&= \{2, \dots, n\} \setminus \{s \mid \beta_s \cap x \text{ is an anchor point}\} \\
			&= {\bf k},
		\end{align*}
		where the second equality follows from the fact that any anchor point intersects $\alpha$ and $\beta$ curves of the same index. Thus, $x_r \in \beta_{j_r} \cap \alpha_{j_{\sigma(r)}}$ for all $x_r \in x'$.
		
		By Lemma \ref{lem:relating}, since $x_r \in \mathring{\bf D}_n$ for all $x_r \in x'$, each $x_r \in x'$ can be naturally identified with some $y_r \in \rho(\eta_{j_r}) \cap \varepsilon_{j_{\sigma(r)}}$ for $r = 1, \dots, m$. In this way the set $\{x_1, \dots, x_m\}$ can be identified with a set
		$$
		\{y_1, \dots, y_m\} \in \Gamma^{\bf j, j}_\sigma(\rho),
		$$
		which is a summand of (\ref{eq:str_four}). We define a map $f$ from generators of $\widehat{CF}(\mathcal{H}(m(K)))$ to summands of (\ref{eq:str_four}) in this way, by sending $x$ to the corresponding set $\{y_1, \dots, y_m\}$ as we've just described.
		
		We verify that $f$ is a bijection by defining an inverse. Suppose
		$$
		\{y_1, \dots, y_m\} \in \Gamma^{\bf j, j}_\sigma(\rho)
		$$
		for some multi-index ${\bf j} \subset \{2, \dots, n\}$ with $|{\bf j}| = m$, and some permutation $\sigma \in S_m$, so that $\{y_1, \dots, y_m\}$ is a summand of (\ref{eq:str_four}). Then, by Lemma \ref{lem:relating}, $\{y_1, \dots, y_m\}$ can be identified with a set of points $x' = \{x_1, \dots, x_m\}$ in the Heegaard diagram $\mathcal{H}(m(K))$ with $x_r \in \beta_{j_r} \cap \alpha_{j_{\sigma(r)}}$ for $r = 1, \dots, m$. We complete $x'$ to a generator of $\widehat{CF}(\mathcal{H}(m(K)))$ by adding the $r$th anchor point for all $r \notin {\bf j}$; for example, the empty index is identified with the generator of $\widehat{CF}(\mathcal{H}(m(K)))$ consisting of all the anchor points. This map is an inverse for $f$.
		
		It remains to see that $f$ is weight-preserving up to the specified factor. For this, let $x = \{x_1, \dots, x_{n - 1}\}$ be a generator for $\widehat{CF}(\mathcal{H}(m(K)))$ such that the anchor points of $x$ are precisely the coordinates $x_{m + 1}, \dots, x_{n - 1}$ for some $0 \leq m \leq n - 1$. Let $\{y_1, \dots, y_m\} \in \Gamma^{\bf j, j}_\sigma$ be the image of $x$ under the bijection $f$, with ${\bf j}$ a multi-index of size $m$ and $\sigma \in S_m$. By Definition \ref{def:grading}, since each anchor point has Alexander grading zero and a negative orientation,
		$$
		w(x) = (-1)^{n - 1 - m + \text{sgn}(\sigma)} \prod_{r = 1}^{m} \text{sgn}(y_r)t^{S(y_r)}.
		$$
		On the other hand, using (\ref{eq:str_four}), $f(x) = \{y_1, \dots, y_m\}$ has weight
		$$
		w(f(x)) = t^{(n - 1 - \text{exp}(\rho))/2} \cdot (-1)^{m + \text{sgn}(\sigma)} \prod_{r = 1}^{m} \text{sgn}(y_r)t^{S(y_r)}.
		$$
		Comparing these two quantities concludes the proof of the theorem.
	\end{proof}
	
	\begin{rmk}
		\label{rmk:str_two}
		In analogy with Proposition \ref{prop:expanded}, the full super trace of $Q(\rho)$ can be written
		$$
			\text{str}(Q(\rho))= t^{- \text{exp}(\rho)/2} \sum_{{\bf j} \subset [n]} \sum_{\sigma \in S_{|{\bf j}|}} \sum_{\{x_1, \dots, x_{|{\bf j}|}\} \in \Gamma_\sigma^{\bf j, j }(\rho)} (-1)^{|{\bf j}| + \text{sgn}(\sigma)} \prod_{r = 1}^{|{\bf j}|} \text{sgn}(x_r)t^{S(x_r)}.
		$$
		The same argument used in the proof of Theorem \ref{thm:clear} gives a weight-preserving bijection between the summands in this formula and the generators of the chain complex \break $\widehat{CFB}(\rho, [n], [n])$ for the full Heegaard diagram $\mathcal{H}(\rho, [n], [n])$, again forgetting the basepoints $p_L$ and $p_R$. In this case, however, since $\mathcal{H}(\rho, [n], [n])$ is a Heegaard diagram for $S^3$ and the invariant $Q$ vanishes for links, the total of both summations is zero. In fact the homology group $\widehat{HFB}(\rho, [n], [n])$ is trivial, so this case is less interesting.
	\end{rmk}
	
	We make two other observations about Theorem \ref{thm:clear}. First, a similar mirroring occurs in Manion's work \cite[Theorem 1.4.2]{man19}, and this mirroring shows up again in the next section. Second, we note that by switching the $\alpha$ and $\beta$ curves in $\mathcal{H}(m(K))$, one obtains a Heegaard diagram $\mathcal{H}(K)$ for $m(m(K)) = K$, whose chain complex $\widehat{CF}(\mathcal{H}(K))$ is dual to $\widehat{CF}(\mathcal{H}(m(K)))$. We could have equivalently used the chain complex $\widehat{CF}(\mathcal{H}(K))$ in Theorem \ref{thm:clear}.
	
	\subsection{Bordered sutured manifolds for tangles}
	\label{sec:bst}
	
	One can interpret Theorem \ref{thm:clear}, and Remark \ref{rmk:str_two}, as demonstrating how Heegaard Floer chain complexes encode the super trace of the $U_q(\mathfrak{gl}(1 \vert 1))$ braid representation. It is natural to ask if a different Floer theory might contain the entire representation. In the next sections, we answer this question affirmatively using Zarev's bordered sutured Heegaard Floer homology \cite{za11}. Of course, from work of Ellis-Petkova-V\'ertesi \cite{epv19} and Manion \cite{man19}, we already know that this and more is true: appropriate tangle Floer theories categorify the whole $U_q(\mathfrak{gl}(1 \vert 1))$ quantum invariant. Our contribution to this problem is a simple, geometric approach.
	
	We will not discuss the details of bordered sutured Floer theory here---we refer the reader to Zarev's paper \cite{za11}, and to \cite{lot14} for more information on Lipshitz, Ozsv\'ath and Thurston's bordered Heegaard Floer homology. Additionally, to keep the scope of the paper manageable, we will only minimally develop our theory here. We will explore this direction in more detail in a follow-up paper.
	
	For our purposes, a {\em sutured manifold} is an oriented three-manifold $Y$ whose boundary is decomposed into two compact surfaces,
	$$
	\partial Y = R_+ \cup R_-,
	$$
	such that $R_+ \cap R_- = \partial R_+ = \partial R_-$, and each boundary component of $Y$ intersects both $R_+$ and $R_-$. The sets $R_+$ and $R_-$ are called the {\em positive} and {\em negative regions} of $\partial Y$ respectively, and the curves which make up their mutual boundary are called {\em sutures}. The regions and sutures are oriented so that the sutures' orientation agrees everywhere with the orientation induced by $R_+$, and disagrees with the orientation induced by $R_-$. The sutured structure on $Y$ is then determined by specifying $R_+$, and we represent sutured manifolds visually by shading the positive region.
	
	An {\em arc diagram} $\mathcal{Z}$ is a finite, unoriented graph $(Z,E)$, with vertex set $Z = \{Z_j\}$ and edge set $E = \{e_k\}$, such that each vertex of $\mathcal{Z}$ is equipped with a linear ordering of the edges which abut it. We represent arc diagrams topologically by realizing each vertex $Z_j$ as an oriented line segment, so that each edge adjacent to $Z_j$ intersects it at a distinct point, and the edge ordering is the order in which the edges intersect the segment. A {\em bordered sutured manifold} $(Y, R_+, \mathcal{Z}, \phi)$ is a sutured manifold $(Y, R_+)$ equipped with an embedding $\phi : \mathcal{Z} \to \partial Y$, such that:
	\begin{itemize}
		\item $\mathcal{Z}$ is (the topological realization of) an arc diagram.
		\item The image $\phi(Z_j)$ of each (line segment) vertex of $\mathcal{Z}$ is contained in a suture of $Y$, so that the orientation of the line segment agrees with the orientation of the suture.
		\item The image $\phi(e_k)$ of each edge of $\mathcal{Z}$ is contained in $R_-$.
	\end{itemize}
	To any bordered sutured manifold, one can associate a {\em bordered sutured Heegaard diagram} $\mathcal{H} = (\Sigma, \alpha^a, \alpha^c, \beta, \mathcal{Z}, \phi)$. This consists of:
	\begin{itemize}
		\item A compact, oriented surface $\Sigma$ with boundary.
		\item An arc diagram $\mathcal{Z} = (Z,E)$, with an embedding $\phi : \mathcal{Z} \to \Sigma$ such that the vertices $Z$ of $\mathcal{Z}$ are mapped orientation-preservingly into $\partial \Sigma$, and the interiors of the edges $E$ are sent to $\text{int}(\Sigma)$.
		\item A set $\alpha^a = \{\alpha^a_1, \dots, \alpha^a_k\}$ of pairwise disjoint, properly embedded arcs in $\Sigma$, such that $|\alpha^a| = |E|$ and each $\alpha_j^a \in \alpha^a$ is equal to $\phi(e_j)$ for some edge $e_j \in E$.
		\item A set $\alpha^c = \{\alpha^c_1, \dots, \alpha^c_m\}$ of pairwise disjoint, simple closed curves in $\Sigma$. The curves $\alpha^c$ are also required to be disjoint from the arcs $\alpha^a$, and the set $\alpha^a \cup \alpha^c$ of arcs and curves is linearly independent in $H_1(\Sigma, \phi(Z))$.
		\item A set $\beta = \{\beta_1, \dots, \beta_n\}$ of pairwise disjoint, simple closed curves in $\Sigma$, which are linearly independent in $H_1(\Sigma, \phi(Z))$ and which intersect $\alpha^a \cup \alpha^c$ transversely.
	\end{itemize}
	A bordered sutured manifold can be constructed from a bordered sutured Heegaard diagram as follows: one starts with $\Sigma \times I$, with $R_- = \Sigma \times \{0\}$ and $R_+  = \partial (\Sigma \times I) \setminus \text{int}(R_-)$. The embedding $\phi : \mathcal{Z} \to \Sigma$ is thought of as an embedding into $\Sigma \times \{0\} = R_-$, and the $\alpha^c$ curves are also considered as subsets of $\Sigma \times \{0\}$, while the $\beta$ curves are thought of as subsets of $\Sigma \times \{1\}$. One then attaches two-handles to $\Sigma \times \{0\}$ along every $\alpha^c$ curve, and attaches two-handles to $\Sigma \times \{1\}$ along every $\beta$ curve, with the understanding that the exposed boundary of an $\alpha^c$ handle after gluing is contained in $R_-$, and the boundary of a $\beta$ handle is contained in $R_+$.
	
	We will only be concerned with the following case of this theory. Let $T \subset I^3$ be an $(m, n)$-tangle in the unit cube, with no closed components, and let $Y(T) = I^3 - N(T)$ be the exterior of $T$, where $N$ denotes a regular open neighborhood.\footnote{We implicitly smooth the corners of $Y(T)$ as needed.} We endow $Y(T)$ with the structure of a sutured manifold by declaring
	$$
	R_+ = \partial N(T) \cup (I \times \{1\} \times I).
	$$
	Separately, let $\mathcal{Z}_k$ be the arc diagram shown in Figure \ref{fig:arc_diag}, with $k + 1$ vertices and $k$ edges. Let $\mathcal{Z} = \mathcal{Z}_m \sqcup \mathcal{Z}_n$ be a disjoint union, and let $\phi : \mathcal{Z} \to Y(T)$ be an embedding of $\mathcal{Z}$ into $\partial Y(T)$ such that:
	\begin{itemize}
		\item The embedding $\phi$ sends $\mathcal{Z}_m \subset \mathcal{Z}$ to the bottom face $I^2 \times \{0\}$ of the cube, with each edge a straight line segment which connects the back face $(I \times \{1\} \times I)$ with a distinct component of $\partial N(T)$.
		\item The embedding $\phi$ sends $\mathcal{Z}_n \subset \mathcal{Z}$ to the top face $I^2 \times \{1\}$ of the cube in an analogous way.
	\end{itemize}
	Then $(Y(T), R_+, \mathcal{Z}, \phi)$ is a bordered sutured manifold canonically associated to the tangle $T$. Figure \ref{fig:bsta} shows $(Y(T), R_+, \mathcal{Z}, \phi)$ for the tangle $T$ corresponding to the geometric braid $\sigma_1^{-2} \in B_2$.
	
	\begin{defn}
		\label{def:bst}
		Given a tangle $T \subset I^3$, we denote the bordered sutured manifold constructed above by $(Y(T), R_+, \mathcal{Z}, \phi)$, or simply by $Y(T)$.
	\end{defn}
	
	\begin{rmk}
		We can extend the definition of $Y(T)$ to tangles with closed components by equipping the boundary of a regular neighborhood of each closed component of $T$ with two oppositely oriented meridional sutures. This is standard in the literature---see \cite[Example 2.4]{juh06}.
	\end{rmk}
	
	\begin{figure}
		\labellist
		\small\hair 2pt
		\pinlabel $e_1$ at 35 115
		\pinlabel $e_2$ at 170 115
		\pinlabel $e_k$ at 505 115
		\endlabellist
		
		\centering
		\includegraphics[height=2cm]{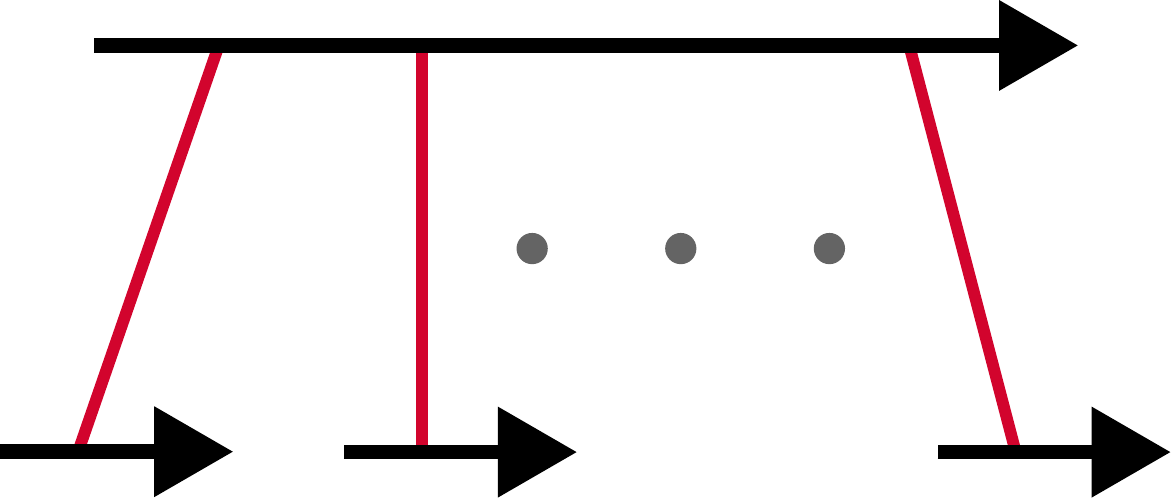}
		\caption{The arc diagram $\mathcal{Z}_k$}
		\label{fig:arc_diag}
	\end{figure}
	
	\begin{figure}
	\centering
	\subcaptionbox{A bordered sutured manifold $Y(T)$ \label{fig:bsta}}{
		\labellist
		\small\hair 2pt
		\pinlabel $\mathcal{Z}_2$ [l] at 650 88
		\pinlabel $R_+$ [l] at 650 315
		\pinlabel $\mathcal{Z}_2$ [l] at 650 570
		\endlabellist
		
		\includegraphics[height=7cm]{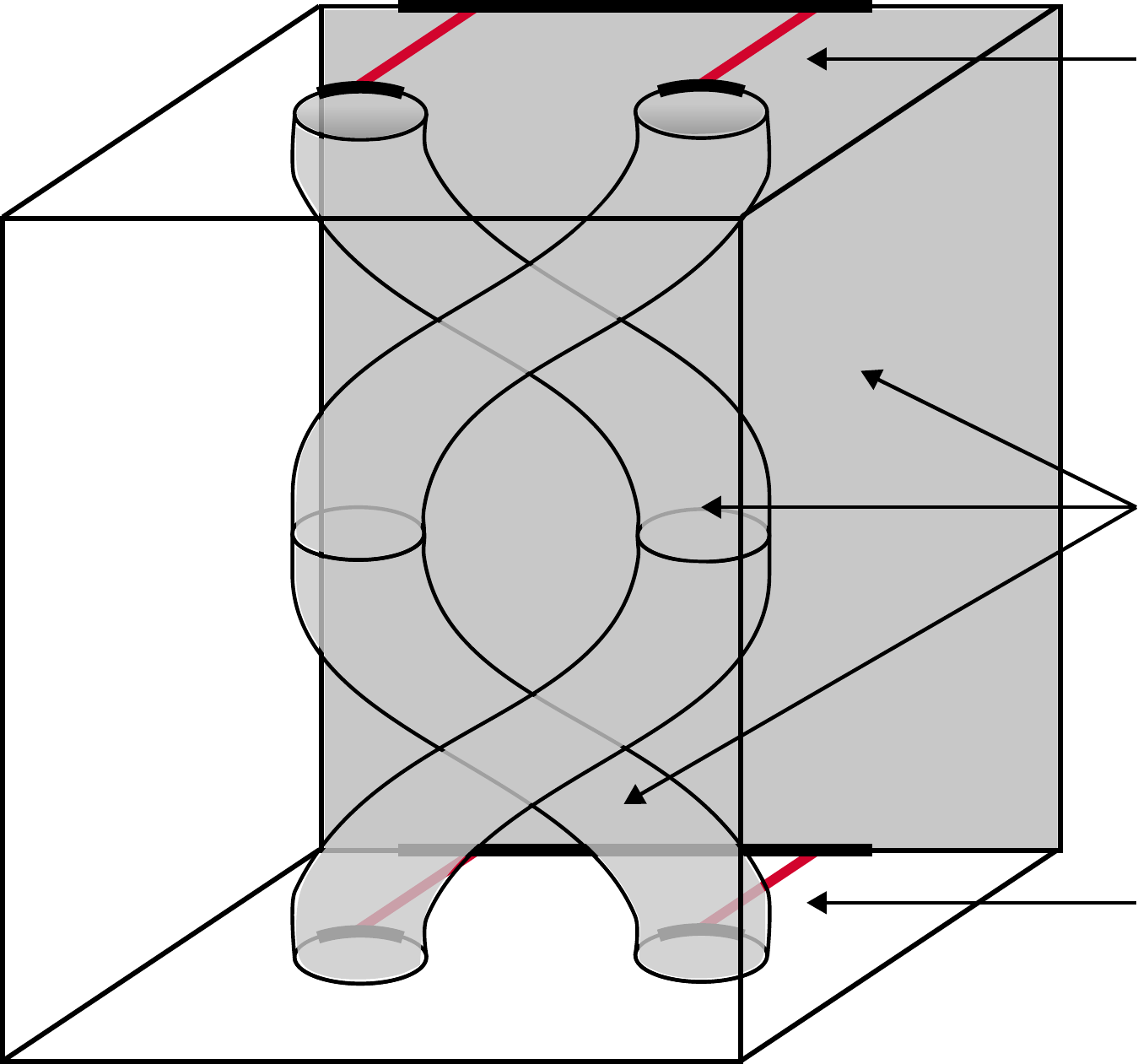}
		\hspace{.5cm}
	}
	\hspace{1cm}
	\subcaptionbox{The bordered sutured Heegaard diagram $\mathcal{H}(T)$ for $Y(T)$ (compare Figure \ref{fig:hda}) \label{fig:bstb}}{
		\labellist
		\small\hair 2pt
		\pinlabel $\alpha^a_1$ at 145 25
		\pinlabel $\alpha^a_2$ at 295 25
		\pinlabel $\alpha^a_3$  at 145 550
		\pinlabel $\alpha^a_4$  at 295 550
		\pinlabel $\beta_1$  at 130 100
		\pinlabel $\beta_2$  at 363 100
		\endlabellist
		
		\includegraphics[height=8cm]{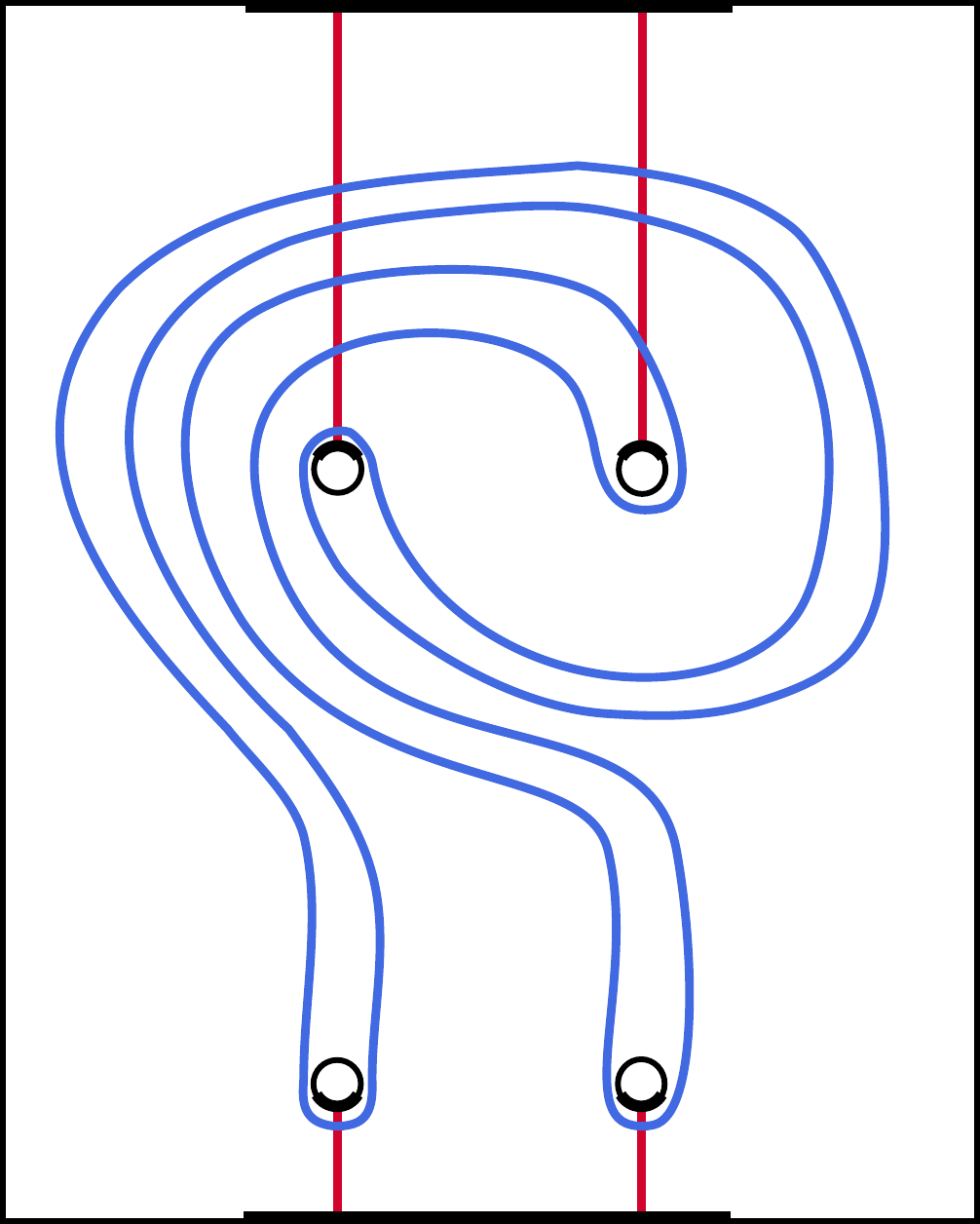}
	}
	\caption{}
	\label{fig:bst}
\end{figure}
	
	Let $\rho \in B_n$ be a braid, and let $T$ be the tangle corresponding to the {\em horizontally reflected braid} $\bar{\rho}$ of Remark \ref{rmk:closure}. In this case, $Y(T)$ admits a particularly nice bordered sutured Heegaard diagram which is closely related to $\mathcal{H}(\rho,[n],[n])$. Let $\C$ be the complex plane, with basepoint sets $z, w \subset \C$ as defined in Section \ref{sec:bridges}. As above, for each puncture $p \in z \cup w$, let $B_\epsilon(p) \subset \C$ be a small disk centered at $p$. We define a bordered sutured Heegaard diagram for $T$,
	\begin{equation}
		\label{eq:bs_diag}
		\mathcal{H}(T) = (\Sigma, \alpha^a, \beta, \mathcal{Z}, \phi),
	\end{equation}
	as follows. The Heegaard surface $\Sigma$ is the $2n$-punctured rectangle
	$$
	\Sigma = ([-1, 1] \times [-2.5, 1]) \setminus (B_\epsilon(z_1) \cup \cdots \cup B_\epsilon(z_n) \cup B_\epsilon(w_1) \cup \cdots \cup B_\epsilon(w_n)),
	$$
	with orientation induced by $\C$, and the curves $\beta = \{\beta_1, \dots, \beta_n\}$ are exactly the same as in the construction of $\mathcal{H}(\rho,[n],[n])$. The are no $\alpha^c$ curves, and there are $2n$ arcs $\alpha^a = \{\alpha^a_1, \dots, \alpha^a_{2n}\}$. For $j \leq n$, the arc $\alpha^a_j$ is a vertical line segment which travels from the bottom boundary of $\Sigma$ to the bottom of $\partial B_\epsilon(w_j)$; for $j > n$, the arc $\alpha^a_j$ is a vertical line segment which travels from the top of $\partial B_\epsilon(z_{j - n})$ to the upper boundary of $\Sigma$. The embedding
	$$
	\phi : \mathcal{Z} = \mathcal{Z}_n \sqcup \mathcal{Z}_n \to \Sigma
	$$ 
	is then determined, up to isotopy, by sending the edges in the bottom copy of $\mathcal{Z}_n$ to the arcs $\alpha^a_1, \dots, \alpha^a_n$, and the edges in the top copy of $\mathcal{Z}_n$ to the arcs $\alpha^a_{n + 1}, \dots, \alpha^a_{2n}$. Figure \ref{fig:bstb} shows the diagram $\mathcal{H}(T)$ for the bordered sutured manifold $Y(T)$ shown in Figure \ref{fig:bsta}.
	
	The diagram $\mathcal{H}(T)$ can be obtained from the diagram $\mathcal{H}(\rho,[n],[n])$ simply by removing the handles and the area of $S^2$ outside of the rectangle $[-1,1] \times [-2.5, 1]$, including the basepoints $p_L$ and $p_R$. In this way, each curve $\alpha_j$ in $\mathcal{H}(\rho,[n],[n])$ yields two arcs $\alpha^a_j$ and $\alpha^a_{j + n}$ in $\mathcal{H}(T)$. This observation leads to:
	
	\begin{lemma}
		\label{lem:int_bij}
		Fix a braid $\rho \in B_n$, and let $\mathcal{H}(T)$ be as in (\ref{eq:bs_diag}). Then the embedding $\Sigma \hookrightarrow \C$ induces a bijection between intersections of $\beta$ curves with $\alpha^a$ arcs in $\mathcal{H}(T)$, and intersections of $\beta$ curves with $\alpha$ curves in $\mathcal{H}(\rho,[n],[n])$.
	\end{lemma}
	
	\begin{rmk}
		In bordered sutured Floer homology (and most other Heegaard Floer theories), it's necessary to work with Heegaard diagrams which satisfy a condition called {\em admissibility}. Since we are interested in decategorified behavior, admissibility won't be important to us in the following sections. However, we note here that by \cite[Proposition 4.4.2]{za11}, any bordered sutured Heegaard diagram can be made admissible by isotopies of the $\beta$ curves.
	\end{rmk}
	
	\subsection{A decategorified tangle Floer invariant}
	\label{sec:bstq}
	
	To any arc diagram $\mathcal{Z}$, Zarev's bordered sutured Floer homology associates a differential graded algebra $\mathcal{A}(\mathcal{Z})$. To a bordered sutured three-manifold $Y = (Y, R_+, \mathcal{Z}, \phi)$, the theory associates an $\mathcal{A}_\infty$-module $\widehat{BSA}(Y, \mathcal{Z})$ over $\mathcal{A}(\mathcal{Z})$, which also has the structure of a $\Z/2$-vector space. The module $\widehat{BSA}(Y, \mathcal{Z})$ is an invariant of $Y$ up to chain homotopy equivalence. Because we are only interested in the decategorified structure of $\widehat{BSA}(Y, \mathcal{Z})$, we will suppress much of the complexity of these objects, and we will focus on the case where $Y = Y(T)$ as in Definition \ref{def:bst}. We refer the reader to \cite{za11, lot14} for more details.
	
	Let $\rho \in B_n$ be a braid, and let $T \subset I^3$ be the tangle corresponding to the reflected braid $\bar{\rho}$. Then the invariant $\widehat{BSA}(Y(T), \mathcal{Z})$ can be computed from the bordered sutured Heegaard diagram
	$$
	\mathcal{H}(T) = (\Sigma, \alpha^a, \beta, \mathcal{Z}, \phi)
	$$ 
	of (\ref{eq:bs_diag}). Although $\widehat{BSA}(Y(T), \mathcal{Z})$ is only defined up to chain homotopy equivalence, we abuse notation slightly by using $\widehat{BSA}(Y(T), \mathcal{Z})$ here to indicate the specific $\mathcal{A}_\infty$-module determined by $\mathcal{H}(T)$. A {\em generator} $x$ of $\widehat{BSA}(Y(T), \mathcal{Z})$, as determined by $\mathcal{H}(T)$, is an $n$-tuple of points
	$$
	x = \{x_1, \dots, x_n\} \subset \Sigma
	$$ 
	such that $x_j \in \alpha^a_{k_j} \cap \beta_{\sigma(j)}$, for some permutation $\sigma \in S_n$ and some $n$-element subset $\{\alpha^a_{k_1}, \dots, \alpha^a_{k_n}\}$ of $\alpha^a$. We say $x$ {\em occupies} the arc $\alpha^a_k$ if $x_j \in \alpha^a_k$ for some $x_j \in x$, and we denote the full set of generators of $\widehat{BSA}(Y(T), \mathcal{Z})$ by $\mathcal{G}(\mathcal{H}(T))$.
	
	The $\mathcal{A}(\mathcal{Z})$ $\mathcal{A}_\infty$-module structure of $\widehat{BSA}(Y(T), \mathcal{Z})$ records information about how generators of $\widehat{BSA}(Y(T), \mathcal{Z})$ interact with the arc diagram $\mathcal{Z}$. The most basic way this happens is via a distinguished set of $2^{n + 1}$ elements of $\mathcal{A}(\mathcal{Z})$, called {\em indecomposable idempotents}. These idempotents are canonically in bijection with subsets of the edge set of $\mathcal{Z}$ \cite[Section 2.2]{za11} \cite[Exercise 1.10]{lot14}, and we label such an idempotent by $\iota_{\bf j, k}$, where ${\bf j} = \{j_1, \dots, j_r\}$ and ${\bf k} = \{k_1, \dots, k_s\}$ are multi-index subsets of $[n]$. Recalling that edges of $\mathcal{Z}$ are identified with arcs of $\alpha^a$, we associate the idempotent $\iota_{\bf j, k}$ to the arc set
	\begin{equation}
		\label{eq:edge_set}
		\{\alpha^a_{j_1}. \alpha^a_{j_2}, \dots, \alpha^a_{j_r}, \alpha^a_{n + k_1}, \alpha^a_{n + k_2}, \dots, \alpha^a_{n + k_s}\}.
	\end{equation}
	In other words, the multi-index ${\bf j}$ of $\iota_{\bf j, k}$ determines the edges corresponding to $\iota_{\bf j, k}$ in the bottom copy of $\mathcal{Z}_n$, and the multi-index ${\bf k}$ determines the edges in the top copy of $\mathcal{Z}_n$.
	
	As a vector space, $\widehat{BSA}(Y(T), \mathcal{Z})$ splits over the action of the indecomposable idempotents:
	$$
		\widehat{BSA}(Y(T), \mathcal{Z}) = \bigoplus_{{\bf j}, {\bf k} \subset [n]} \widehat{BSA}(Y(T), \mathcal{Z}) \cdot \iota_{\bf j,k}.
	$$
	For fixed multi-indices ${\bf j, k} \subset [n]$ as above, a generator $x \in \mathcal{G}(\mathcal{H}(T))$ belongs to the $\iota_{\bf j,k}$ summand if and only if its set of occupied arcs is precisely (\ref{eq:edge_set}). We denote the set of generators in the $\iota_{\bf j,k}$ summand by $\mathcal{G}(\mathcal{H}(T)) \cdot \iota_{\bf j, k}$.
	
	\begin{rmk}
		Since our arc diagram $\mathcal{Z}$ is a disjoint union of two copies of the arc diagram $\mathcal{Z}_n$, we have
		$$
		\mathcal{A}(\mathcal{Z}) \cong \mathcal{A}(\mathcal{Z}_n) \otimes \mathcal{A}(\mathcal{Z}_n).
		$$
		Zarev describes the algebra $\mathcal{A}(\mathcal{Z}_n)$, which he calls $\mathcal{A}(\mathcal{W}_n)$, in detail in \cite[Section 9.1]{za11}.
	\end{rmk}
	
	In addition to being grouped by idempotents, generators of $\widehat{BSA}(Y(T), \mathcal{Z})$ are equipped with gradings. First, each generator $x \in \mathcal{G}(\mathcal{H}(T))$ has a $\text{Spin}^c$ grading $\mathfrak{s}(x)$, analogous to the Alexander grading in knot Floer homology. Like the Alexander grading, the $\text{Spin}^c$ grading can be understood relatively, using difference classes \cite[Proof of Proposition 14]{za11}. Let $x = \{x_1, \dots, x_n\}$ and $y = \{y_1, \dots, y_n\}$ be two generators of $\mathcal{H}(T)$. Then each curve $\beta_j \in \beta$ contains exactly one point $x_r \in x$ and one point $y_s \in y$, and we let $b_j \subset \beta_j$ be an oriented arc from $x_r$ to $y_s$. Similarly, if an arc $\alpha^a_j \in \alpha^a$ is occupied by both $x$ and $y$, we let $a_j \subset \alpha^a_j$ be an arc from the relevant point of $y$ to the relevant point of $x$. If $\alpha^a_j$ is only occupied by one generator, say by $x_r \in x$, then we let $a_j \subset \alpha^a_j$ be an arc from $\partial \Sigma$ to $x_r$. Finally, if $\alpha^a_j$ is not occuppied by $x$ or $y$, then we let $a_j$ be empty. We call the union of these arcs,
	$$
	\gamma_{x,y} = a_1 \cup \cdots \cup a_{2n} \cup b_1 \cup \dots \cup b_n,
	$$
	a {\em difference cycle} for $x$ and $y$, and we observe that $\gamma_{x,y}$ is a cycle in $H_1(Y(T), \phi(\mathcal{Z}))$. With some abuse of notation, the $\text{Spin}^c$ grading can be defined as a relative $H_1(Y(T), \phi(\mathcal{Z}))$-grading by
	\begin{equation}
		\label{eq:rel_spin}
		\mathfrak{s}(x) - \mathfrak{s}(y) = [\gamma_{x,y}] \in H_1(Y(T), \phi(\mathcal{Z})).
	\end{equation}
	We let $\mathcal{G}(\mathcal{H}(T), \mathfrak{s})$ be the set of generators with $\text{Spin}^c$ grading $\mathfrak{s}$, and $\mathcal{G}(\mathcal{H}(T), \mathfrak{s}) \cdot \iota_{\bf j,k}$ the generators in both $\mathcal{G}(\mathcal{H}(T), \mathfrak{s})$ and the idempotent summand $\mathcal{G}(\mathcal{H}(T)) \cdot \iota_{\bf j,k}$.
	
	To compare $\widehat{BSA}(Y(T), \mathcal{Z})$ with the quantum $\mathfrak{gl}(1 \vert 1)$ braid representation, we simplify the $\text{Spin}^c$ grading to a relative $\Z$-grading. To this end, we note that $H_1(Y(T), \phi(\mathcal{Z}))$ is generated by the meridians of components of $T$, plus a straight line segment in $\partial I^3$ which connects the two copies of $\mathcal{Z}_n$.
	
	\begin{defn}
		\label{def:bs_grading}
		Let
		$$
		h : H_1(Y(T), \phi(\mathcal{Z})) \to \Z
		$$
		be the homomorphism which sends any positively oriented meridian of $T$ to $1$, and which sends the line segment described above to $0$. Define a relative $\Z$-grading $\widetilde{A}$ on $\mathcal{G}(\mathcal{H}(T))$, as a quotient of the relative grading (\ref{eq:rel_spin}), by
		$$
		\widetilde{A}(x) - \widetilde{A}(y) = h([\gamma_{x,y}]).
		$$
	\end{defn}
	
	In addition to the $\text{Spin}^c$ grading, the generators $\mathcal{G}(\mathcal{H}(T), \mathfrak{s})$ in any fixed $\text{Spin}^c$ structure $\mathfrak{s}$ are equipped with a relative grading \underline{gr}. This grading can be quite complicated in general, but in our case the fact that $H_1(\mathcal{Z}) = 0$ simplifies things greatly. Since $H_1(\mathcal{Z}) = 0$, \underline{gr} takes values in $\frac{1}{2}\Z$ by definition \cite[Section 2.4]{za11}. In fact, more is true:
	
	\begin{lemma}
		\label{lem:bs_maslov}
		Fix generators $x, y \in \mathcal{G}(\mathcal{H}(T), \mathfrak{s}) \cdot \iota_{\bf j,k}$. Then $x$ and $y$ occupy the same arcs in $\alpha^a$, which we enumerate as $\alpha^a_{r_1}, \dots, \alpha^a_{r_n}$. We also write $x = \{x_1, \dots, x_n\}$ and $y = \{y_1, \dots, y_n\}$, so that $x_j \in \alpha_{r_s} \cap \beta_{\sigma_x(s)}$ and $y_j \in \alpha_{r_s} \cap \beta_{\sigma_y(s)}$ for all $s \in [n]$ and some fixed permutations $\sigma_x, \sigma_y \in S_n$.
		
		Orient each $\alpha^a$ arc upward, and each $\beta$ curve counter-clockwise. Then
		$$
		\text{\underline{gr}}(x) - \text{\underline{gr}}(y) \in \Z.
		$$
		Furthermore, 
		$$
		\text{\underline{gr}}(x) - \text{\underline{gr}}(y) \equiv 0 \mod 2
		$$
		if and only if
		$$
		(-1)^{\text{sgn}(\sigma_x)} \prod_{r = 1}^n \text{sgn}(x_r) = (-1)^{\text{sgn}(\sigma_y)} \prod_{r = 1}^n \text{sgn}(y_r).
		$$
	\end{lemma}
	
	Lemma \ref{lem:bs_maslov} says that, when restricted to $\mathcal{G}(\mathcal{H}(T), \mathfrak{s}) \cdot \iota_{\bf j,k}$, the relative grading \underline{gr} behaves like the Maslov grading in Lemma \ref{lem:maslov}.
	
	\begin{proof}
		Let $x$ and $y$ be as in the statement of the lemma. Since $x$ and $y$ have the same $\text{Spin}^c$ grading, there exists a domain $\mathcal{D}$ in $\mathcal{H}(T)$ connecting $x$ to $y$ \cite[Proposition 4.14]{za11}. The boundary of this domain is a difference cycle $\gamma_{x,y}$ for $x$ and $y$, and since $x$ and $y$ occupy the same $\alpha^a$ arcs, $\gamma_{x,y}$ is a collection of immersed closed curves in $\mathcal{H}(T)$. In other words, $\mathcal{D}$ avoids $\partial \Sigma$. In this case, the defining combinatorial formula for $\text{\underline{gr}}(x) - \text{\underline{gr}}(y)$ matches the combinatorial formula for the Maslov index, and the claim follows from the same argument used to prove Lemma \ref{lem:maslov}. Although Krutowski's result \cite[Theorem 3.1]{kru24} cited there is proven only in the non-bordered framework, a close reading shows Krutowski's argument works in our case as well.
	\end{proof}
	
		Lemma \ref{lem:bs_maslov} motivates the following definition.
	
	\begin{defn}
		Fix multi-indices ${\bf j,k} \subset [n]$, and let $x = \{x_1, \dots, x_r\} \in \mathcal{G}(\mathcal{H}(T)) \cdot \iota_{\bf j,k}$ be an arbitrary generator. Using the notation of Lemma \ref{lem:bs_maslov}, define an absolute $\Z/2$-grading $\underline{M}$ on $\mathcal{G}(\mathcal{H}(T)) \cdot \iota_{\bf j,k}$ by
		$$
		\underline{M}(x) = \begin{cases}
			0 &  \text{if } (-1)^{\text{sgn}(\sigma_x)} \prod_{r = 1}^n \text{sgn}(x_r) = 1 \\
			1  & \text{if } (-1)^{\text{sgn}(\sigma_x)} \prod_{r = 1}^n \text{sgn}(x_r) = -1 
		\end{cases}.
		$$
	\end{defn}
	
	With the gradings $A$ and $\underline{M}$ in place, we now propose a ``decategorification'' of each summand $\widehat{BSA}(Y(T), \mathcal{Z}) \cdot \iota_{\bf j, k}$.
	
	\begin{defn}
		\label{def:t_alex}
		Let $\mathcal{H}$ be {\em any} bordered sutured Heegaard diagram for the bordered sutured manifold $Y(T)$, and fix an absolute grading $A$ lifting the relative grading $\widetilde{A}$ on generators $\mathcal{G}(\mathcal{H})$ of $\mathcal{H}$. Given multi-indices ${\bf j, k} \subset [n]$, define a polynomial $\Delta_{T, {\bf j}, {\bf k}} \in \Z[t^{\pm 1}]$ by
		$$
		\Delta_{T, {\bf j}, {\bf k}}(t) = \sum_{x \in \mathcal{G}(\mathcal{H}) \cdot \iota_{\bf j, k}} (-1)^{\underline{\text{M}}(x)} t^{A(x)}.
		$$
	\end{defn}
	
	Although we've only defined the gradings $\underline{M}$ and $\widetilde{A}$ for our specific Heegaard diagram $\mathcal{H}(T)$, the same definitions work for any Heegaard diagram for $Y(T)$, and we make this extension implicitly in Definition \ref{def:t_alex}. It also easy to generalize the definition of $\Delta_{T, {\bf j}, {\bf k}}$ to tangles which are not braids, but we leave this to the reader. The polynomial $\Delta_{T, {\bf j},{\bf k}}$ is our decategorified invariant.
	
	\begin{prop}
		For fixed ${\bf j, k} \subset [n]$, the polynomial $\Delta_{T, {\bf j},{\bf k}}(t)$ of Definition \ref{def:t_alex} is an invariant of the bordered sutured manifold $Y(T)$, up to multiplication by a unit of $\Z[t, t^{-1}]$.
	\end{prop}
	
	\begin{proof}
		Since the $\mathcal{A}_\infty$-module $\widehat{BSA}(Y(T), \mathcal{Z})$ is an invariant of $Y(T)$ up to chain homotopy equivalences which respect the module structure, it makes sense to discuss the summand $\widehat{BSA}(Y(T), \mathcal{Z}) \cdot \iota_{\bf j, k}$ for ${\bf j, k} \subset [n]$. Considering $\widehat{BSA}(Y(T), \mathcal{Z}) \cdot \iota_{\bf j, k}$ as a chain complex, this complex decomposes as a direct sum along $\text{Spin}^c$-structures, hence along $A$ gradings. Additionally, within each $\text{Spin}^c$ summand the differential decreases the grading \underline{gr} by one, hence also the grading $\underline{M}$ by Lemma \ref{lem:bs_maslov}. Therefore, to show $\Delta_{T, {\bf j}, {\bf k}}$ is an invariant of $\widehat{BSA}(Y(T), \mathcal{Z}) \cdot \iota_{\bf j, k}$, it suffices to show the grading $\underline{M}$ is invariant when comparing elements of $\widehat{BSA}(Y(T), \mathcal{Z}) \cdot \iota_{\bf j, k}$ across different $\text{Spin}^c$ gradings. Since $\Delta_{T, {\bf j},{\bf k}}$ is only determined up to a unit of $\Z[t, t^{-1}]$, we need only show $\underline{M}$ is invariant as a relative grading.
		
		Equivalently, let $\mathcal{H}'$ be any bordered sutured Heegaard diagram for $Y(T)$, and \break $\widehat{BSA}(Y(T), \mathcal{Z})'$ the resulting $\mathcal{A}_\infty$-module. Let $\widehat{BSA}(Y(T), \mathcal{Z})$ denote the module determined by $\mathcal{H}(T)$. We must show that there exists a chain homotopy equivalence
		$$
		\widehat{BSA}(Y(T), \mathcal{Z}) \to \widehat{BSA}(Y(T), \mathcal{Z})'
		$$
		such that, for all $x, y \in \mathcal{G}(\mathcal{H}(T)) \cdot \iota_{\bf j,k}$, if $x'$ and $y'$ denote their respective images under this equivalence, then
		$$
		\underline{M}(x) - \underline{M}(y) = \underline{M}(x') - \underline{M}(y') \in \Z/2. 
		$$
		
		Since this is a decategorified statement, it is not too difficult to prove. Any two bordered sutured Heegaard diagrams for $Y(T)$ are related by a sequence of handleslides, diagram isotopies, stabilizations and destabilizations \cite[Proposition 4.5]{za11}. In each case, Zarev gives a graded chain homotopy equivalence of the $\mathcal{A}_\infty$-modules before and after each move \cite[Theorem 7.8]{za11} (see also \cite{lot18, lip06, os04c}). One can check directly that each of these chain homotopy equivalences respects $\underline{M}$ as a relative grading. (Stabilizations and destabilizations can result in a global shift of $\underline{M}$ as an absolute grading.) For the sake of space, we omit this checking here.
	\end{proof}
	
	\subsection{Bordered sutured Floer homology and quantum $\mathfrak{gl}(1 \vert 1)$}
	\label{sec:bstqq}
	
	In this section we relate our decategorified invariant $\Delta_{T, {\bf j, k}}$ of $\widehat{BSA}(Y(T), \mathcal{Z})$ to the $U_q(\mathfrak{gl}(1 \vert 1))$ braid representation, proving Theorem \ref{thm:main_uq2}. The arguments are similar to those in Section \ref{sec:last}.
	
	First, we use Lemma \ref{lem:int_bij} to give an absolute grading
	$$
	A : \mathcal{G}(\mathcal{H}(T)) \to \Z,
	$$ 
	lifting the relative grading $\widetilde{A}$ of Definition \ref{def:bs_grading}. Let $A_1^\text{loc}$ be the local grading preceeding Definition \ref{def:grading}; then, via the bijection of Lemma \ref{lem:int_bij}, $A_1^\text{loc}$ induces a map
	$$
	A_1^\text{loc} : \{x \in \alpha^a_j \cap \beta_k \mid \alpha^a_j \in \alpha^a, \beta_k \in \beta \} \to \Z.
	$$
	\begin{defn}
		\label{def:bs_abs}
		Define an absolute $\Z$-grading $A$ on the generators of $\mathcal{H}(T)$ by
		$$
		A(x) = \sum_{x_j \in x} A^\text{loc}_1(x_j).
		$$
	\end{defn}
	
	The next lemma shows $A$ is a lift of $\widetilde{A}$.
	
	\begin{lemma}
		\label{lem:gr_again}
		Let $x$ and $y$ be two generators of $\mathcal{H}(T)$, and let $\gamma_{x,y}$ be a difference cycle. Then
		$$
		A(x) - A(y) = h([\gamma_{x,y}]),
		$$
		where $h$ is the homomorphism of Definition \ref{def:bs_grading}.
	\end{lemma}
	
	\begin{proof}
		We follow the proof of Lemma \ref{lem:grading}. Define a set of properly embedded, oriented arcs $\ell_1, \dots, \ell_{2n} \subset \Sigma$ so that $\ell_j$ is parallel and to the right of the arc $\alpha^a_j$, and oriented vertically upward in $\Sigma$. From the definition of $h$, we have
		$$
		h([\gamma_{x,y}]) = \sum_{j = 1}^{2n} \gamma_{x,y} \cdot \ell_j.
		$$
		In fact, as in the proof of Lemma \ref{lem:grading}, we can choose the difference class $\gamma_{x,y}$ so that $\gamma_{x,y} \cap \ell_j = \varnothing$ for all $j \leq n$, so that we have
		$$
		h([\gamma_{x,y}]) = \sum_{j = n}^{2n} \gamma_{x,y} \cdot \ell_j.
		$$
		It then follows from the definition of $A$ and $A^\text{loc}_1$, and from Lemma \ref{lem:s_intersect}, that
		$$
		A(x) - A(y) = \sum_{j = n}^{2n} \gamma_{x,y} \cdot \ell_j
		$$
		as well.
	\end{proof}
	
	Finally, we prove the module $\widehat{BSA}(Y(T), \mathcal{Z})$ recovers the $U_q(\mathfrak{gl}(1 \vert 1))$ braid representation. In fact, we prove that $\widehat{BSA}(Y(T), \mathcal{Z})$ recovers the representation $\varphi^\wedge_n$; by Proposition \ref{prop:isom}, this is an equivalent statement. In the theorem below, for a multi-index ${\bf j} \subset [n]$, we let ${\bf j}^* \subset [n]$ denote its complement $[n] \setminus {\bf j}$.
	
	\begin{thm}
		\label{thm:last_one}
		Let $\rho \in B_n$ be a braid, and $T \subset I^3$ the tangle corresponding to $\bar{\rho}$. Then, for all multi-indices ${\bf j, k} \subset [n]$, we have
		$$
		(\varphi^\wedge_n(\rho))^{\bf j}_{\bf k} = \Delta_{T, {\bf j}^*, {\bf k}}(t)
		$$
		up to a unit of $\Z[t, t^{-1}]$.
	\end{thm}
	
	\begin{proof}
		The proof of Theorem \ref{thm:last_one} is analogous to the proof of Theorem \ref{thm:clear}. First, since there are $n$ $\beta$ curves in $\mathcal{H}(T)$, any generator $x \in \mathcal{G}(\mathcal{H}(T))$ occupies $n$ $\alpha^a$ arcs. Thus $\Delta_{T, {\bf j}^*, {\bf k}}$ is zero unless
		$$
		|{\bf j}^*| + |{\bf k}| = n,
		$$
		or equivalently $|{\bf j}| = |{\bf k}|$. This is consistent with the behavior of $\varphi^\wedge_n$, and with this in mind let $|{\bf j}| = |{\bf k}| = m$. Let
		$$
		x = \{x_1, \dots, x_n\} \in \mathcal{G}(\mathcal{H}) \cdot \iota_{{\bf j}^*,{\bf k}}
		$$
		be a generator, and index the elements of $x$ so that $x_r \in \beta_r$ for all $r$. By definition, $x$ occupies the arc $\alpha^a_{j}$ for all $j \in {\bf j}^*$. Each of these arcs intersects only one $\beta$ curve, $\beta_j$, and only in one place---these are intersections which map to anchor points under the correspondence of Lemma \ref{lem:int_bij}. It follows that, for all $r \in [n]$, $x_r$ lies on an arc $\alpha^a_j$ with $j \leq n$ if and only $r \in {\bf j}^*$. Equivalently, $x_r$ lies on an arc $\alpha^a_j$ with $j > n$ if and only if $r \in {\bf j}$.
		
		With the proof of Theorem \ref{thm:clear} in mind, the preceeding discussion shows that generators $x \in \mathcal{G}(\mathcal{H}(T)) \cdot \iota_{{\bf j}^*,{\bf k}}$ are in bijective correspondence with elements of $\bigcup_{\sigma \in S_m} \Gamma^{\bf j,k}_\sigma(\rho)$, where $\Gamma$ is defined as in (\ref{eq:Gamma}). Applying Lemmas \ref{lem:ugly}, \ref{lem:bs_maslov}, and \ref{lem:gr_again}, we have
		$$
		\Delta_{T, {\bf j}^*,{\bf k}}(t) = (-1)^{|{\bf j}^*|}\varphi^\wedge_n(\rho)^{\bf j}_{\bf k},
		$$
		where the constant in front results from accounting for the ``anchor point'' intersections on the arcs $\alpha^a_j$ with $j \leq n$.
	\end{proof}
	
	\begin{rmk}
	Since the arc diagram $\mathcal{Z}$ has two components, both copies of $\mathcal{Z}_n$, there is an equivalent formulation of $\widehat{BSA}(Y(T), \mathcal{Z})$ as a bimodule $\widehat{BSDA}(Y(T), \mathcal{Z})$. In this set-up, $\widehat{BSDA}(Y(T), \mathcal{Z})$ is a type $D$ structure over the bottom copy of $\mathcal{Z}_n$, and an $\mathcal{A}_\infty$-module over the top copy of $\mathcal{Z}_n$. The type $D$ structure, which is dual to the $\mathcal{A}_\infty$-module perspective in a certain sense, explains the appearance of ${\bf j}^*$ in Theorem \ref{thm:last_one} rather than ${\bf j}$.
	\end{rmk}
	
	\bibliography{main_bib}{}
	\bibliographystyle{amsplain}
	
\end{document}